\documentclass[preprint,3p,12pt]{elsarticle}

\usepackage{amssymb, bm}
\usepackage{amsthm, amsmath}
\usepackage{breqn}

\usepackage{graphicx,color}
\usepackage{caption}
\usepackage{subcaption}
\usepackage{multirow}
\usepackage{mathtools,xparse}

\usepackage{tikz}
\usetikzlibrary{calc}
\usetikzlibrary{arrows.meta,positioning,shapes.geometric}

\usepackage{pgfplots}
\pgfplotsset{compat=newest}

\newtheorem{theorem}[subsection]{Theorem}
\newtheorem{assumption}{Assumption}
\newtheorem{remark}{Remark}
\newtheorem{definition}{Definition}
\newtheorem{lemma}{Lemma}

\usepackage{listings}

\usepackage{booktabs}
\usepackage{siunitx}
\usepackage{multirow}
\usepackage[colorlinks=true, linkcolor=blue, citecolor=blue, urlcolor=blue]{hyperref}
 
\NewDocumentCommand{\ts}{O{c} e{^?_}}{
  \begin{smallermatrix}[#1]
  \mathstrut\IfValueT{#2}{#2} \\
  \mathstrut\IfValueT{#3}{#3} \\
  \mathstrut\IfValueT{#4}{#4}
  \end{smallermatrix}%
}

\journal{Arxiv}

\begin{document}
\begin{frontmatter}



\title{A cell centered Galerkin method for miscible displacement in heterogeneous porous media}

 
\author[inst1]{Maurice S. Fabien}

\affiliation[inst1]{organization={Center for Computational Science and Engineering, Schwarzman College of Computing, \\Massachusetts Institute of Technology},
            addressline={77 Massachusetts Avenue}, 
            city={Cambridge},
            state={MA},
            postcode={02139}, 
            country={USA}} 

\begin{abstract}
 In this paper we present a cell centered Galerkin (CCG) method applied to miscible displacement problems in heterogeneous porous media.  The CCG approach combines concepts from finite volume and discontinuous Galerkin (DG) methods to arrive at an efficient lowest-order approximation (one unknown per cell).  We demonstrate that the CCG method can be defined using classical DG weak formulations, only requires one unknown per cell, and is able to deliver comparable accuracy and improved efficiency over traditional higher-order interior penalty DG methods. In addition, we prove that the CCG method for a model Poisson problem gives rise to a inverse-positive matrix in 1D.  A plethora of computational experiments in 2D and 3D showcase the effectiveness of the CCG method for highly heterogeneous flow and transport problems in porous media. Comparisons between CCG and classical DG methods are included.
\end{abstract}




\end{frontmatter}


\section{Introduction}
	This work introduces an efficient numerical method for solving the incompressible miscible displacement problem in porous media.	The miscible displacement problem is a central model for subsurface transport with direct relevance to enhanced oil recovery, polymer and solvent miscible floods, geothermal heat transport, and groundwater remediation and tracer tests. Its physics are representative yet challenging: there is no sharp interface; composition-dependent viscosity and density couple flow and transport; and strong advection competes with anisotropic hydrodynamic dispersion in heterogeneous, anisotropic permeability fields, producing viscous and gravity fingering and mixing-controlled recovery. Mathematically, the model couples an incompressible Darcy pressure equation to an advection–diffusion–dispersion equation for concentration written in conservative (mass-balance) form.  As the underlying mathematical model enjoys a mass balance, many successful discretizations attempt to mimic this property at the discrete level. These include various finite difference, finite volume, and finite element methods \cite{helmig1997multiphase}.  There is a substantial literature on designing and analyzing numerical methods for the miscible displacement problem. Desirable properties include negligible artificial diffusion; robustness and accuracy on structured and unstructured meshes; faithful treatment of heterogeneity and anisotropy; and computational costs compatible with large-scale simulations.
  
	In this paper, we apply the so-called cell centered Galerkin method to resolve both flow and transport in the miscible displacement problem.	The cell centered Galerkin method was first introduced in \cite{ref_ccg1} as a way to devise a lowest-order (one unknown per cell) method, which is able to use preexisting DG frameworks.  The CCG method relies on gradient reconstruction operators as well as certain incomplete polynomial spaces, which replace classical complete polynomial spaces in discrete formulations \cite{ref_ccg3,ref_ccg4}.  This framework allows for convergence analysis in a typical finite element framework (including convergence to minimal regularity solutions for elliptic problems). Other attractive features of DG methods are also retained for CCG methods \cite{ref_ccg2} (such as $hp$-adaptivity, ability to use general meshes, can handle anisotropic tensor coefficients, local and global mass conservation, etc.).  Additionally, for linear second order elliptic problems, the CCG method boasts comparable accuracy to piecewise linear interior penalty discontinuous Galerkin (IPDG) schemes (second order accuracy in $L^2$-norm) for solutions with suitable regularity \cite{ref_ccg1}.
 
	We prove in 1D that the CCG method for second-order elliptic problems results in a inverse-positive disretization matrix. That is, the matrix arising from the discretization of the CCG matrix for diffusion problems is monotone, but not necessarily an $M$-matrix \cite{refmono1}. High-order DG methods (even piecewise linear approximations) cannot be expected to posses this property (including multi-numerics schemes), even if the underlying PDE obeys a maximum principle \cite{ref_dg_mono1}.  Furthermore, the CCG method has a computational cost comparable to CCFV methods, since there is only one unknown per cell.  All of these features of the CCG method make it an interesting alternative for convection-diffusion type problems where implicit solves are mandatory, and controlling the number of degrees of freedom is critical for time-to-solution performance.  
  
Although traditional two-point flux cell-centered finite volume methods (CCFV) are relatively straightforward to analyze and implement, they have several drawbacks.  For instance, there are significant restrictions on the mesh (e.g., $K$-orthogonality conditions) \cite{ref_FV_grid1}, strong anisotropy and tensor coefficients can cause numerical pollution \cite{ref_FV_ani1}, convection-dominated flows using first-order upwind schemes often results in artificial smearing that misrepresents sharp fronts or interfaces \cite{ref_FV_smear1}.

Conversely, numerous studies indicate that high-order methods can overcome many of the aforementioned drawbacks of low-order methods \cite{ref_hi_lo1}.  High-order methods are generally know to be robust, are capable of handling unstructured grids, map well to emerging massively parallel computational hardware, and boast superior accuracy for smooth solutions \cite{ref_hi_lo2,ref_hi_lo3}.  Despite their many advantages, a common concern with high-order methods is that they have significantly more unknowns than the industry-standard two-point flux finite volume scheme.  This increase in unknowns can lead to prohibitive computational costs, as simulation run time is largely dominated by solving the linear systems arising from the discretization. 

The increased computational complexity has prompted researchers to seek more cost-effective and robust high-order methods.  Examples of techniques developed to reduce the computational burden include but are not limited to: hybridization, mortaring, multi-point flux finite element methods, mimetic finite differences, and multi-numerics schemes.

In \cite{li2016numerical}, the authors consider high-order mixed finite element for flow and DG for transport.  It is demonstrated that the high-order schemes are capable of resolving sharp wave fronts, handle unstructured grids, as well as can be computationally more efficient than CCFV. Other related works include DG methods for both flow and transport \cite{li2015high}.  Similar findings are reported. However, the issue of computational cost is not addressed.

	Hybridized discontinuous Galerkin methods (HDG) are a variant of the DG method in which most element-interior degrees of freedom (DOFs) are locally eliminated (static condensation), leaving only trace unknowns on the element boundaries as the globally coupled DOFs.  Through special choices of numerical fluxes and traces, the static condensation process lends itself well to parallelism.  For large meshes and large polynomial degrees of approximation, HDG can provide substantial savings over classical DG in terms of DOFs \cite{ref_dof1,ref_dof2}.  A series of papers has examined the use of hybridized DG for flow and transport \cite{fabien2020high,fabien2025high,fabien2018hybridizable}, showing that hybridization can significantly reduce the computational footprint of HDG schemes. Other features of the HDG method such as superconvergent post-processing of state variables and continuity of normal components of numerical fluxes provide additional accuracy for approximation.

	Multi-numerics schemes offer a practical way to curb the cost of high-order methods. Examples include (i) mortar-space upscaling, which partitions the domain into subdomains with possibly different physics/discretizations; well-posedness hinges on judicious interface matching conditions \cite{peszynska2002mortar}; (ii) method-of-characteristics couplings with finite elements/finite differences for flow and transport, which fold convection into a material derivative so the solution advances along (typically backward) characteristics—reducing spurious oscillations and stringent CFL limits—while diffusion/reaction are handled by standard spatial discretizations; implementations can be delicate \cite{douglas1982numerical}; and (iii) hybrid FV–DG approaches that deploy high order where needed and low order elsewhere, with challenges in selecting interfaces, treating coupling terms, reduced accuracy near interfaces, and accommodating constraints inherited from the FV side (e.g., DG–CCFV on $K$-orthogonal meshes) \cite{riviere2014convergence,doyle2020multinumerics}.

	We also mention that methods related to continuous finite elements have also been utilized in the context of flow and transport applications in porous media \cite{ref_cg1,ref_cg2,ref_cg3,ref_cg4}.  Here the primary advantage is that most finite element practitioners are familiar with continuous Galerkin schemes, and these methods have fewer unknowns than their DG counterparts.  However, special attention is required for continuous finite element approximations for flow and transport applications, because local mass conservation is typically violated \cite{ref_cg_conservation1,ref_cg_conservation2,ref_cg_conservation3,ref_cg_conservation4}.  Moreover, convection-dominated transport resolved with continuous approximations require intricate stabilization and post-processing; since most slope limiting procedures do not enforce continuity across element interfaces.

	In the literature there of course exist many numerical methods that rely on one unknown per cell (or one unknown per vertex), however, they often are not related to variational problems or require completely different methodologies to analyze \cite{amaziane2009posteriori,cances2014posteriori}.  On the other hand, the CCG method is a finite element method, and is a relatively straightforward modification of well-studied DG methods.  Keeping all of the numerical methods under a single framework can assist with implementation and analysis. Many scientific computing software packages specialize on one family of numerical method (e.g., finite elements but not spectral methods), which makes coding multi-numerics somewhat challenging.
 

	The paper is organized as follows. Section~\ref{sec:prelim} collects the necessary preliminaries and notation and introduces the CCG Galerkin function spaces. Section~\ref{sec:model} presents the governing PDE system, and Sections~\ref{sec:ccg1} and \ref{sec:ccg2} describe its CCG discretization. Section~\ref{sec:splitting} details the fully discrete scheme. Section~\ref{sec:numerics} reports numerical experiments, and Section~\ref{sec:conclusions} summarizes the findings. Appendix~\ref{sec:app} contains a new theoretical result: we prove the inverse of the 1D CCG matrix for a model second-order elliptic problem is nonnegative.

\section{Preliminaries}\label{sec:prelim}
In this section we introduce some notations and terminology required to describe the CCG method.  Let $\Omega \subset \mathbb{R}^{\tt{d}}$ be a porous medium which is an open bounded connected polygonal (polyhedral domain) for ${\tt{d}}\in \{1,2,3\}$. 

We consider a meshing of $\Omega$, denoted by $\mathcal{T}_h$.  The collection $\mathcal{T}_h$ is finite, and consists of nonempty and disjoint simplices which form a partitioning of $\Omega$. Each $E\in \mathcal{T}_h$ is referred to as an element of the mesh, and we define $h=\max_{E\in \mathcal{T}_h} h_E$, where $h_E$ is the diameter of the element $E$. Denote the cell center of $E\in \mathcal{T}_h$ as $\bar{x}_E$.  This mesh is assumed to be shape regular and consist of simplices, but more general star shaped elements are viable for the CCG method \cite{ref_shape_regular}.

The mesh skeleton is given by $\Gamma_h$.  Each $F\in \Gamma_h$ has positive $({\tt{d}}-1)$-dimensional measure, and either satisfies $F\subset E_1\cap E_2$ for $E_1,E_2\in \mathcal{T}_h$ (interior face) or $F\subset E \cap \partial \Omega$ for $E \in \mathcal{T}_h$ (boundary face).  The collection of all interior faces is represented by $\Gamma_h^\circ$, and all boundary faces are collected in $\Gamma_h^\partial$, so that $\Gamma_h= \Gamma_h^\circ\cup \Gamma_h^\partial$.

For every interior face $F\in \Gamma_h^\circ$, we arbitrarily select (and fix) a unit normal vector ${\bm n}_F ={\bm n}_{E_1,F} =-{\bm n}_{E_2,F}$, where $F\subset E_1\cap E_2$, where ${\bm n}_{E_i,F}$ are the outward unit normal for element $E_i$, $i=1,2$.  If $F\in \Gamma_h^\partial$, we set ${\bm n}_{F}$ to be the unit normal pointing outward of $\Omega$.

\begin{assumption} (Admissible mesh) \label{sec:prelim_assumption1}
We say that the mesh $\mathcal{T}_h$ is admissible if for each $F\in \Gamma_h^\circ$, there exists a set $\mathcal{B}_h^F\subset\mathcal{T}_h$ with cardinality $|\mathcal{B}_h^F|={\tt d}+1$ so that the cell centers $\{\bar{x}_T\}_{T\in \mathcal{B}_h^F}$ form a non-degenerate simplex $S_F$ of $\mathbb{R}^{\tt d}$.  
\end{assumption}
 
 \begin{remark}
 Admissible meshes result from certain regularity requirements such as shape- and contact-regularity \cite{ref_dg_book1}.  We note that admissible meshes include general polyhedral subdivisions, and allowing nonconforming interfaces (hanging nodes) \cite{ref_ccg1}.  Moreover, for a given facet $F$, the set $\mathcal{B}_h^F$, is not necessarily unique.
 \end{remark}
The standard DG finite element space is given by
\[
V_h^k := \{v \in L^2(\Omega): \forall E \in \mathcal{T}_h,~v_E\in \mathbb{P}_{\tt d}^k(E)  \},
\]
where $\mathbb{P}_{\tt d}^k(E)$ is the space of ${\tt d}$-variate polynomials of total degree at most $k\ge0$ on $E$.  Functions in $V_h^k$ are discontinuous along the interior faces of the mesh, so we introduce jumps and averages to assist in defining traces.  Given $v\in V_h^k$, and $F \in \Gamma_h^\circ$ (with $F \subset E_1\cap E_2$, $E_1,E_2\in \mathcal{T}_h$, and ${\bm n}_F$ oriented from $E_1$ to $E_2$), we set the jump and average of $v$, respectively, as
\[
[\![ v ]\!] \stackrel{def}{=} v_{E_1} - v_{E_2},
\quad
\{\!\{ v \}\!\} \stackrel{def}{=} \frac{v_{E_1} + v_{E_2}}{2},
\]
where it is understood that $v_{E_i}$ is the restriction of $v$ to $E_i$ for $i=1,2$.  As a matter of convention, for boundary faces $F \in \Gamma_h^\partial$ (with $F \subset E \cap \partial\Omega$, $E  \in \mathcal{T}_h$), the definition is extended as follows:
\[
[\![ v ]\!] =\{\!\{ v \}\!\}= v_{E }.
\]
For the cell average of $v_h\in V_h^k$ on cell $E\in \mathcal{T}_h$, we use the notation $\overline{v_h}_E \stackrel{def}{=} \frac{1}{|E|}\int_E v_h$.

To define the CCG spaces, a trace interpolator and gradient reconstruction operator are required.  
\begin{definition} \label{sec:prelim_def1}
For any $v_h\in V_h^0$, we define the barycentric trace interpolator as
\begin{equation}
I_F(v_h) \stackrel{def}{=} 
\sum_{E\in \mathcal{B}_h^F} \lambda_{E,F} \overline{v_h}_E, \quad \forall F\in \Gamma_h,
\label{eq_bary}
\end{equation}
where $\lambda_{E,F}$ is the barycentric coordinate for the simplex $S_F$ described in Assumption \ref{sec:prelim_assumption1}.  See Fig.~\ref{fig:barycentric_doodle} for a depiction of $I_F$.
\end{definition} 
\begin{figure}[htb!]
\centering
\begin{tikzpicture}[scale = 2.4, line join=round, every node/.style={font=\small}]

  \coordinate (A) at (0,0);
  \coordinate (B) at (2,0);
  \coordinate (C) at (0,1.5);

  \coordinate (D) at (1,-1);   

  \coordinate (E) at (-1,0.5); 

  \newcommand{\centroid}[4]{%
    \path let \p1=(#1), \p2=(#2), \p3=(#3),
              \n1={(\x1+\x2+\x3)/3},
              \n2={(\y1+\y2+\y3)/3}
         in coordinate (#4) at (\n1,\n2);}

  \centroid{A}{B}{C}{G1} 
  \centroid{A}{B}{D}{G2} 
  \centroid{A}{C}{E}{G3} 

  \draw[thick] (A)--(B)--(C)--cycle;   
  \draw[thick] (A)--(B)--(D)--cycle;   
  \draw[thick] (A)--(C)--(E)--cycle;   


  \fill (G1) circle (0.045) node[right=2.0pt] {$\bar x_{E_1}$};
  \fill (G2) circle (0.045) node[below=2.0pt] {$\bar x_{E_2}$};
  \fill (G3) circle (0.045) node[below=2.0pt] {$\bar x_{E_3}$};

  \fill[red] (0.7,0) circle (0.045) node[below left =2.0pt,red] {$I_F(v_h)$};

  \node at ($(G1)+(-0.18,0.38)$) {$E_1$};
  \node at ($(G2)+(0.38,-0.10)$)  {$E_2$};
  \node at ($(G3)+(0.15, 0.38)$) {$E_3$};
  \node[blue] at ($(G3)+(0.80, -0.30)$) {$S_F$};

  \node at ($(G2)+(0.10, 0.42)$) {$F$};  
  
  \draw[thick,dotted,red] (G1) -- (0.7,0);
  \draw[thick,dotted,red] (G2) -- (0.7,0);
  \draw[thick,dotted,red] (G3) -- (0.7,0);

	\draw[thick,dashed,blue] (G1) -- (G2);
	\draw[thick,dashed,blue] (G2) -- (G3);
	\draw[thick,dashed,blue] (G1) -- (G3);
\end{tikzpicture}
\caption[]{Visualization of the barycentric trace interpolator from Definition \protect\ref{sec:prelim_def1} on a triangular mesh in 2D.  The simplex $S_F$ (dashed blue lines) is formed using the cell centers of the triangles from $\mathcal{B}_F=\{E_1,E_2,E_3\}$.  The interpolator $I_F$ along the interface is denoted in red.}
\label{fig:barycentric_doodle}
\end{figure}
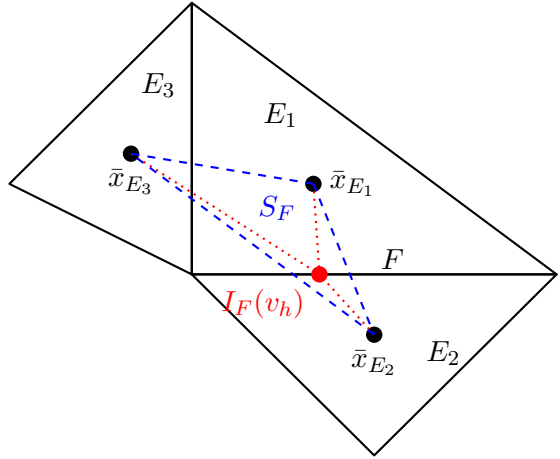
\begin{definition}\label{sec:prelim_def2}
For all $v_h\in V_h^0$ and $E\in \mathcal{T}_h$, the gradient reconstruction operator ${\bm G}_h:V_h^0\to (V_h^0)^{\tt d}$ is given by
\[
{\bm G}_h(v_h) \bigg|_E 
\stackrel{def}{=}  
\sum_{F\in \partial E} \frac{|F|}{|E|}
(I_F(v_h) - \overline{v_h}_E) {\bm n}_{E,F}
,
\]
where $|F|$ and $|E|$ are the measures of the face $F$ and element $E$, respectively.
\end{definition}
The CCG space is constructed using Definition \ref{sec:prelim_def2}.  The linear reconstruction operator $\mathcal{A}_h:V_h^0 \to V_h^1$ maps piecewise constant functions onto piecewise linear function so that for all $v_h\in V_h^0$ and $E\in \mathcal{T}_h$,
\begin{equation} 
\mathcal{A}_h(v_h)\bigg|_E 
\stackrel{def}{=}
\overline{v_h}_E + {\bm G}_h(v_h) \bigg|_E  \cdot (\vec{x} - \overline{x}_E).
\label{eq_ccg_reconstruction}
\end{equation}
Then, the CCG space is defined as
\begin{equation} 
V_h^{\textrm{ccg}} \stackrel{def}{=} \mathcal{A}_h(V_h^0)\subset V_h^1.
\label{def_ccg_space}
\end{equation}
Hence, $V_h^{\textrm{ccg}} $ is an incomplete polynomial space with its dimension equal to the cardinality of $\mathcal{T}_h$. 

\section{Model problem}\label{sec:model} 
Our model for incompressible miscible displacement is defined by the following system of partial differential equations:
\begin{align}
\vec{u} &= - \frac{{\bm K}}{\mu(c)} \nabla p,    &&\textrm{ in } \Omega \times (0,T),
\label{eq_model1}
\\
\nabla\cdot \vec{u}  &= q^I - q^P, &&\textrm{ in } \Omega \times (0,T),
\label{eq_model2}
\\
  \frac{\partial }{\partial t}(\phi c) -\nabla\cdot ( {\bm D} (\vec{u}) \nabla c) + \nabla \cdot (\vec{u} c) &= \widetilde{c} q^I - q^P c,
  &&\textrm{ in } \Omega \times (0,T).
  \label{eq_model3}
\end{align}
for unknowns $c$ (fluid concentration), $\vec{u} $ (Darcy flux), and $p$ (fluid pressure). Here,  $\phi$ is the porosity of the medium $\Omega$, ${\bm K} $ is the permeability, $q^I$ and $q^P$ are source terms (injection well and production well, respectively),
${\bm D} (\vec{u})$ is the diffusion-dispersion tensor, $\mu(c)$ is the fluid viscosity, and $\widetilde{c}$ is the fluid concentration at the injection well.  For $\vec{u}\in \mathbb{R}^{\tt d}$, diffusion-dispersion tensor is assumed to take the following form:
\begin{equation} 
{\bm D} (\vec{u}) =  
(a_t| \vec{u} | + d_m){\bm I}
+ 
\bigg(\frac{ a_l-a_t }{| \vec{u} |}\bigg)(\vec{u}\vec{u}^{\,T}),
\label{eq:diffdisp}
\end{equation}
where $|\vec{u}|= \sqrt{\vec{u}^{\,T}\vec{u}}$, and $a_l$ (longitudinal  dispersivity), $a_t$ (transverse dispersivity), $d_m$ (molecular diffusivity) are parameters.  To complete the system, we set no flow boundary conditions on $\partial\Omega$:
\begin{align*}
\vec{u}\cdot \vec{n} &=0,  ~ \textrm{on } \partial\Omega  ,
\\
({\bm D}(\vec{u}) \nabla c) \cdot \vec{n} &=0,  ~ \textrm{on }\partial\Omega .
\end{align*}
 To ensure the pressure has a unique solution, the following compatibility condition is enforced:
\[
0 = \int_\Omega (q^I - q^P).
\] 
Such a configuration results in the simulations being driven by the source terms $q^I$ and $q^P$.  The PDE is supplied with an initial condition of the form
\[
c(\vec{x},t=0) = c_0(\vec{x}),\quad
 \vec{x} \in \Omega.
\]
 
\section{CCG method for flow}\label{sec:ccg1}
With the definition of the CCG space in \eqref{def_ccg_space}, we can state the CCG method for the Darcy flow system \eqref{eq_model1}-\eqref{eq_model2}. An advantage of the CCG method is that it can utilize the same interior penalty DG framework, but replace the full DG space $V_h^k$ with $V_h^{\textrm{ccg}}$.  In particular, the CCG scheme seeks $p_h\in V_h^{\textrm{ccg}}$ so that
\begin{align}
B_p(p_h,v_h) &= \ell_p (v_h),\quad \forall v_h \in V_h^{\textrm{ccg}}, 
\label{eq_DG_flow}
\end{align}
where
\begin{dmath}
B_p(p_h,v_h) = \sum_{E\in \mathcal{T}_h} \int_E {\bm K} \nabla p_h \cdot \nabla v_h
-
\sum_{e\in \Gamma_h^\circ} \int_e \{\!\{ {\bm K} \nabla p_h  \cdot \vec{n}_e \}\!\} [\![ v_h ]\!]
+
\epsilon
\sum_{e\in \Gamma_h^\circ} \int_e \{\!\{ {\bm K} \nabla v_h  \cdot \vec{n}_e \}\!\} [\![ p_h ]\!]
+
\sum_{e\in \Gamma_h^\circ} \int_e \frac{\sigma_{e,p}}{|e|} [\![ p_h ]\!] [\![ v_h ]\!],
\end{dmath}
and
\begin{dmath}
\ell_p(v_h) = 
 \sum_{E\in \mathcal{T}_h} \int_E  (q^I-q^P)  v_h.
\end{dmath}

\section{CCG method for transport}\label{sec:ccg2}
The transport PDE \eqref{eq_model3} is discretized in space using an interior penalty DG framework for diffusion, and upwinding for the convection term.  We seek $c_h\in V_h^{\textrm{ccg}}$ so that
\begin{align}
B_c(c_h,v_h) 
+
\sum_{E\in \mathcal{T}_h} \int_E v_h \frac{\partial}{\partial t}(\phi c_h) 
 &= \ell_c (v_h),\quad \forall v_h \in V_h^{\textrm{ccg}}, 
 \label{eq_DG_transport} 
\end{align}
with
\begin{dmath}
B_c(c_h,v_h) = 
\sum_{E\in \mathcal{T}_h} \int_E
\bigg(
 {\bm D}(\vec{u}) \nabla c_h \cdot \nabla v_h
+  c q^P  v_h
-   (\vec{u}  \cdot \nabla v_h) c_h
\bigg)
+
 \sum_{e\in \Gamma_h^\circ} \int_e   c_h^\uparrow [\![ v_h ]\!]
-
\sum_{e\in \Gamma_h^\circ} \int_e 
\{\!\{ {\bm D}(\vec{u}) \cdot \vec{n}_e \nabla c_h \}\!\} [\![ v_h ]\!]
+
\epsilon
\sum_{e\in \Gamma_h^\circ} \int_e
 \{\!\{ {\bm D}(\vec{u}) \nabla v_h \cdot \vec{n}_e  \}\!\} [\![ c_h ]\!]
+
 \sum_{e\in \Gamma_h^\circ} \int_e
 \frac{\sigma_{e,c}}{|e|} [\![ c_h ]\!] [\![ v_h ]\!]
 ,
\end{dmath}
and
\begin{dmath}
\ell_c(v_h) = 
 \sum_{E\in \mathcal{T}_h} \int_E   \widetilde{c} q^I  v_h,
\end{dmath}
where $\sigma_{e,c}>0$ is a penalty parameter.  The upwind term $c_h^\uparrow$ on a face $e\in \Gamma_h^\circ$ (shared by elements $E_1$ and $E_2$), with normal vector $\vec{n}_e$ oriented from $E_1$ to $E_2$ is defined as
\[
c_h^\uparrow
=
\begin{cases}
c_h  |_{E_1}, &\mbox{ if } \vec{u}\cdot \vec{n}_e \ge 0,
\\
c_h  |_{E_2}, &\mbox{ if } \vec{u}\cdot \vec{n}_e < 0 .
\end{cases}
\] 
\section{Fully discrete approximation}\label{sec:splitting}
To arrive at a fully discrete linearized scheme, finite differences in time and a time-lagging technique are applied.  Due to the relatively mild coupling between flow and transport, given current time step data $C_h(\vec{x},t_n)$, we first solve equation \eqref{eq_DG_flow} for $p_h(\vec{x},t_n)$.  Then, we reconstruct the Darcy flux $\vec{u}(\vec{x},t_n)$ based on $p_h(\vec{x},t_n)$.  This flux is inserted into \eqref{eq_DG_transport} to obtain $C_h(\vec{x},t_{n+1})$.  See Fig.~\ref{fig:impes_doodle} for a visualization of this linearization.
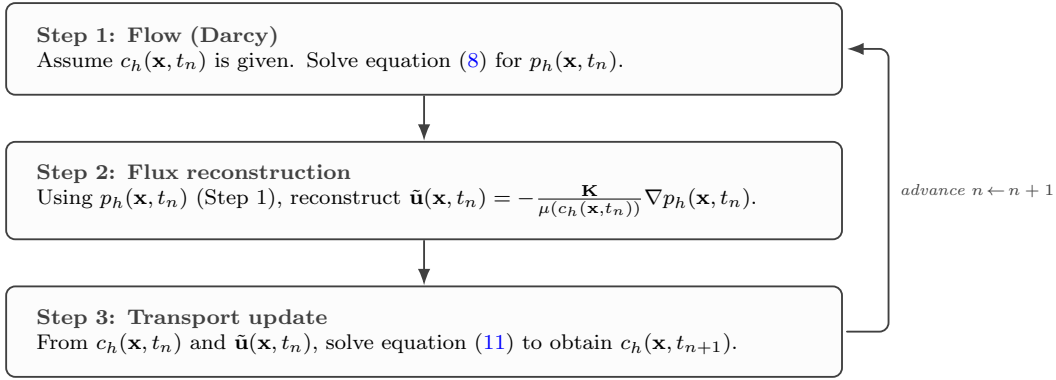
\begin{figure}[htb!]
\centering
\newlength{\bubblewidth}
\setlength{\bubblewidth}{0.62\linewidth} 
\definecolor{accent}{gray}{0.25}         

\begin{tikzpicture}[
  scale=1.00, every node/.style={transform shape},
  >=Latex,
  arrow/.style={-{Latex[length=2.2mm]}, line width=.7pt, draw=accent},
  bubble/.style={
    rectangle, rounded corners=1.2mm, draw=accent, line width=.7pt,
    fill=black!1, text width=\bubblewidth,
    inner xsep=4mm, inner ysep=3mm, align=left, font=\scriptsize
  },
  node distance=6mm
]

\node[bubble] (b1) {%
  {\bfseries\color{accent} Step 1: Flow (Darcy)}\\
  Assume \(c_h(\mathbf x,t_n)\) is given. Solve 
equation \eqref{eq_DG_flow} for \(p_h(\mathbf x,t_n)\).
};

\node[bubble, below=of b1] (b2) {%
  {\bfseries\color{accent} Step 2: Flux reconstruction}\\
  Using \(p_h(\mathbf x,t_n)\) (Step 1), reconstruct \(\tilde{\mathbf u}(\mathbf x,t_n)
= - \frac{\mathbf K}{ \mu(c_h({\mathbf x},t_n)) } \nabla p_h(\mathbf x,t_n)
  \).
};

\node[bubble, below=of b2] (b3) {%
  {\bfseries\color{accent} Step 3: Transport update}\\
  From \(c_h(\mathbf x,t_n)\) and \(\tilde{\mathbf u}(\mathbf x,t_n)\),
  solve equation \eqref{eq_DG_transport} to obtain \(c_h(\mathbf x,t_{n+1})\).
};

\draw[arrow] (b1) -- (b2);
\draw[arrow] (b2) -- (b3);

\newcommand{\loopsep}{6mm} 
\coordinate (r3) at ($(b3.east)+(\loopsep,0)$);
\coordinate (r1) at ($(b1.east)+(\loopsep,0)$);

\draw[arrow, rounded corners=1.2mm, line cap=round, line join=round,
      shorten >=1.5pt, shorten <=1.5pt]
      (b3.east) -- (r3)
      -- node[midway, right, font=\tiny\itshape, text=accent]
         {advance \(n\!\leftarrow\!n+1\)}
      (r1) -- (b1.east);
\end{tikzpicture}
\caption{Diagram of the linearization strategy. Each step occurs once per time step.}
\label{fig:impes_doodle}
\end{figure}
	We use implicit time-stepping for improved stability (and larger stable time steps).  As the CCG method has spatial accuracy which is comparable to $P_1$-DG, we focus on backward Euler (BE) and Crank–Nicolson (CN) finite differences in time.
	
 \section{Numerical experiments}\label{sec:numerics}
 For all numerical experiments below, we utilize backward Euler time stepping unless stated otherwise. For the CCG method, all integrals are safely approximated with one-point quadrature rules. As a result, the CCG mass matrix is diagonal.
\subsection{Manufactured solution} 
A smooth manufactured solution is selected for this example.  The domain is $\Omega=[0,1]^2$, and the simulation is from $t=0$ until $t=0.5$, with $\Delta t=10^{-4}$. For the remaining parameters we put $\phi=1$, $\mu(c) = ( 1 + 0.05 c )^{4} $, ${\bf K} = {\bf I}$, and
\[
p(x,y,t) = e^{-t} \sin(\pi x)\sin(\pi y),
\quad
c(x,y,t) = (t^2/2) + \cos(2\pi x) \cos(2\pi y) 
.
\]
For viscosity we take $\mu(c) = (1 + 0.0524\cdot c)^{4.75}$.  The initial condition is taken as the $L^2$-projection of exact solution evaluated at $t=0$.  We note that with this choice of parameters, the PDE system will have nonzero source terms $q^I$ and $q^P$ which can be determined from the manufactured solution.  For convenience, we use Dirichlet boundary conditions on $\partial\Omega$.  A sequence of uniform triangulations is used.  The CCG and $P_1$-DG methods are compared ($\epsilon=-1$ and $\sigma_{e,c}=\sigma_{e,p}=14$); Tables~\ref{tab:manu_p} and~\ref{tab:manu_c} have the $L^2$-norm errors and convergence rates.
 \begin{table}[htb!] 
 \centering
 \begin{tabular}{rccccc}
\toprule
$|\mathcal{T}_h|$ & CCG Error & CCG Rate & $P_1$-DG Error & $P_1$-DG Rate \\
\hline
128   & $1.3161\times10^{-2}$ & --       & $1.6006\times10^{-2}$ & --       \\
512   & $3.1102\times10^{-3}$ & $2.0812$& $3.5433\times10^{-3}$ & $2.1754$\\
2048  & $7.5379\times10^{-4}$ & $2.0448$& $8.1496\times10^{-4}$ & $2.1203$\\
8192  & $1.8609\times10^{-4}$ & $2.0182$& $1.9422\times10^{-4}$ & $2.0690$\\
32768 & $4.6290\times10^{-5}$ & $2.0072$& $4.7337\times10^{-5}$ & $2.0367$\\
\hline
\end{tabular}
\caption[]{Convergence rates and $L^2$-norm errors ($\|p_h-p\|_{L^2(\Omega)}$) for the time-dependent exact pressure solution.}
\label{tab:manu_p}
 \end{table} 
\begin{table}[htb!]
 \centering
\begin{tabular}{rccccc}
\toprule
$|\mathcal{T}_h|$ & CCG Error & CCG Rate & $P_1$-DG Error & $P_1$-DG Rate \\
\hline
128   & $1.5196\times10^{-2}$ & --       & $2.0561\times10^{-2}$ & --       \\
512   & $3.9185\times10^{-3}$ & $1.9553$& $4.4553\times10^{-3}$ & $2.2063$\\
2048  & $9.8989\times10^{-4}$ & $1.9850$& $9.5430\times10^{-4}$ & $2.2230$\\
8192  & $2.4841\times10^{-4}$ & $1.9945$& $2.1312\times10^{-4}$ & $2.1628$\\
32768 & $6.2199\times10^{-5}$ & $1.9978$& $4.9689\times10^{-5}$ & $2.1007$\\
\hline
\end{tabular}
\caption[]{Convergence rates and $L^2$-norm errors ($\|c_h-c\|_{L^2(\Omega)}$) for the time-dependent exact concentration solution.}
\label{tab:manu_c}
\end{table} 
	Both the CCG and $P_1$-DG methods provide similar accuracy and rates of convergence (approximately order two) for both the pressure and concentration.  Further comparisons between the CCG and DG approximations for nonsmooth solutions is given in subsequent sections.
 
\subsection{Homogeneous porous medium} \label{sec:numerics_homo}
 	This experiment concerns the well-known quarter five spot problem. Here, the medium is $\Omega=[0,1]^2~m^2$ with homogeneous permeability ${\bm K} = 9.44 \cdot 10^{-3} {\bm I} ~m^2$ and porosity $\phi=0.2$.  The injection and production rates satisfy
 	\[
 	\int_\Omega q^I = 	\int_\Omega q^P = 0.018~ \frac{m^2}{s},
 	\]
 	where $q^I$ piecewise constant on $[0,0.1]^2$ and zero elsewhere, and $q^P$ piecewise constant on $[0.9,1]^2$ and zero elsewhere. No flow boundary conditions are imposed.  The initial concentration at the injection well is $\widetilde{c}=1$, and initial condition is $c_0 = 0$.  We take $\sigma_{e,p} =  \sigma_{e,c}= 1$ and $\epsilon=-1$.  The final simulation time is $T=7.5$ seconds, and the time step is $\Delta t = 0.05$ seconds. We take the following parameters for the diffusion-dispersion tensor defined in \eqref{eq:diffdisp}:
\[ 
    d_m = (1.8)\cdot 10^{-7}, ~
    a_l = (1.8)\cdot 10^{-5},~
    a_t = (1.8)\cdot 10^{-6}. 
\] 	
\begin{figure}[htb!]
    \centering
    \begin{subfigure}[t]{0.23\textwidth}
        \centering
        \caption*{Mesh}
        \vspace*{0.3ex}
        \includegraphics[width=\textwidth ,trim=0cm 0cm 1.62cm 0.05cm,clip]{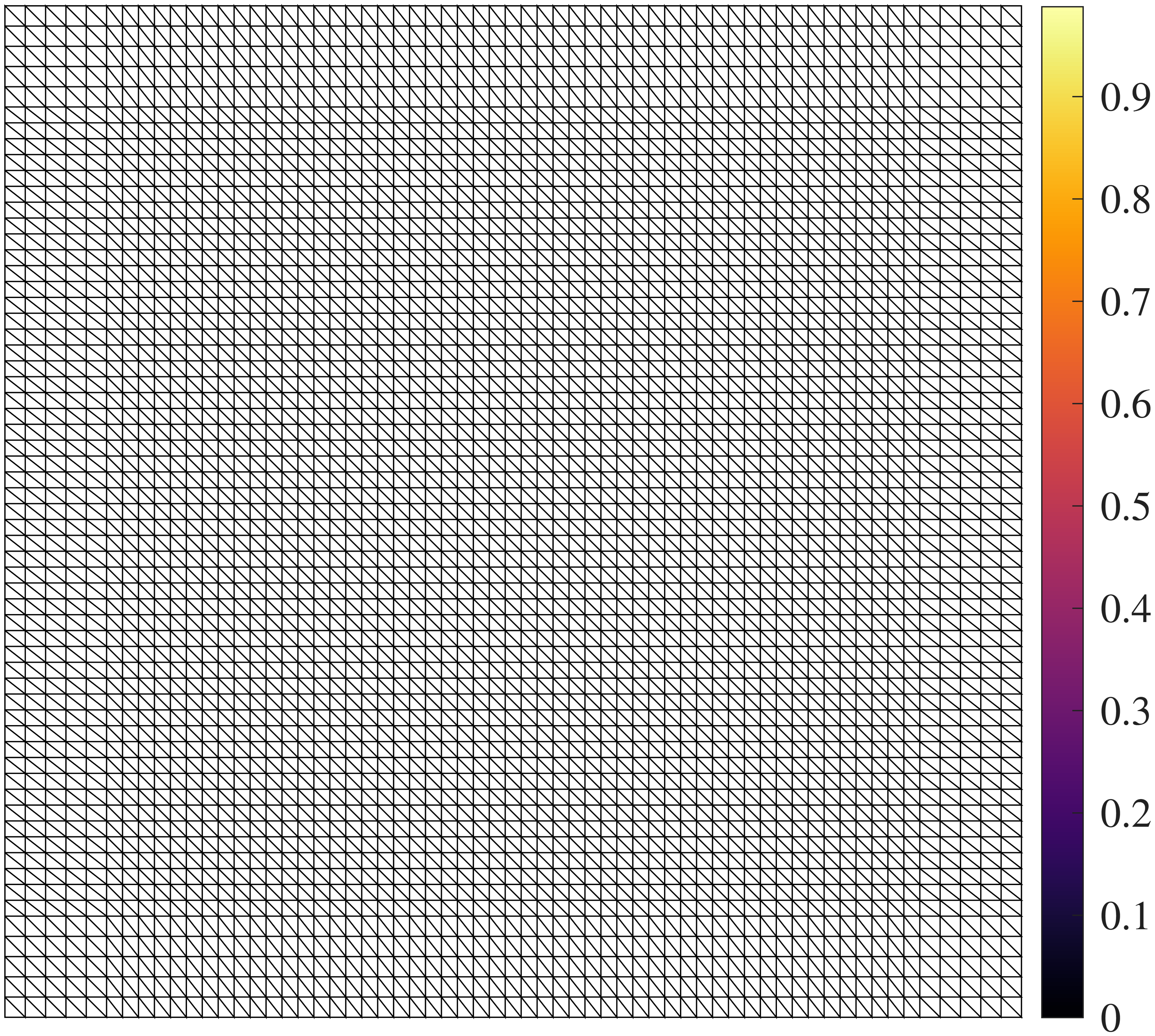} 
    \end{subfigure}
    \begin{subfigure}[t]{0.23\textwidth}
        \centering
        \caption*{$t=2.5$ (seconds)}
        \includegraphics[width=\textwidth , trim=0cm 0cm 1.62cm 0cm,clip]{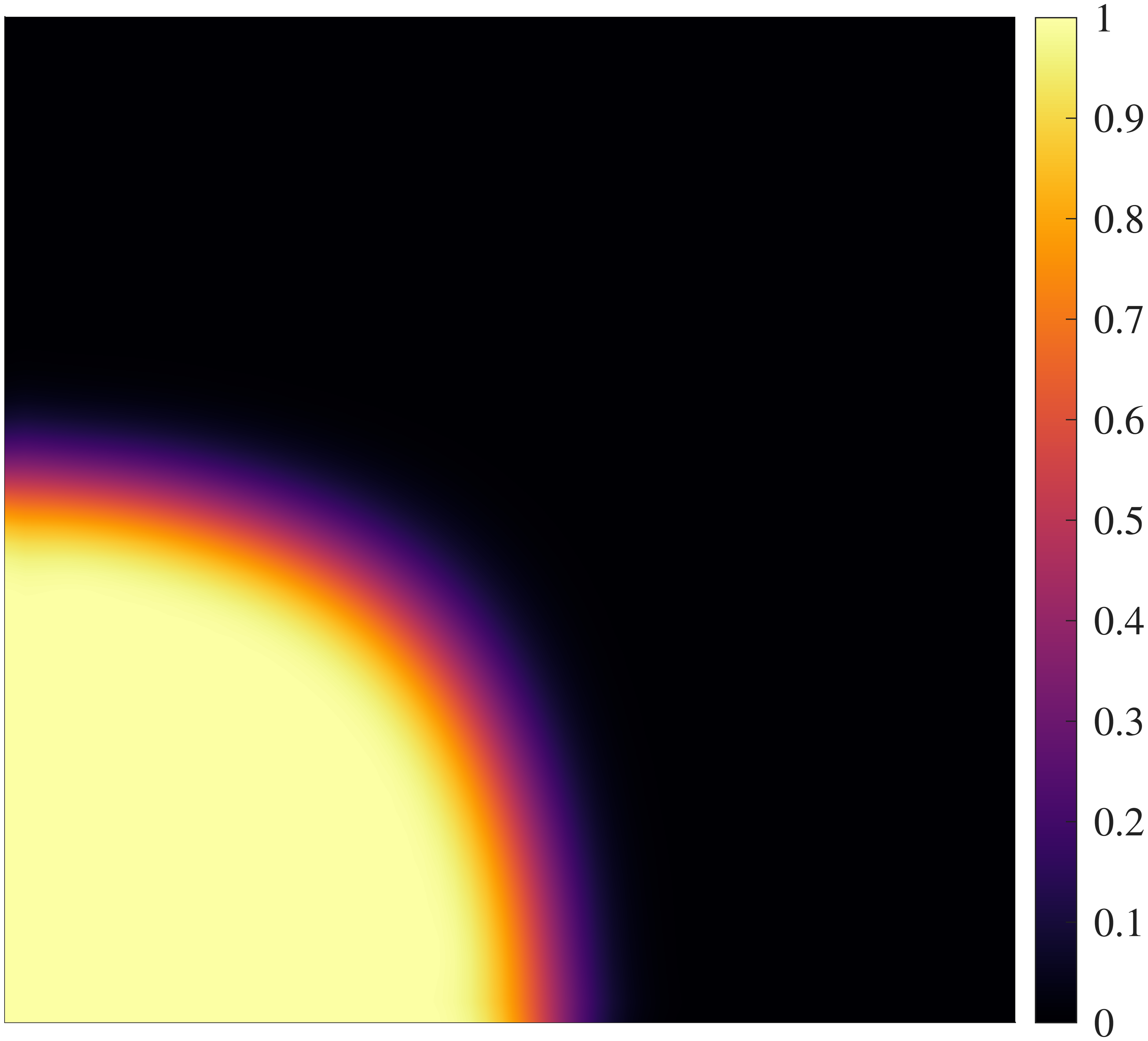} 
    \end{subfigure}
    \begin{subfigure}[t]{0.23\textwidth}
        \centering 
        \caption*{$t=5.0$ (seconds)}     
        \includegraphics[width=\textwidth , trim=0cm 0cm 1.62cm 0cm,clip]{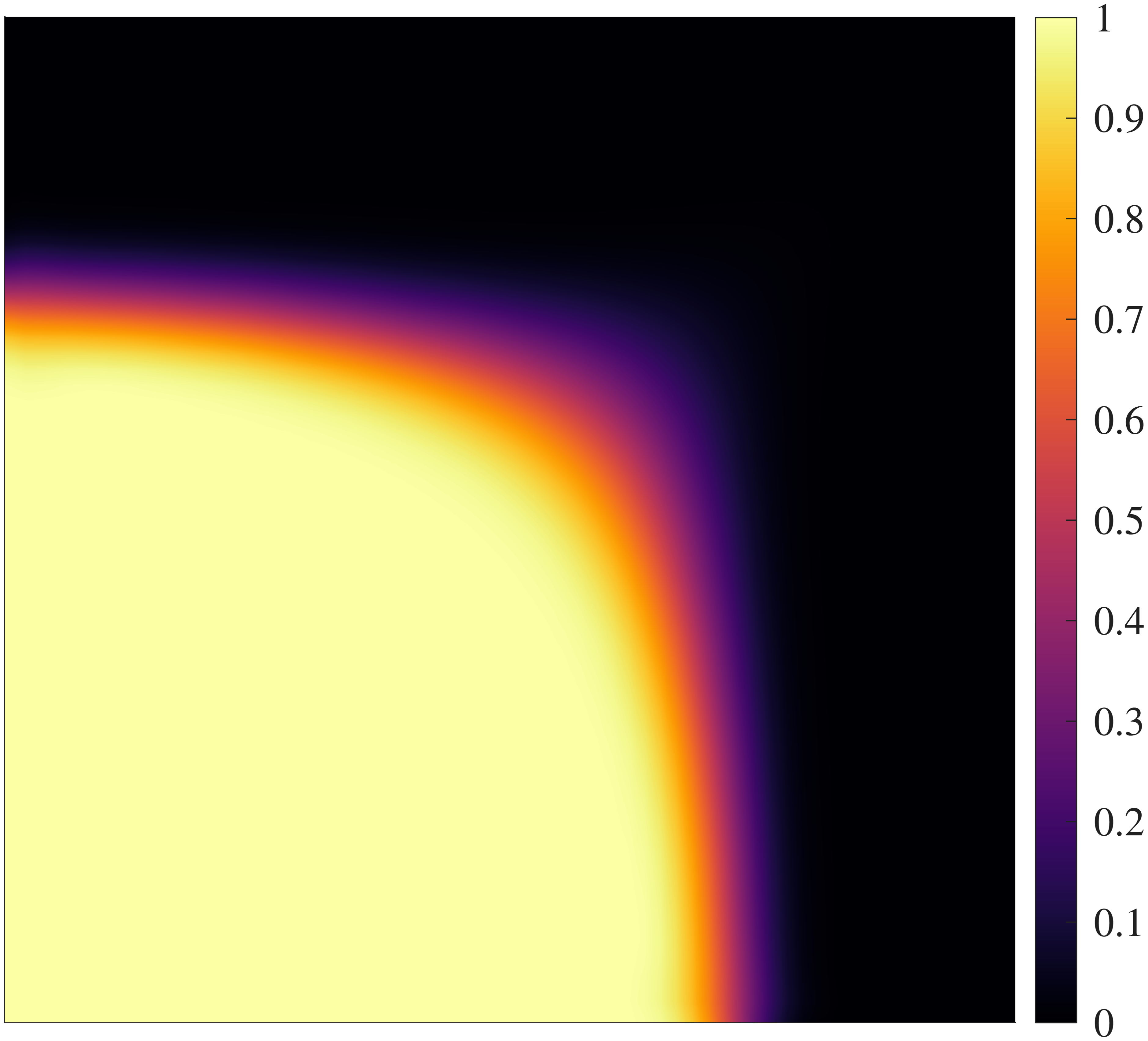} 
    \end{subfigure}
    \begin{subfigure}[t]{0.23\textwidth}
        \centering
		\caption*{$t=7.5$ (seconds)}     
        \includegraphics[width=1.112\textwidth]{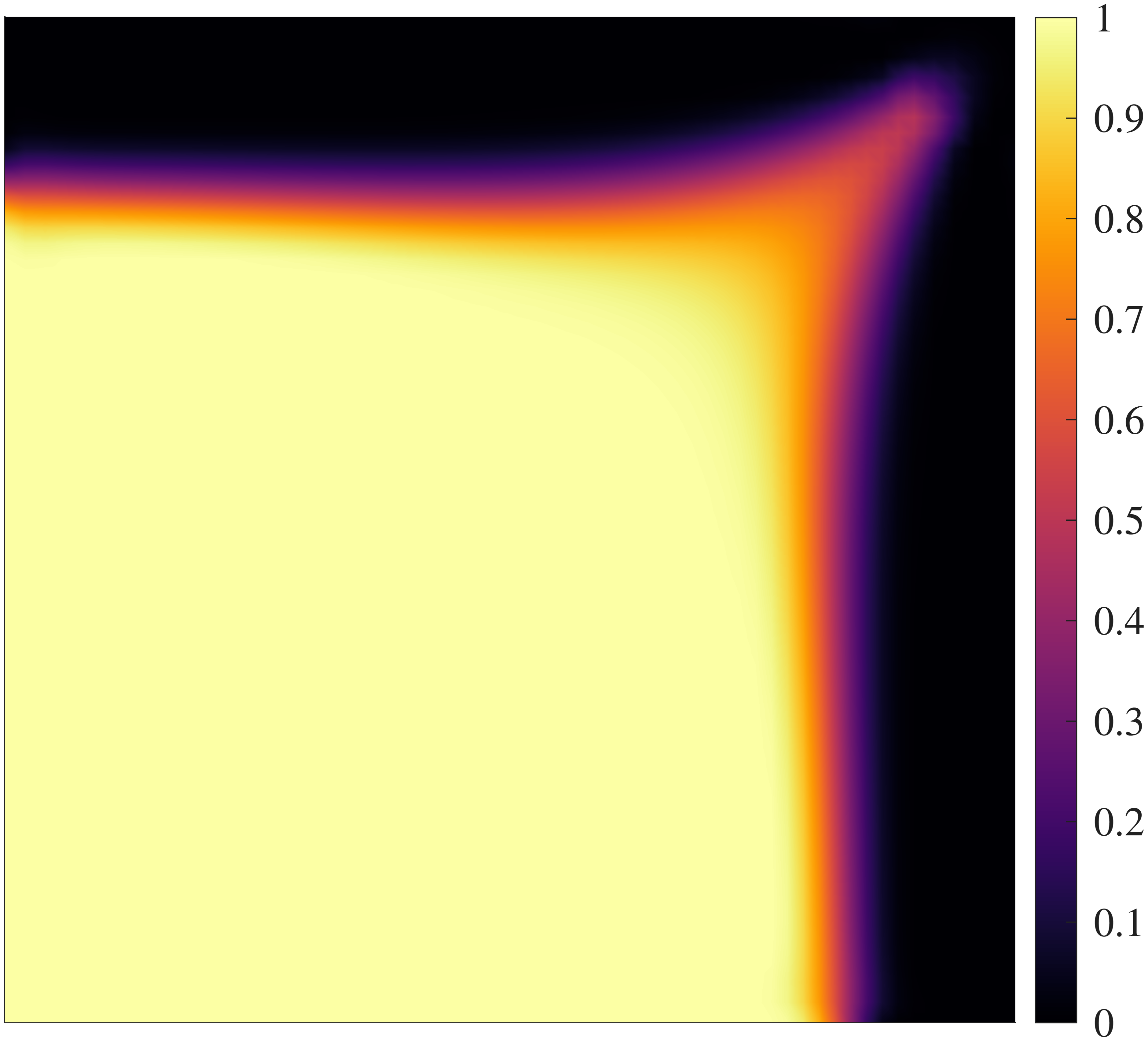} 
    \end{subfigure}
    
    \vspace{0.5em} 
    \begin{subfigure}[b]{0.23\textwidth}
        \centering 
        \includegraphics[width=\textwidth , trim=0cm 0cm 1.62cm 0cm,clip]{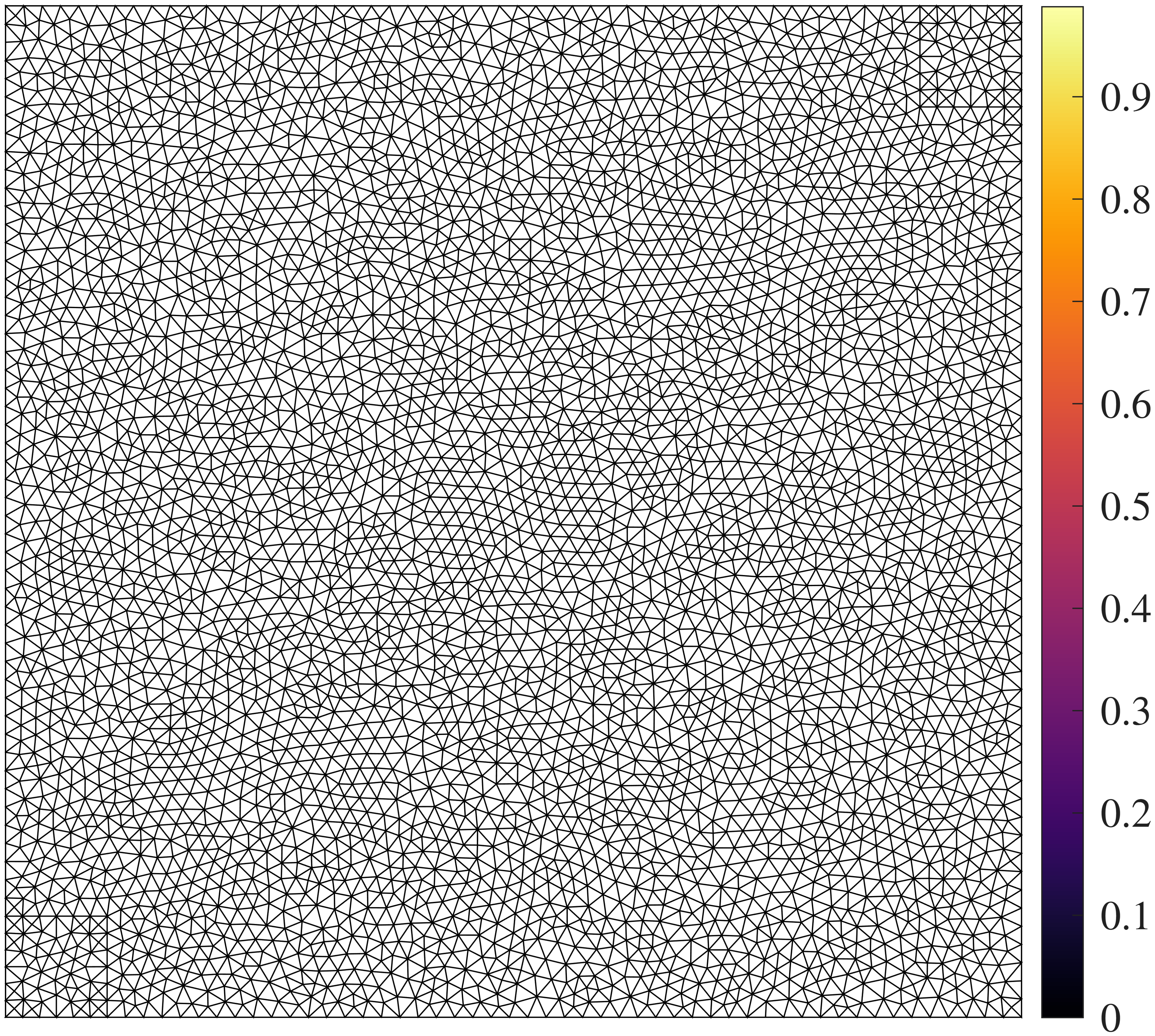}
    \end{subfigure}
    \begin{subfigure}[b]{0.23\textwidth}
        \centering
        \includegraphics[width=\textwidth , trim=0cm 0cm 1.62cm 0cm,clip]{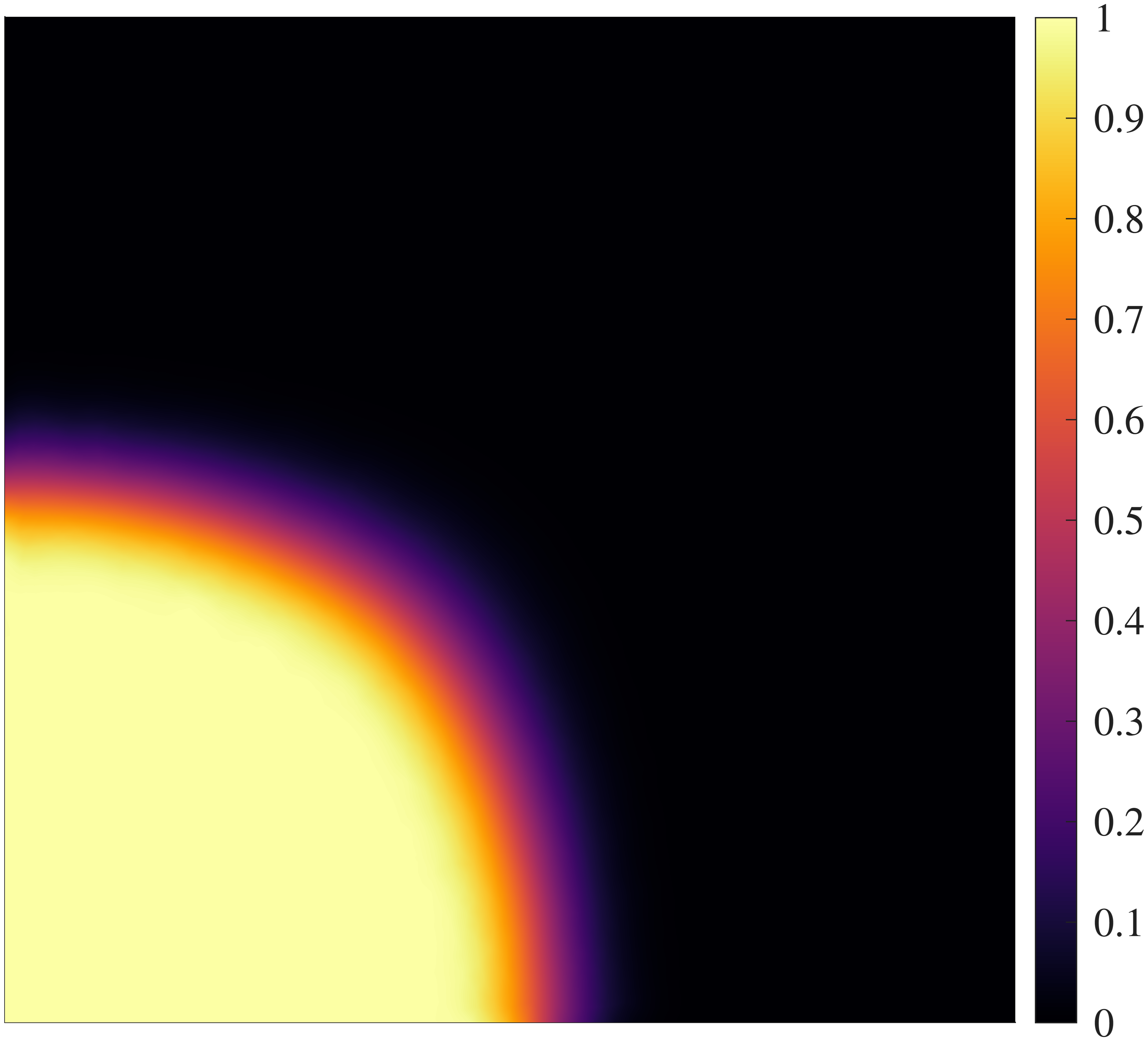}
    \end{subfigure}
    \begin{subfigure}[b]{0.23\textwidth}
        \centering
        \includegraphics[width=\textwidth , trim=0cm 0cm 1.62cm 0cm,clip]{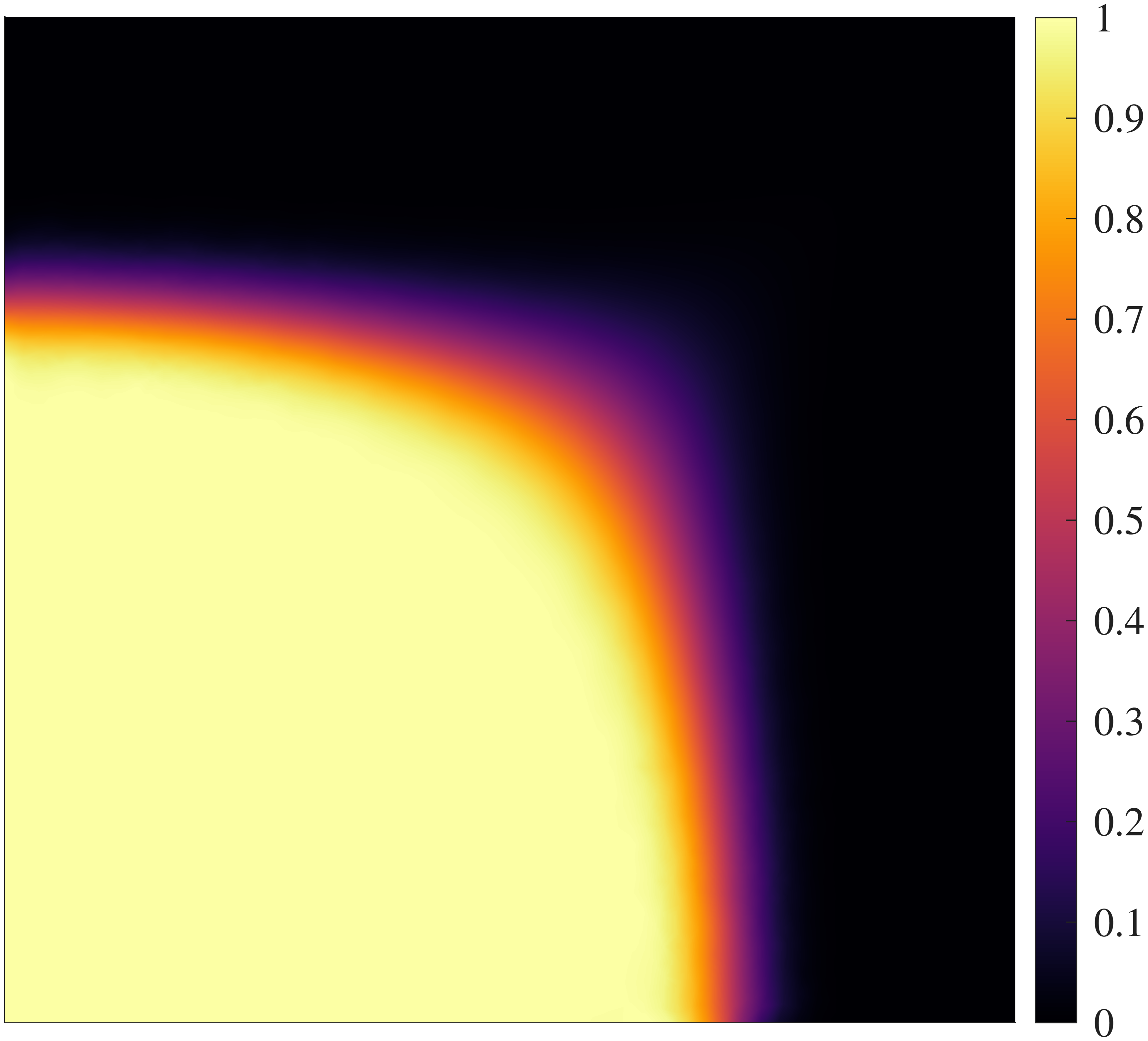}
    \end{subfigure}
    \begin{subfigure}[b]{0.23\textwidth}
        \centering
        \includegraphics[width=1.112\textwidth]{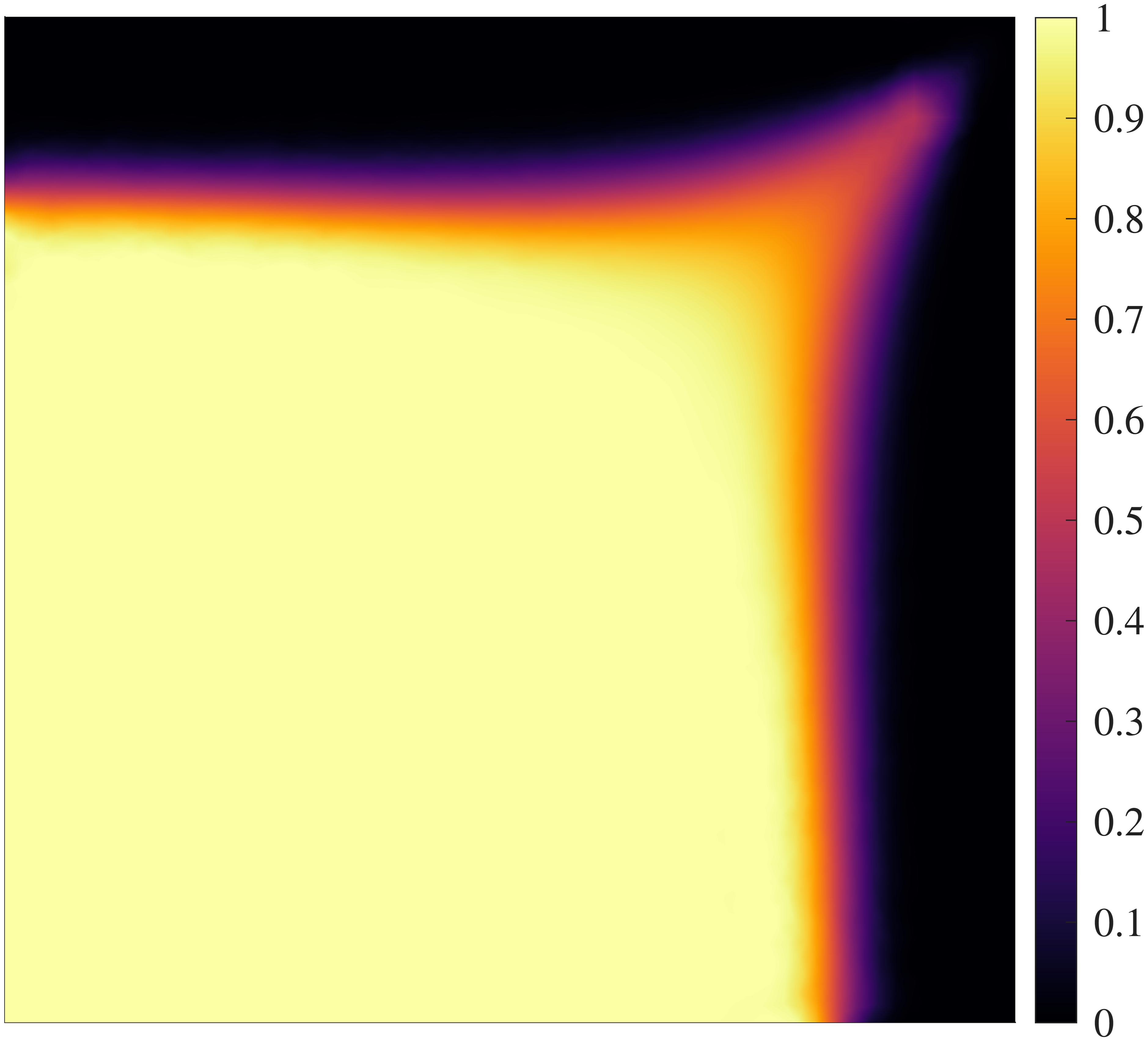}
    \end{subfigure} 
    \caption[]{ 
    Homogeneous quarter-five spot experiment from section \ref{sec:numerics_homo}.  Plot of the CCG fluid concentration: the first row corresponds to a structured mesh and the second row to an unstructured mesh.}
    \label{fig_sec:numerics_homo1}
\end{figure} 	
 	Both structured and unstructured triangular meshes are examined (both meshes have roughly 8000 elements).  The injection and production well locations are aligned with the mesh.  Fig.~\ref{fig_sec:numerics_homo1} displays the results of the experiment. We observe similar results for the structured and unstructured meshes.  As time increases, the fluid concentration moves from the injection well to the production well.
 	
 	To analyze the data from a different perspective, we consider a snapshot of the concentration along a specific cross section. The concentration profile at $t=5$ seconds along the line $y=x$ is visualized in Fig.~\ref{fig_sec:numerics_homo2}.  Backward Euler (BE) time stepping results in a more diffusive concentration front. On the other hand, the Crank–Nicolson (CN) method captures a sharper concentration front. Spatial mesh refinement also produces a sharper concentration front.
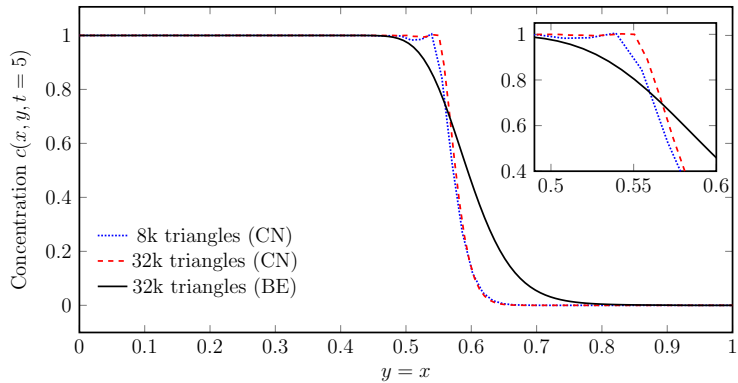
\begin{figure}[htbp]
  \centering
%
%
\begin{tikzpicture}[scale=0.65]
  \begin{axis}[
    name=main, 
    width=0.9\linewidth,
    height=0.5\linewidth,
    xlabel={$y=x$},
    ylabel={Concentration $c(x,y,t=5)$},
	legend style={at={(0.35,0.35)},anchor=north east,fill=none,draw=none},
    tick style={semithick},
    line width=1pt,
    xmin=0, xmax=1,
  ]
    \addplot+[no marks, densely dotted, draw=blue!] table [x=x,y=y, col sep=comma]{profile_struc_8k_CN.csv};
    \addlegendentry{8k triangles (CN)};

    \addplot+[no marks, dashed, draw=red!] table [x=x,y=y, col sep=comma]{profile_struc_32k_CN.csv};
    \addlegendentry{32k triangles (CN)};

    \addplot+[no marks, draw=black!] table [x=x,y=y, col sep=comma]{profile_struc_32k_BE.csv};
    \addlegendentry{32k triangles (BE)};

  \end{axis}

  \begin{axis}[
    name=zoom,
    at={(main.north east)}, anchor=north east, 
    xshift= -0.02\linewidth, yshift=-0.02\linewidth, 
    width=0.32\linewidth,
    height=0.28\linewidth,
    tick style={semithick},
    line width=1pt,
    legend=false,           
    xmin=0.49, xmax=0.6,   
    ymin=0.40, ymax=1.05,
  ]
    \addplot+[no marks, densely dotted, forget plot]
      table [x=x,y=y, col sep=comma]{profile_struc_8k_CN.csv};
    \addplot+[no marks, dashed, forget plot, draw=red!]
      table [x=x,y=y, col sep=comma]{profile_struc_32k_CN.csv};
    \addplot+[no marks, forget plot, draw=black!100]
      table [x=x,y=y, col sep=comma]{profile_struc_32k_BE.csv};
  \end{axis}
\end{tikzpicture} 
  \caption{Comparison of concentration profiles (at $t=5.0$ seconds) along the line $y=x$ with different meshes and time stepping.  Top-right: zoomed-in inset near $x=0.5$.}
  \label{fig_sec:numerics_homo2}
\end{figure}
 	 	 
\subsection{Comparison of CCG and DG with full polynomial spaces} \label{sec:numerics_homo2} 	 	 
	Here we contrast the CCG solution to that of the traditional DG solution with full polynomial spaces.  The experiment from section \ref{sec:numerics_homo2} is revisited, but this time the DG method with full polynomial spaces is utilized to approximate the solution. All parameters are the same as in section \ref{sec:numerics_homo}, except we fix the mesh to be structured with 8192 elements. For the DG method, we take $\epsilon = -1$ for both flow and transport, and $\sigma_{e,p}=\sigma_{e,c}= 4 (\vec{n}_e\cdot ({\bm K}\vec{n}_e)) \frac{|e|}{\min\{|E_1|,|E_2|\}},$ where edge $e$ is shared by elements $E_1$ and $E_2$.  For boundary edges the same formula is used with $E_2=E_1$.
	\begin{figure}[htb!]
    \begin{subfigure}[b]{0.23\textwidth}
        \centering 
        \includegraphics[width=\textwidth , trim=0cm 0cm 1.6cm 0cm,clip]{ex_homo_unstruc3.png}
        \caption{CCG}
    \end{subfigure}
    \begin{subfigure}[b]{0.23\textwidth}
        \centering
        \includegraphics[width=\textwidth , trim=0cm 0cm 1.6cm 0cm,clip]{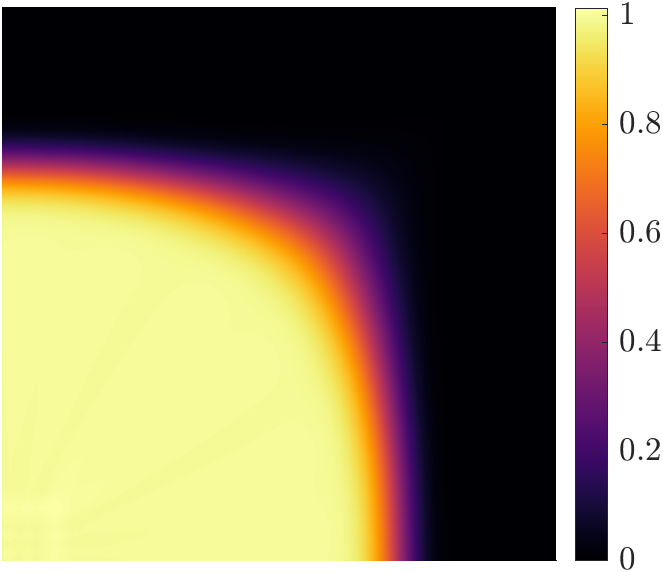}
        \caption{DG (degree one)}
    \end{subfigure}
    \begin{subfigure}[b]{0.233\textwidth}
        \centering
        \includegraphics[width=\textwidth , trim=0cm 0cm 1.6cm 0cm,clip]{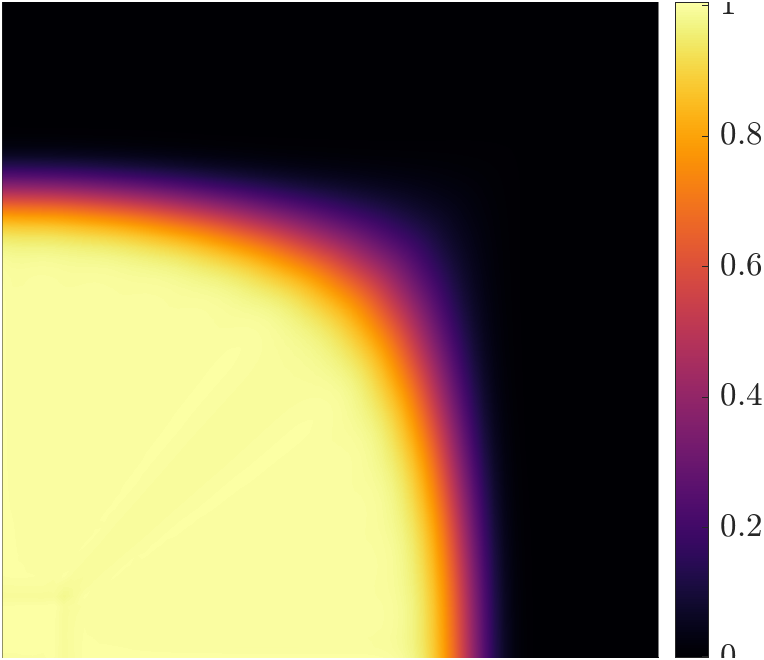}
        \caption{DG (degree two)}
    \end{subfigure}
    \begin{subfigure}[b]{0.23\textwidth}
        \centering 
        \includegraphics[width=1.19\textwidth]{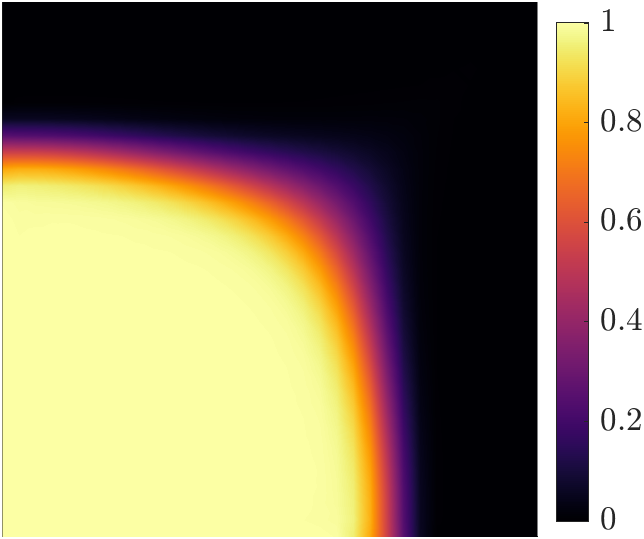}
        \caption{DG (degree three)}
    \end{subfigure}
    \caption[]{ 
    Comparison of CCG and DG concentration at $t=5$ for the homogeneous quarter-five spot experiment from section \ref{sec:numerics_homo}.}
    \label{fig_sec:numerics_homo3}
\end{figure}
	Fig.~\ref{fig_sec:numerics_homo3} plots the concentration at $t=5$, and it is evident that there is good agreement between the CCG and DG approximations.  
	
	To better assess the approximation quality, we consider again consider a snapshot of the concentration along the  cross section $y=x$. The profile snapshots can be found in Fig.~\ref{fig_sec:numerics_homo4}. Using DG in space, and resolving the temporal derivative with a higher-order time marching scheme produces a sharper concentration front. However, backward Euler in time and DG in space yields a more diffused profile.  This is similar to what was observed in Fig.~\ref{fig_sec:numerics_homo2} for the CCG method.    
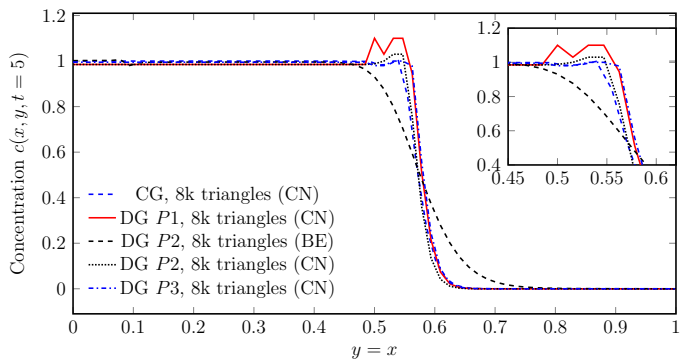
\begin{figure}[htbp!]
  \centering
%
%
%
%
\begin{tikzpicture}[scale=0.6]
  \begin{axis}[
    name=main, 
    width=0.9\linewidth,
    height=0.5\linewidth,
    xlabel={$y=x$},
    ylabel={Concentration $c(x,y,t=5)$},
    legend style={at={(0.45,0.45)},anchor=north east,fill=none,draw=none},
    tick style={semithick},
    line width=1pt,
    xmin=0, xmax=1,
  ]
    \addplot+[no marks, dashed] table [x=x,y=y, col sep=comma]{profile_struc_8k_CN.csv};
    \addlegendentry{CG, 8k triangles (CN)};

    \addplot+[no marks] table [x=x, y=y, col sep=space]{profile_struc_DG1_4k_CN_out.dat};
    \addlegendentry{DG $P1$, 8k triangles (CN)};

    \addplot+[no marks, dashed, draw=black!] table [x=x, y=y, col sep=space]{profile_struc_DG2_4k_BE.dat};
    \addlegendentry{DG $P2$, 8k triangles (BE)};

    \addplot+[no marks, densely dotted, draw=black!] table [x=x, y=y, col sep=space]{profile_struc_DG2_4k_CN_out.dat};
    \addlegendentry{DG $P2$, 8k triangles (CN)};

    \addplot+[no marks, dashdotted, draw=blue!] table [x=x, y=y, col sep=space]{profile_struc_DG3_4k_CN_out.dat};
    \addlegendentry{DG $P3$, 8k triangles (CN)};
  \end{axis}

  \begin{axis}[
    name=zoom,
    at={(main.north east)}, anchor=north east, 
    xshift= 0.00\linewidth, yshift=-0.02\linewidth, 
    width=0.32\linewidth,
    height=0.28\linewidth,
    tick style={semithick},
    line width=1pt,
    legend=false, 
    xmin=0.45, xmax=0.62,     
    ymin=0.40, ymax=1.20,
  ]
    \addplot+[no marks, dashed, forget plot] table [x=x,y=y, col sep=comma]{profile_struc_8k_CN.csv};
    \addplot+[no marks,        forget plot, draw=red!] table [x=x, y=y, col sep=space]{profile_struc_DG1_4k_CN_out.dat};
    \addplot+[no marks, dashed,forget plot, draw=black!] table [x=x, y=y, col sep=space]{profile_struc_DG2_4k_BE.dat};
    \addplot+[no marks, densely dotted, forget plot, draw=black!] table [x=x, y=y, col sep=space]{profile_struc_DG2_4k_CN_out.dat};
    \addplot+[no marks, dashdotted, forget plot, draw=blue!] table [x=x, y=y, col sep=space]{profile_struc_DG3_4k_CN_out.dat};
  \end{axis}
\end{tikzpicture}

  \caption[]{Comparison of CCG and DG concentration profiles (at $t=5.0$ seconds) along the line $y=x$ with different time stepping.  Top-right: zoomed-in inset of the overshoot near $x=0.5$.}
  \label{fig_sec:numerics_homo4}
\end{figure}
	It can also be seen that the DG solutions have more pronounced localized overshoot and undershoot behavior. Focusing on the high-order time marching scheme, the DG solutions have the sharpest fronts, but the CCG solution is also competitive.
\subsection{Computational cost of CCG and DG with full polynomial spaces} \label{sec:computational_cost1} 	
	In addition to solution quality, computational efficiency is also an important aspect of numerical simulations.  For the discretization given in \eqref{eq_DG_flow} of the Darcy problem, the CCG method has a total of $|\mathcal{T}_h|$ unknowns, whereas the $P_k$-DG scheme has 
	$ 
\binom{k+{\tt d}}{{\tt d }}
	\cdot |\mathcal{T}_h|$ unknowns.  Hence, the DG method with polynomial degree $k$ has $\binom{k+{\tt d}}{{\tt d }}$ more unknowns than the CCG scheme.  The CCG method can have significantly fewer unknowns than higher-order IPDG methods, but a better assessment of computational cost is the number of nonzero entries in the discretization matrix \cite{chang2018comparative}.  It is also worth mentioning that for the miscible displacement problem we must solve the coupled PDE system given by equations \eqref{eq_model1}, \eqref{eq_model2}, and \eqref{eq_model3}; as such, the number of primal unknowns will double.
		
	The CCG method is based on the linear reconstruction operator defined in~\eqref{eq_ccg_reconstruction}; and unlike classical DG methods, the analogous basis functions from the CCG method have support which can extend beyond a given element (see Appendix~\ref{sec:app} for examples). As a consequence of this larger support, the average number of nonzero entries per row can be somewhat large compared to $P_1$-IPDG.  Table~\ref{tab:ccg_vs_dg_cost1} displays various information about the number of nonzeros associated with the discretization matrix for the CCG and DG methods in two dimensions. 
\begin{table}[htbp]
\centering
\caption{Results by method and mesh type in two-dimensions.  The abbreviation `nnz' refers to the total number of nonzero entries in the discretization matrix for a given method. Similarly, `nnz per row' refers to the average number of nonzero entries per row in the discretization matrix.}
\begin{tabular}{ll
                S[table-format=5.0]
                S[table-format=7.0]
                c
                }
\toprule
\textbf{Mesh} & \textbf{Method} & $|\mathcal{T}_h|$ & nnz & nnz per row\\
\midrule
\multirow{9}{*}{Structured}
  & $P_1$-DG   & 2048  & 66560   & 10.83 \\
  & $P_1$-DG   & 8192  & 268290  & 10.91 \\
  & $P_1$-DG   & 32768 & 1077200 & 10.95 \\
  & $P_2$-DG   & 2048  & 200060  & 16.28 \\
  & $P_2$-DG   & 8192  & 805630  & 16.39 \\
  & $P_2$-DG   & 32768 & 3233300 & 16.44 \\
  & CCG  & 2048  & 50622   & 24.71 \\
  & CCG  & 8192  & 208320  & 25.42 \\
  & CCG  & 32768 & 845410  & 25.80 \\
\midrule
\multirow{9}{*}{Unstructured}
  & $P_1$-DG   & 3424  & 111970  & 10.90 \\
  & $P_1$-DG   & 14006 & 460150  & 10.95 \\
  & $P_1$-DG   & 56528 & 1861300 & 10.97 \\
  & $P_2$-DG   & 3424  & 336290  & 16.36 \\
  & $P_2$-DG   & 14006 & 1381200 & 16.43 \\
  & $P_2$-DG   & 56528 & 5585500 & 16.46 \\
  & CCG  & 3424  & 85930   & 25.09 \\
  & CCG  & 14006 & 355540  & 25.38 \\
  & CCG  & 56528 & 1444700 & 25.55 \\
\bottomrule
\end{tabular}
\label{tab:ccg_vs_dg_cost1}
\end{table}	
	
	 In terms of total number of unknowns and nonzero entries, the CCG method has fewer than the DG methods with full polynomial spaces.  On the other hand, the stencil (which we will define as the average number of nonzeros per row) for the CCG method is quite large. Some techniques exist to reduce the stencil for the CCG method. For instance, certain trace interpolation operators that shrink the support of CCG basis functions have been explored in \cite{ref_ccg3}.

For computations in three-dimensions, the situation is similar: the CCG method has fewer unknowns and nonzero entries in the discretization matrix, but the average number of nonzeros per row is large.  The 3D results are collected in Table~\ref{tab:ccg_vs_dg_cost2}. In summary, in two and three dimensions, the $P_1$-IPDG method has roughly 24–40\% more nonzero entries in the discretization matrix when compared to CCG with the barycentric trace reconstruction (see Fig.~\ref{fig:ccg_vs_dg_cost_percent}). This represents a considerable computational saving, particularly given that two PDEs must be solved at every time step.  
\begin{table}[htbp!]
\centering
\caption{Results by method in three-dimensions (tetrahedral meshes).  The abbreviation `nnz' refers to the total number of nonzero entries in the discretization matrix for a given method. Similarly, `nnz per row' refers to the average number of nonzero entries per row in the discretization matrix.}
\begin{tabular}{l
                S[table-format=5.1]
                S[table-format=7.0]
                c
                }
\toprule
 \textbf{Method} & $|\mathcal{T}_h|$ & nnz & nnz per row\\
\midrule
   $P_1$-DG   & 3072   & 221952   & 18.06 \\
   $P_1$-DG   & 24576  & 1821696  & 18.53 \\
   $P_1$-DG   & 196608 & 14757888 & 18.76 \\
   $P_1$-DG   & 1572864 & 118800384& 18.88 \\   
   CCG        & 3072   & 157574   & 51.29 \\
   CCG        & 24576  & 1389062  & 56.52 \\
   CCG        & 196608 & 11644934 & 59.22 \\
   CCG        & 1572864& 95326214 & 60.60 \\
\bottomrule
\end{tabular}
\label{tab:ccg_vs_dg_cost2}
\end{table}		 
		 	 	
\begin{figure}[htb!]
\centering
\begin{tikzpicture}
\begin{axis}[
  width=11cm, height=6.4cm,
  xmode=log, log basis x=10,
  xmin=1e3, xmax=2e6,                 
  xtick={1e3,1e4,1e5,1e6},
  xticklabels={$10^{3}$,$10^{4}$,$10^{5}$,$10^{6}$},
  minor x tick num=0,                  
  xlabel={Number of elements},
ylabel={ $100\cdot\frac{\textrm{nnz}(DG)- \textrm{nnz}(CCG)}{\textrm{nnz}(CCG)}$  },
  yticklabel=\pgfmathprintnumber{\tick}\,\%,
  yticklabel style={/pgf/number format/fixed},
  ymin=24, ymax=41, ytick distance=2, minor y tick num=0,
  grid=both,
  mark options={solid},
]
\addplot+[thick, mark=o] table[ 
  row sep=\\,
  x=H,
  y expr={100*(\thisrow{DG1}/\thisrow{CCG}-1)}
] {
H        CCG        DG1 \\
2048     50622      66560 \\
3424     85930      111970 \\
8192     208320     268290 \\
32768    845410     1077200 \\
131072   3405406    4317184 \\
524288   13670622   17285120 \\
};

\addplot+[thick, mark=*] table[ 
  row sep=\\,
  x=H,
  y expr={100*(\thisrow{DG1}/\thisrow{CCG}-1)}
] {
H        CCG        DG1 \\
3072     157574      221952 \\
13182    728864      971412
24576    1389062     1821696 \\
105456   6179582     7892976
196608   11644934    14757888 \\
584016   35079422    44004336\\ 
1572864  95326214    118800384\\
};
\legend{2D (triangles),3D (tetrahedra)} 
\end{axis}
\end{tikzpicture}
\label{fig1}
\caption[]{CCG vs $P_1$-IPDG nonzero entry improvement (values larger than 0\% imply CCG is more efficient).  Vertical axis represents the percent increase of nonzero entries from the $P_1$-IPDG method over CCG.}
\label{fig:ccg_vs_dg_cost_percent}
\end{figure}
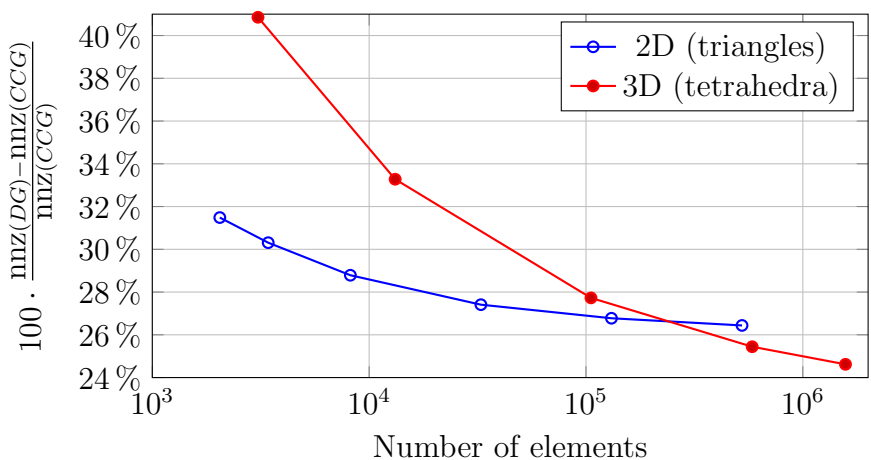

 \subsection{Heterogeneous porous medium (permeability lens)}\label{sec:numerics_lens} 
 Here a heterogeneous medium is considered.  The domain is $[0,1]^2~m^2$, and the permeability is $9.44 \cdot 10^{-6} ~m^2$ in the square $[0.25,0.5]^2$, and $9.44 \cdot 10^{-3} ~m^2$ everywhere else. Such a contrast in permeability creates a region where the fluid concentration cannot penetrate.  We take $\sigma_{e,p} = 10^{-1},~ \sigma_{e,c}= 10^{-2}$ and $\epsilon=1$.  All of the other parameters are the same as in section~\ref{sec:numerics_homo}. Fig.~\ref{fig_sec:numerics_lens1} displays snapshots of the simulation under mesh refinement.
\begin{figure}[htbp!]
   \centering
    \begin{subfigure}[t]{0.23\textwidth}
        \centering
        \caption*{$t=1.5$ (seconds)}
        \includegraphics[width=\textwidth , trim=0cm 0cm 1.6cm 0cm,clip]{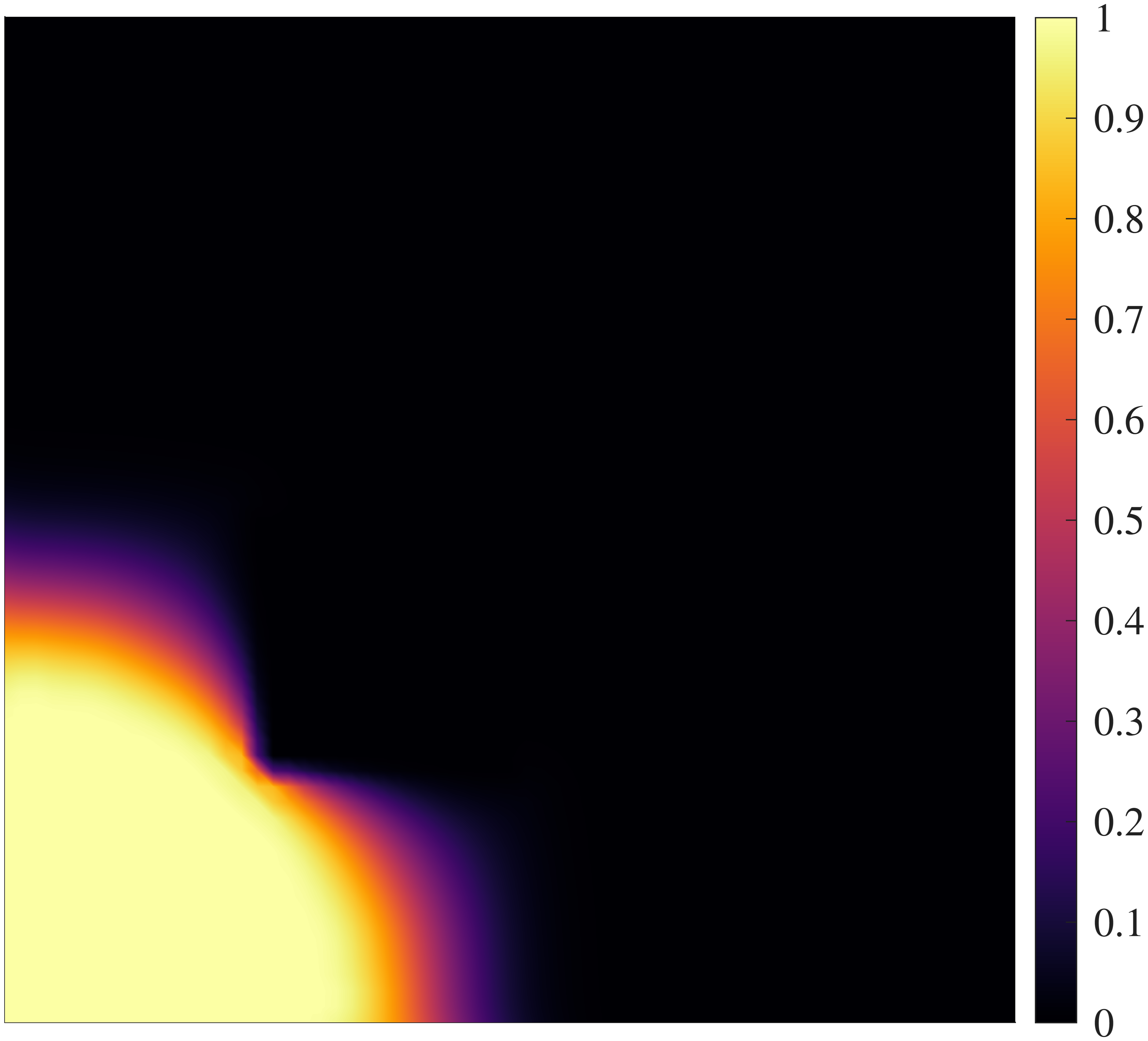} 
    \end{subfigure}
    \begin{subfigure}[t]{0.23\textwidth}
        \centering
        \caption*{$t=2.5$ (seconds)}
        \includegraphics[width=\textwidth , trim=0cm 0cm 1.6cm 0cm,clip]{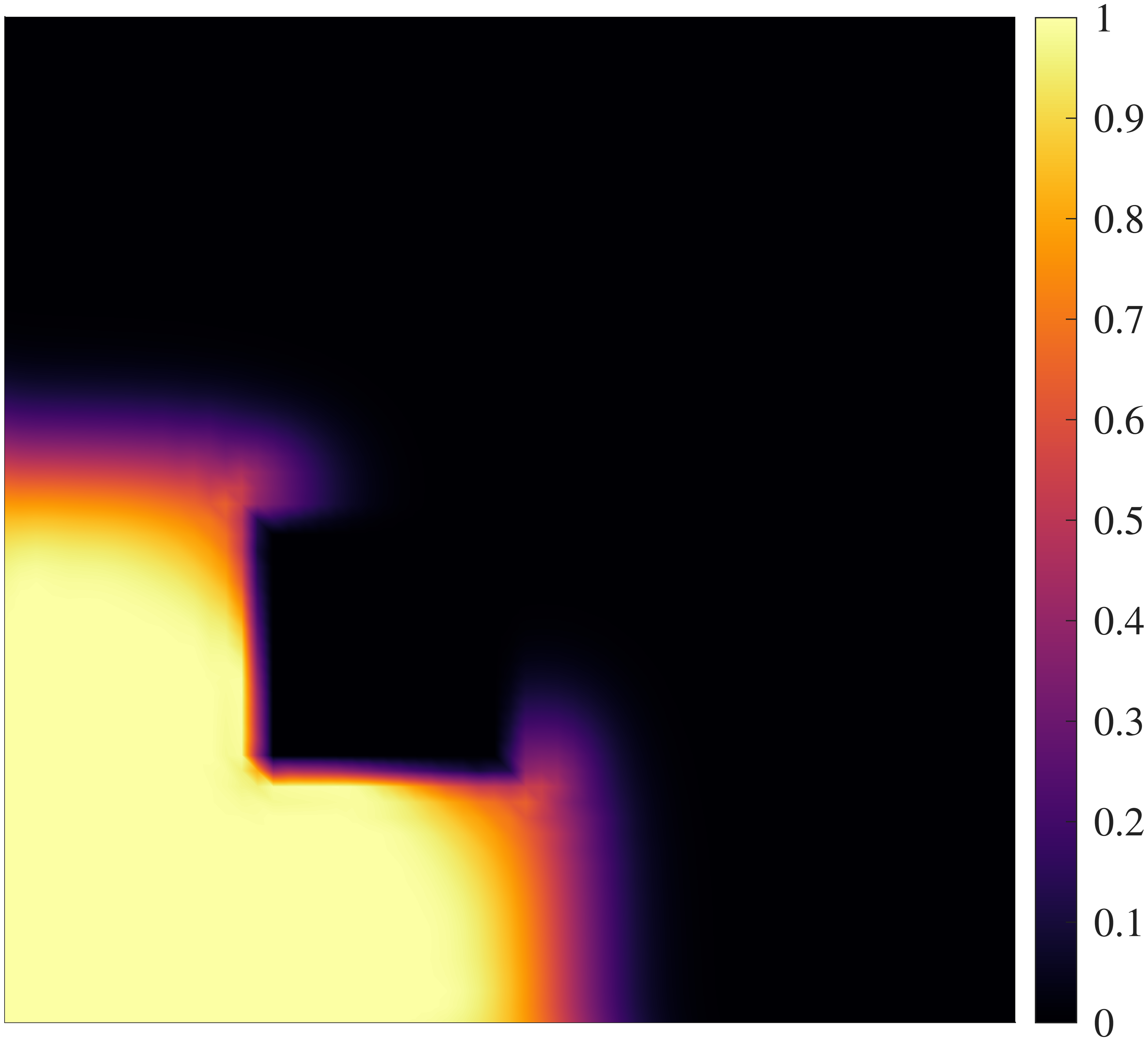} 
    \end{subfigure}
    \begin{subfigure}[t]{0.23\textwidth}
        \centering 
        \caption*{$t=5.0$ (seconds)}     
        \includegraphics[width=\textwidth , trim=0cm 0cm 1.6cm 0cm,clip]{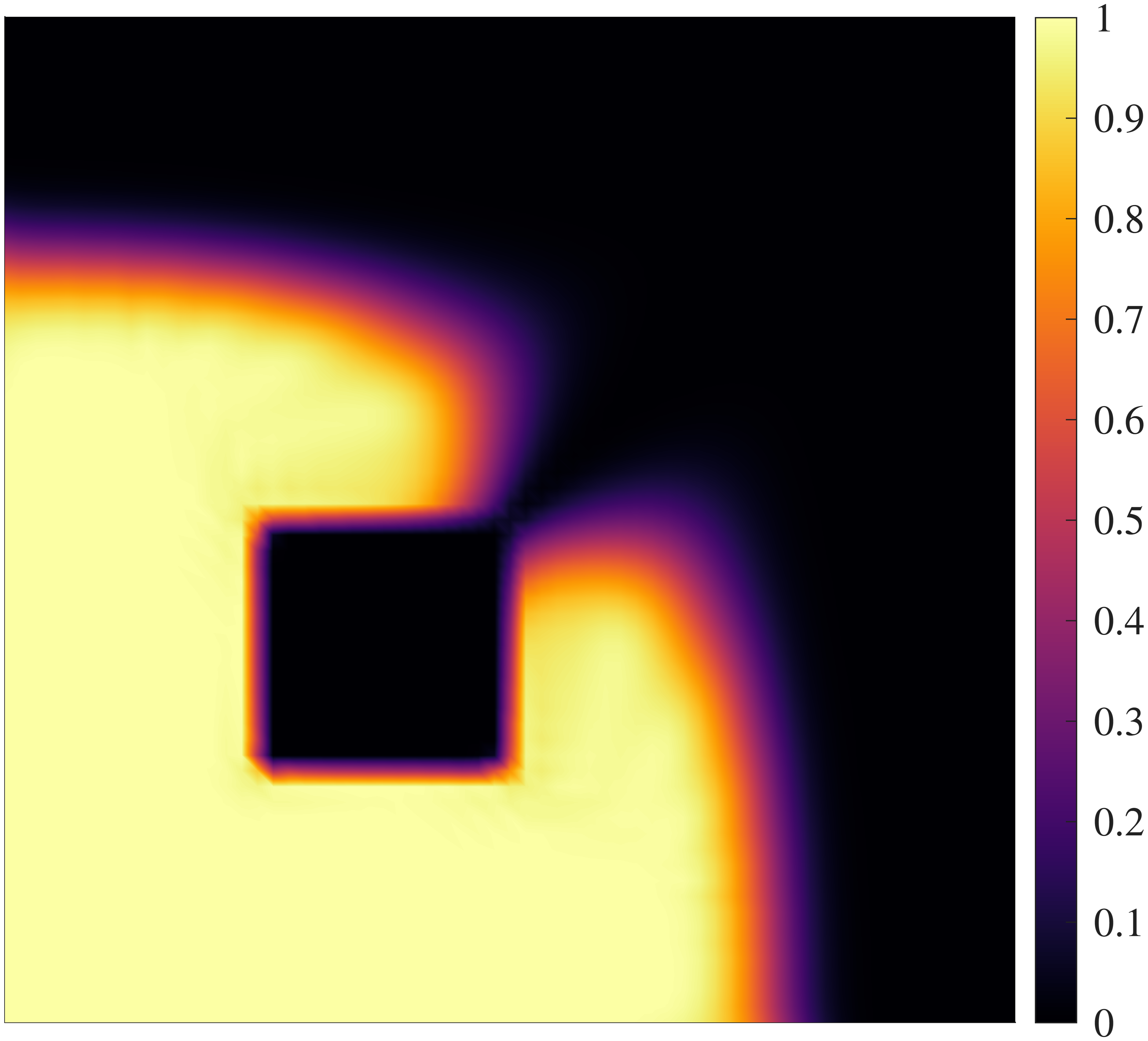} 
    \end{subfigure}
    \begin{subfigure}[t]{0.23\textwidth}
        \centering
		\caption*{$t=7.5$ (seconds)}     
        \includegraphics[width=1.112\textwidth]{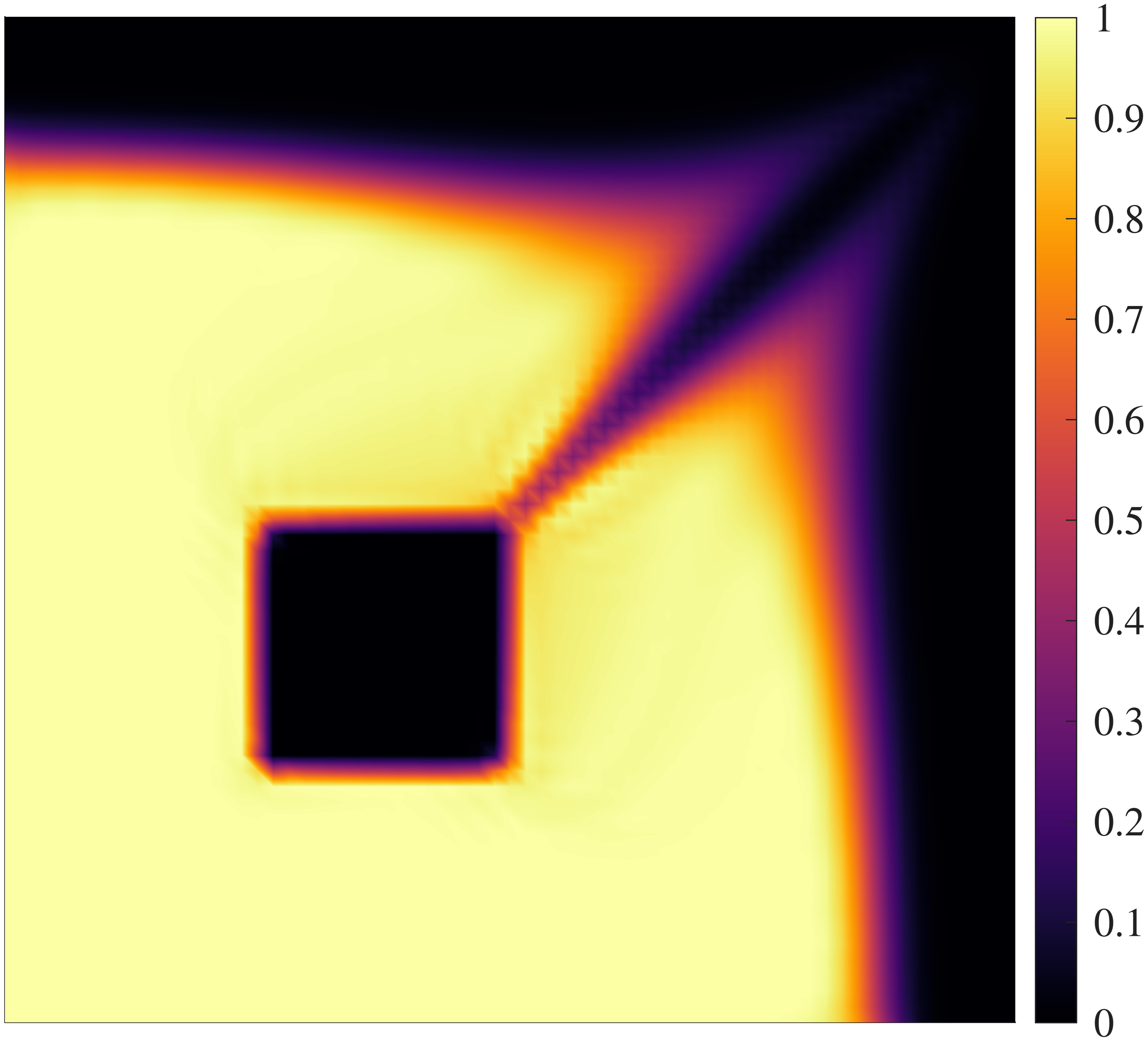} 
    \end{subfigure}
    
    \vspace{0.5em} 
    \begin{subfigure}[b]{0.23\textwidth}
        \centering 
        \includegraphics[width=\textwidth , trim=0cm 0cm 1.6cm 0cm,clip]{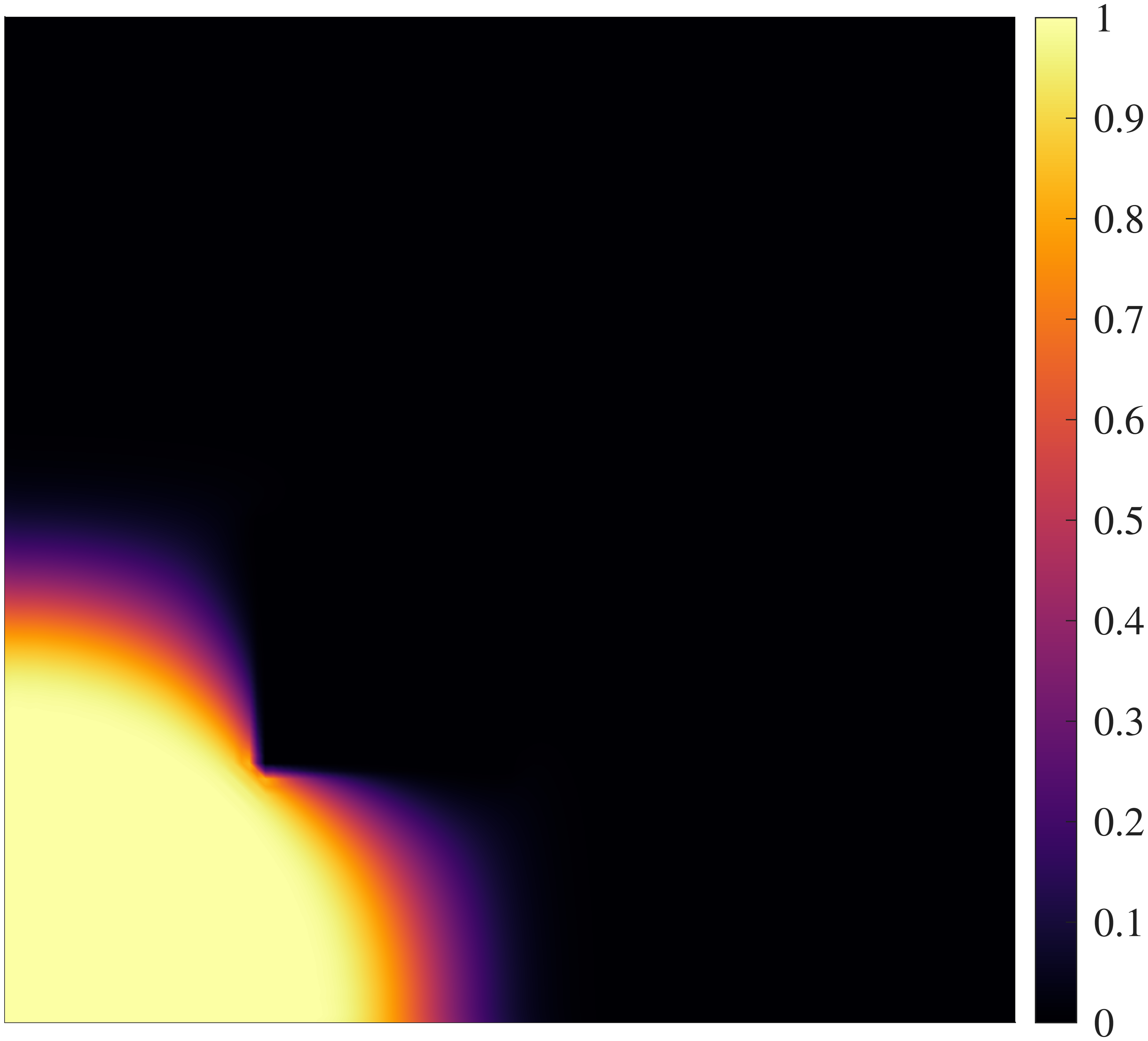}
    \end{subfigure}
    \begin{subfigure}[b]{0.23\textwidth}
        \centering
        \includegraphics[width=\textwidth , trim=0cm 0cm 1.6cm 0cm,clip]{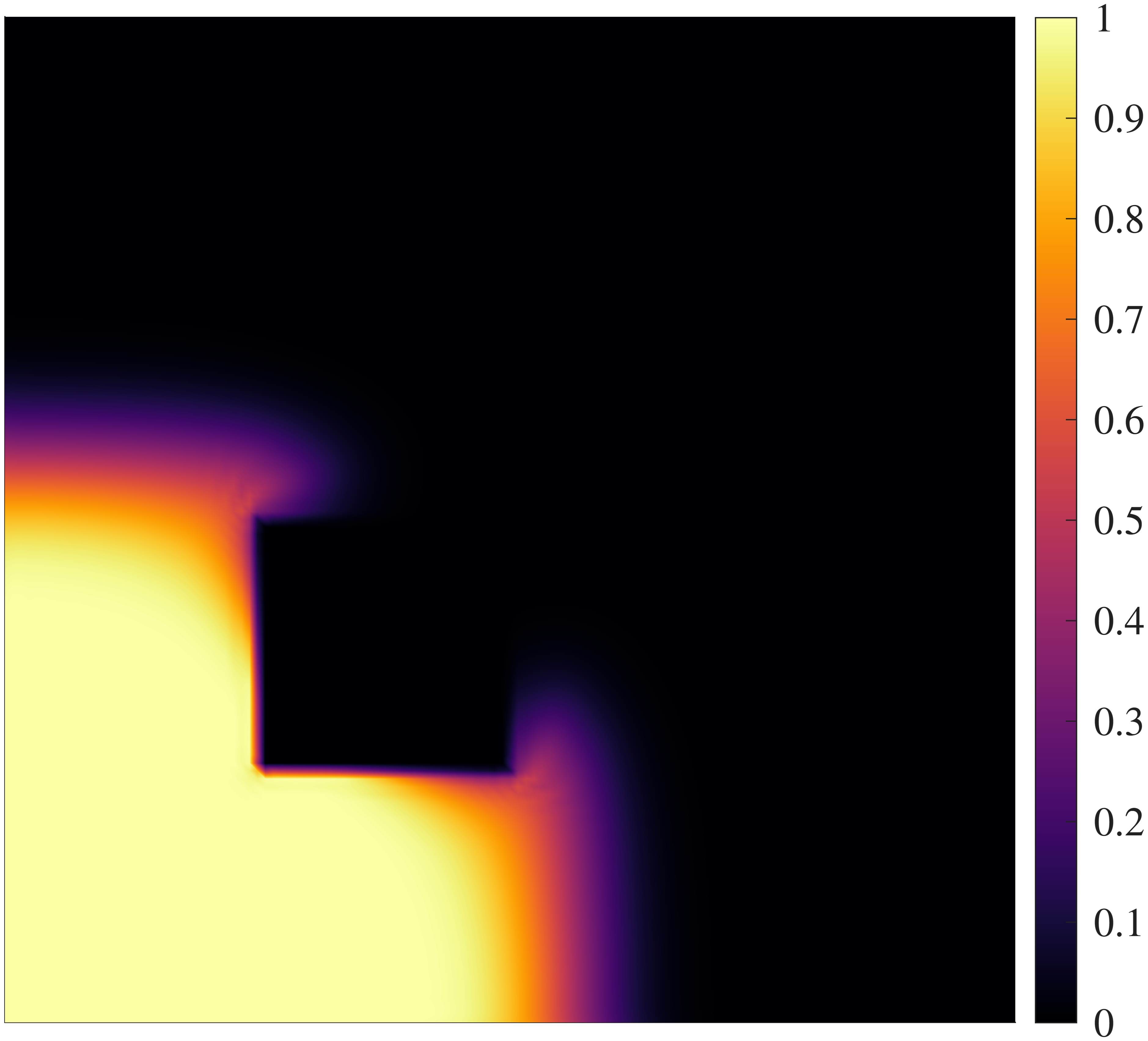}
    \end{subfigure}
    \begin{subfigure}[b]{0.23\textwidth}
        \centering
        \includegraphics[width=\textwidth , trim=0cm 0cm 1.6cm 0cm,clip]{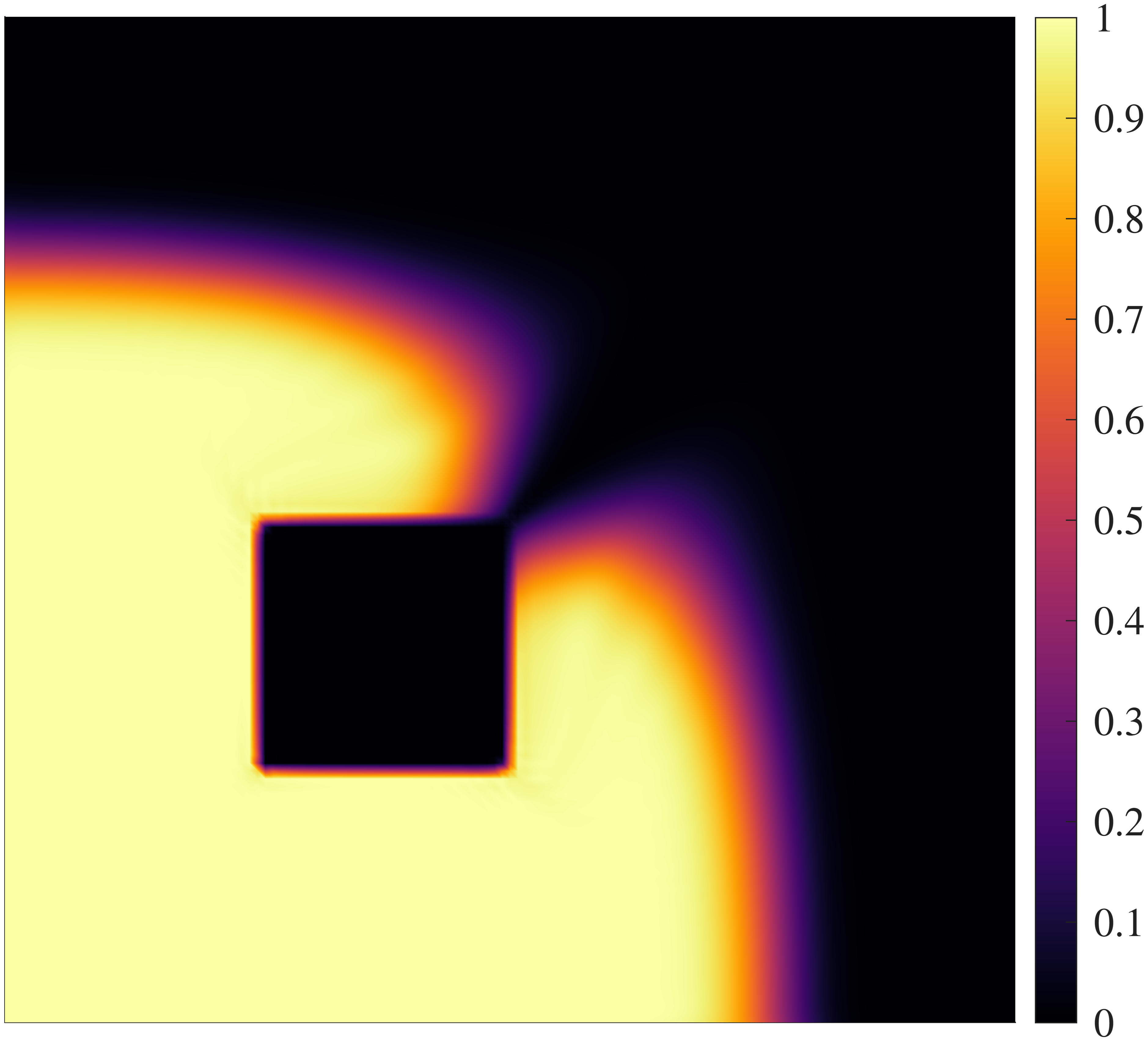}
    \end{subfigure}
    \begin{subfigure}[b]{0.23\textwidth}
        \centering
        \includegraphics[width=1.112\textwidth]{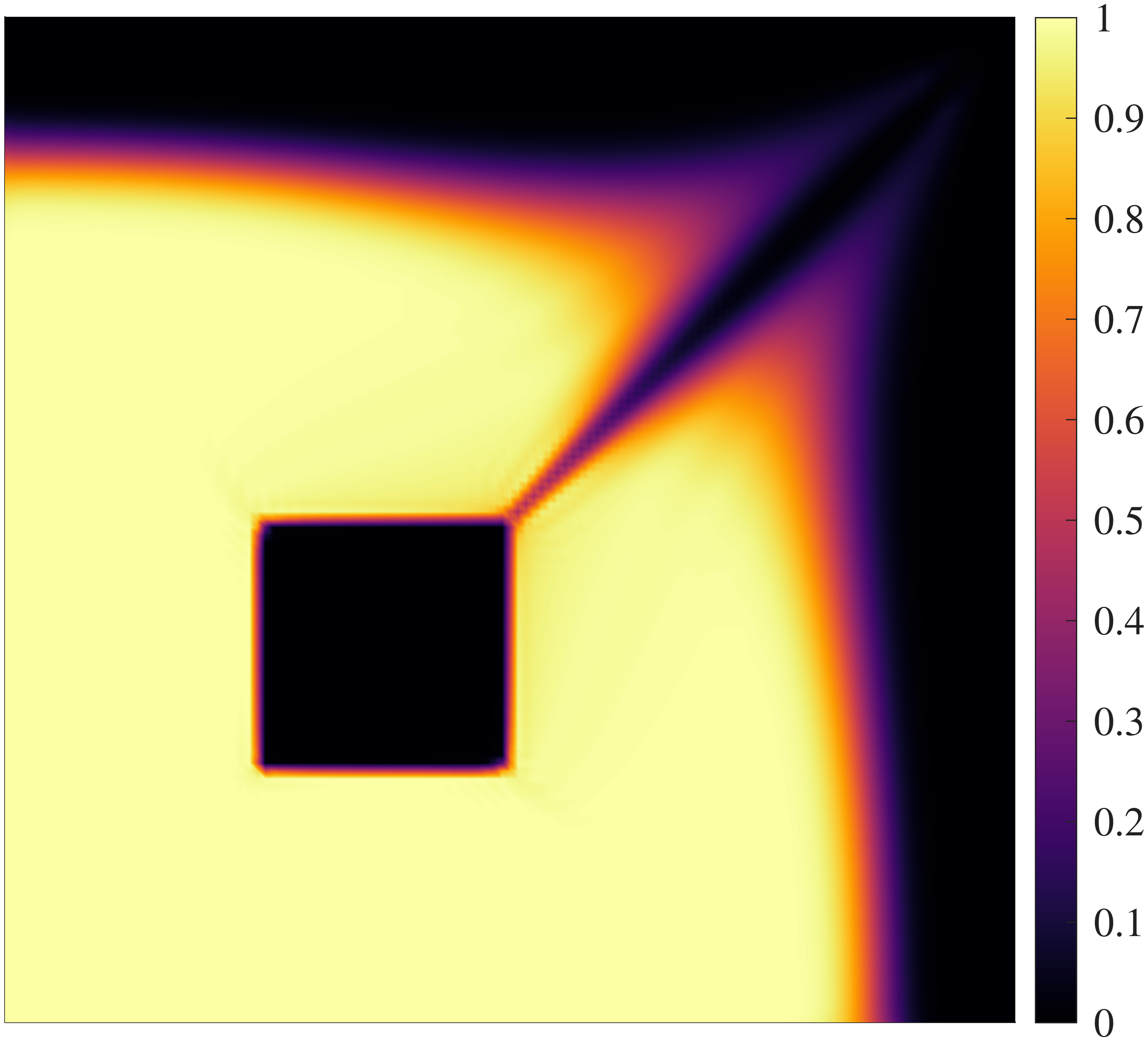}
    \end{subfigure} 
    \caption[]{ 
    Heterogeneous quarter-five spot experiment from section \ref{sec:numerics_lens}.  Plot of the CCG fluid concentration: the first row corresponds to a mesh with 8192 triangles, and the second row corresponds to a mesh with 32768 triangles.}
    \label{fig_sec:numerics_lens1}
\end{figure} 	 

 As time evolves, the fluid concentration moves from the injection well to the production well, however, the region of low permeability acts as a barrier.  The first row of Fig.~\ref{fig_sec:numerics_lens1} uses a mesh with 8192 elements, and the second row are the results from using 32768 elements.  It is clear that as the mesh is refined, the fluid concentration around the lens is captured sharply with little smearing or diffusion (see $t=5.0$ or $t=7.5$).  The CCG scheme is able to generate the expected solution behavior even in the case of heterogeneous permeability.

 \subsection{Heterogeneous porous medium (SPE 10)}  
The CCG method is tested with a realistic, highly heterogeneous permeability field from the 10th SPE comparative solution project.  To better showcase the capabilities of the CCG method, a few 2D slices are extracted from the data set.  The permeability is isotropic in the $x$ and $y$ directions, but can vary several orders of magnitude in a very short distance.  We take $\epsilon=1$, $\sigma_{e,p}=10$, $\sigma_{e,c}=10^{-2}$, and the permeability fields are taken from the SPE 10 data see (see Fig.~\ref{fig:SPE10_layers}). A quarter-law power mixing rule for viscosity is adopted:
 	\[
 	\mu(c) = 
 	(c   \mu_s^{-0.25} + (1-c)   \mu_o^{-0.25})^{-4},\quad  \mu_s = 2.9,~\mu_o = 5.8.
 	\] 
 The remaining parameters are the same as in Section~\ref{sec:numerics_homo}.
\begin{figure}[htbp!]
   \centering
    \begin{subfigure}[t]{0.2\textwidth}
        \centering
        \caption*{Layer 1 (SPE10)}
        \includegraphics[width=\textwidth]{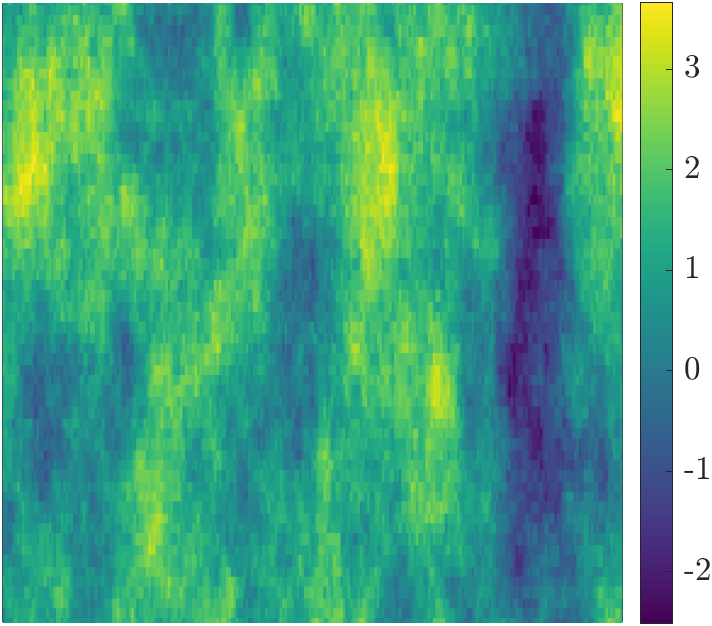} 
    \end{subfigure}
    \begin{subfigure}[t]{0.2\textwidth}
        \centering
        \caption*{Layer 44 (SPE10)}
        \includegraphics[width=\textwidth]{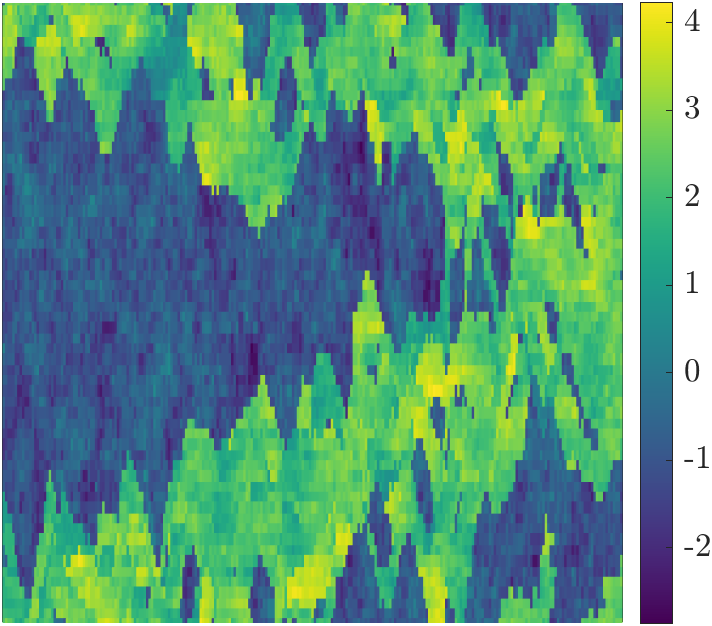} 
    \end{subfigure}
    \begin{subfigure}[t]{0.2\textwidth}
        \centering 
        \caption*{Layer 74 (SPE10)}     
        \includegraphics[width=\textwidth, trim=0cm 0cm 0cm 0cm,clip]{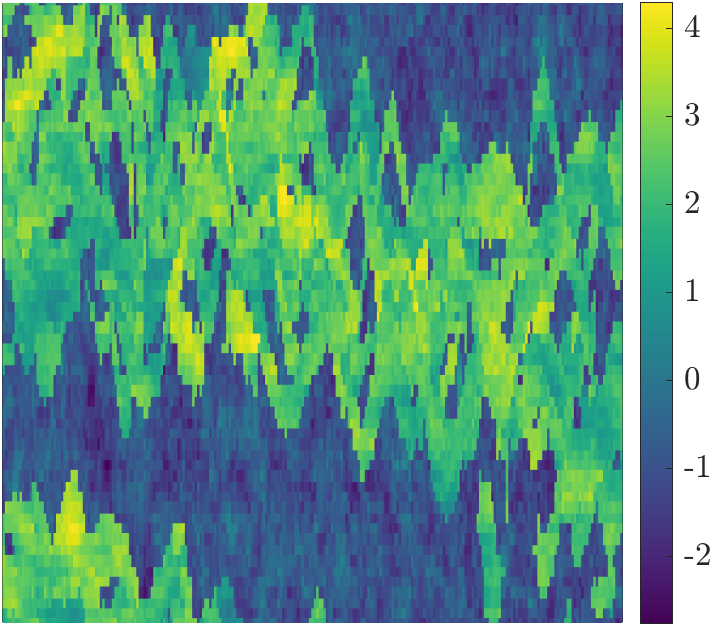} 
    \end{subfigure}
    \caption[]{Permeability layers from the SPE 10 project (base ten logarithmic scale) on a $256 \times 256$ grid.}
    \label{fig:SPE10_layers}
\end{figure}

In Fig.~\ref{fig:SPE10_layer1_CCG}, a comparison between CCG (top row) and high-order IPDG with full polynomial spaces (bottom row) is given. Here, a $(P_1-P_2)-$DG approximation is used (piecewise linears for the concentration, and piecewise quadratics for the pressure).  Both methods are roughly in agreement with the global behavior of the concentration.  However, the IPDG method with full polynomial spaces exhibits spurious oscillations throughout the domain.  These oscillations are nonphysical, but remain localized and bounded. This phenomena is well known, and there are various techniques which can reduce or eliminate these oscillations ($H$(div) flux reconstruction \cite{bastian2003superconvergence,ern2007accurate}, weighted averages \cite{ern2009discontinuous}, limiters/stablizations \cite{berger2005analysis,guermond2011entropy}, and so on). In contrast, the CCG method does not have significant oscillations, but the concentration is slightly more diffusive. A clear advantage of the CCG method is that less post-processing of the Darcy flux and fluid concentration is required.
\begin{figure}[htbp!]
   \centering
    \begin{subfigure}[t]{0.23\textwidth}
        \centering
        \caption*{$t=0.5$ (seconds)}
        \includegraphics[width=\textwidth , trim=0cm 0cm 1.6cm 0cm,clip]{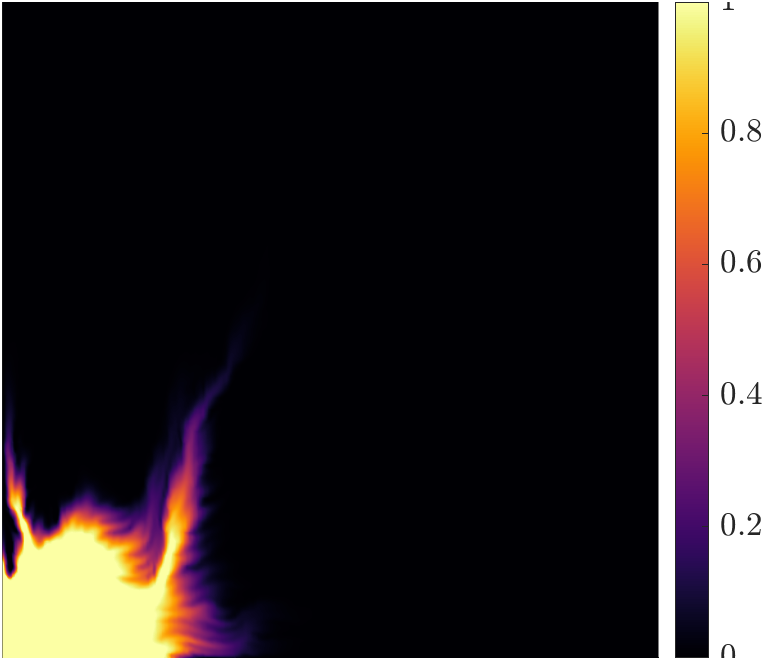} 
    \end{subfigure}
    \begin{subfigure}[t]{0.23\textwidth}
        \centering
        \caption*{$t=1.0$ (seconds)}
        \includegraphics[width=\textwidth , trim=0cm 0cm 1.6cm 0cm,clip]{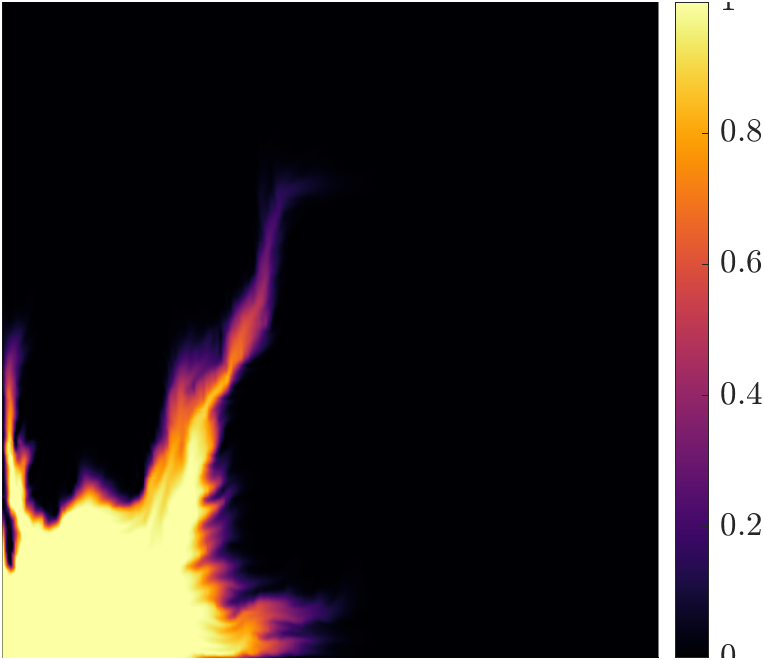} 
    \end{subfigure}
    \begin{subfigure}[t]{0.23\textwidth}
        \centering 
        \caption*{$t=2.5$ (seconds)}     
        \includegraphics[width=1.13\textwidth]{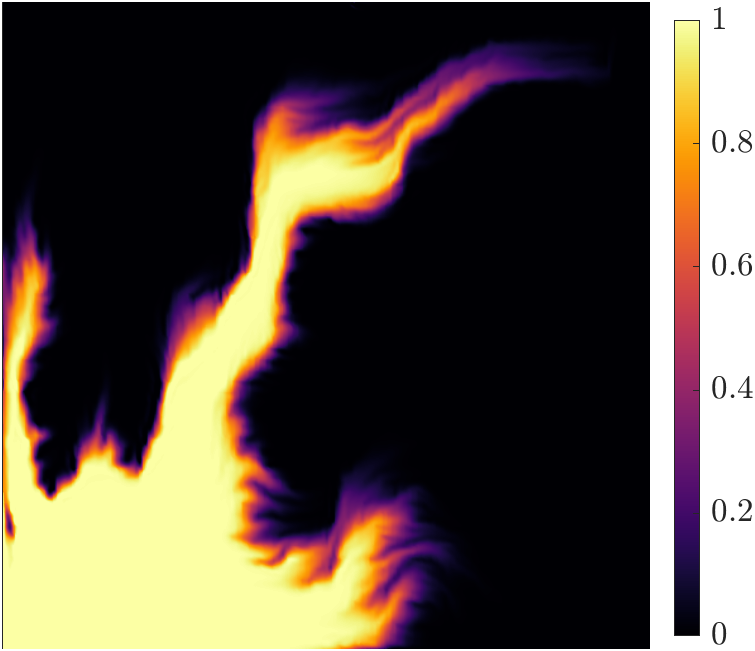} 
    \end{subfigure} 
    
    \vspace{0.5em} 
    \begin{subfigure}[b]{0.23\textwidth}
        \centering 
        \includegraphics[width=\textwidth , trim=0cm 0cm 1.6cm 0cm,clip]{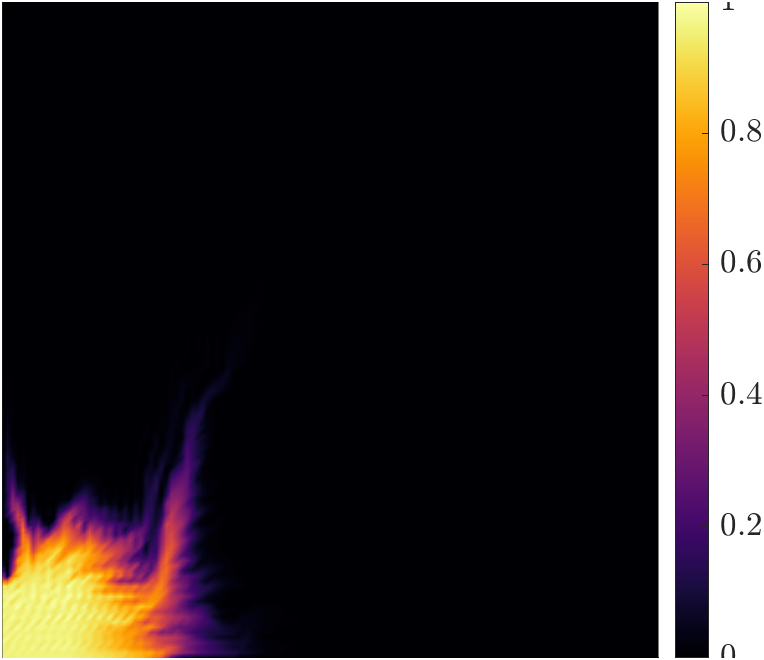}
    \end{subfigure}
    \begin{subfigure}[b]{0.23\textwidth}
        \centering
        \includegraphics[width=\textwidth , trim=0cm 0cm 1.6cm 0cm,clip]{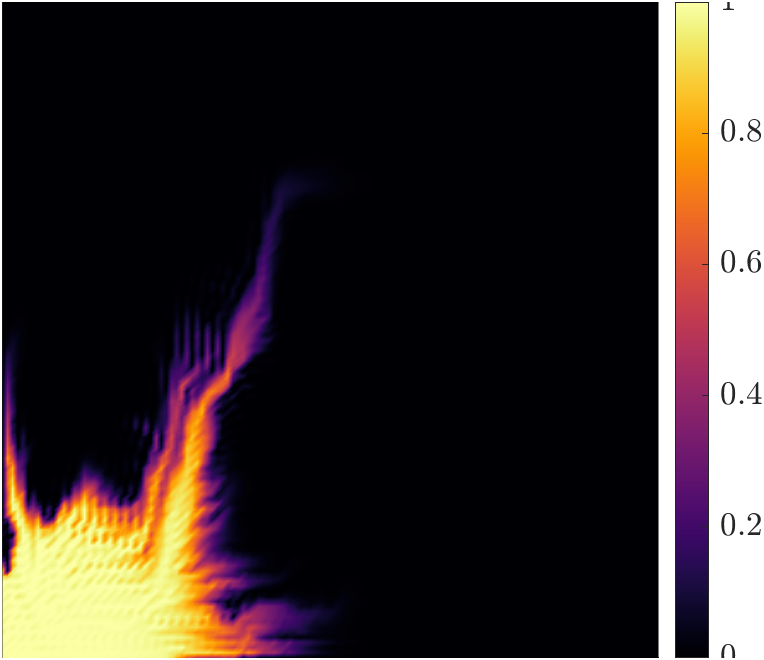}
    \end{subfigure}
    \begin{subfigure}[b]{0.23\textwidth}
        \centering
        \includegraphics[width=1.14\textwidth]{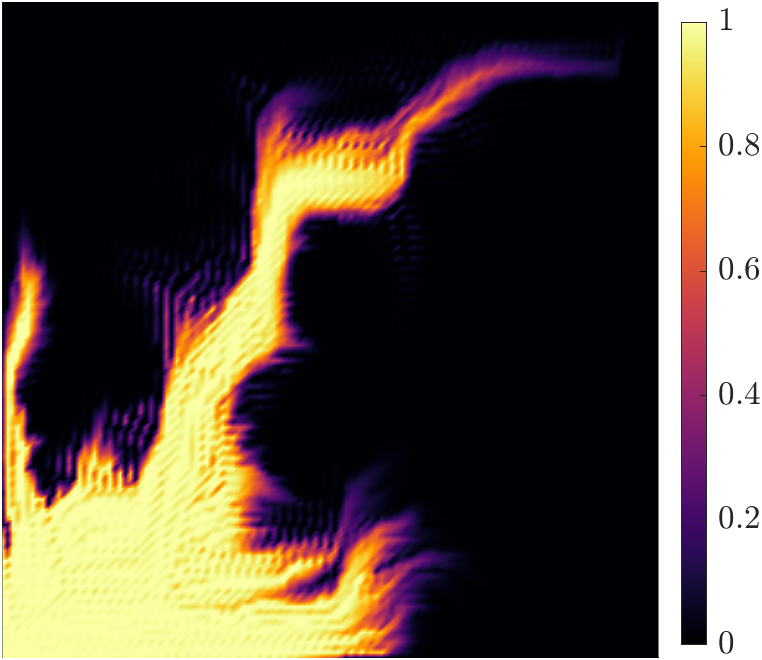}
    \end{subfigure} 
    \caption[]{SPE 10 layer 1.  Concentration snapshots for CCG (top row) and IPDG with full polynomial spaces (bottom row).  The IPDG scheme uses $P_1$ for transport system and $P_2$ for the Darcy system.}
    \label{fig:SPE10_layer1_CCG}
    \end{figure}

\begin{figure}[htbp!]
   \centering
    \begin{subfigure}[t]{0.23\textwidth}
        \centering
        \caption*{$t=0.5$ (seconds)}
        \includegraphics[width=\textwidth , trim=0cm 0cm 1.6cm 0cm,clip]{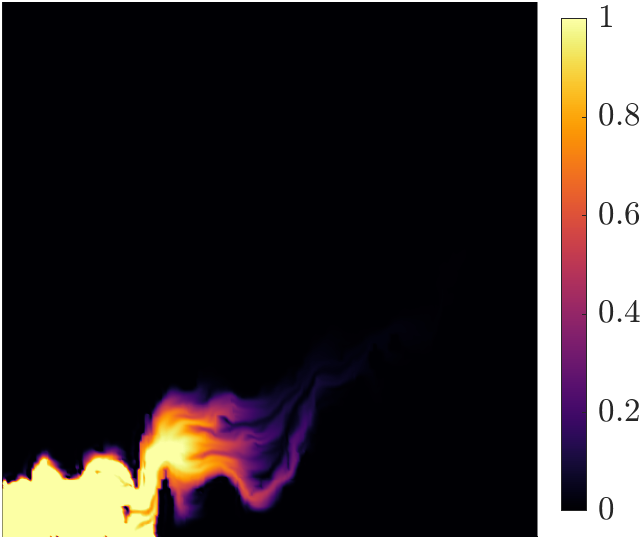} 
    \end{subfigure}
    \begin{subfigure}[t]{0.23\textwidth}
        \centering
        \caption*{$t=1.0$ (seconds)}
        \includegraphics[width=\textwidth , trim=0cm 0cm 1.6cm 0cm,clip]{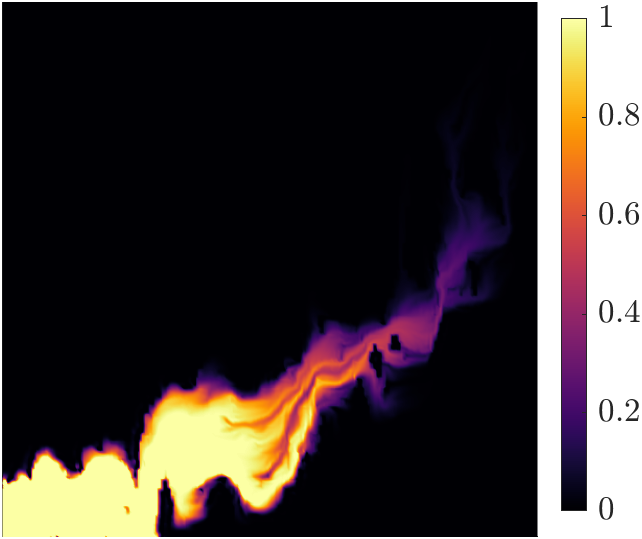} 
    \end{subfigure}
    \begin{subfigure}[t]{0.23\textwidth}
        \centering 
        \caption*{$t=2.5$ (seconds)}     
        \includegraphics[width=1.17\textwidth]{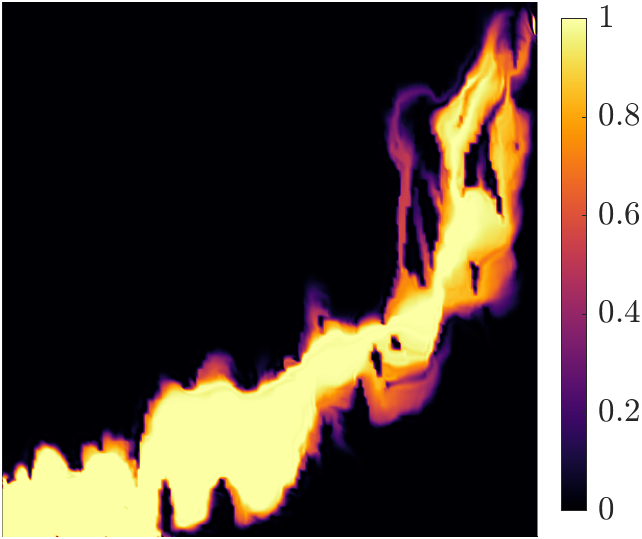} 
    \end{subfigure} 
    
    \vspace{0.5em} 
    \begin{subfigure}[b]{0.23\textwidth}
        \centering 
        \includegraphics[width=\textwidth , trim=0cm 0cm 1.6cm 0cm,clip]{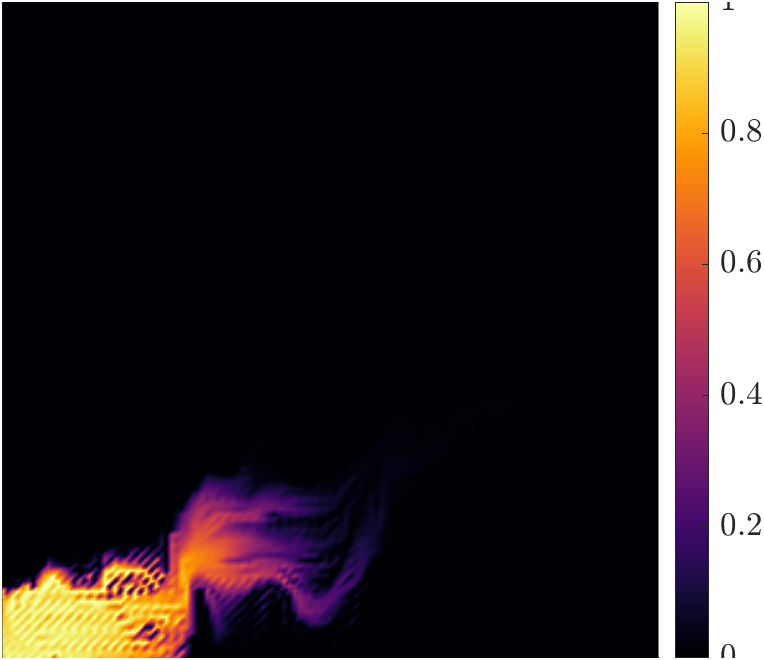}
    \end{subfigure}
    \begin{subfigure}[b]{0.23\textwidth}
        \centering
        \includegraphics[width=\textwidth , trim=0cm 0cm 1.6cm 0cm,clip]{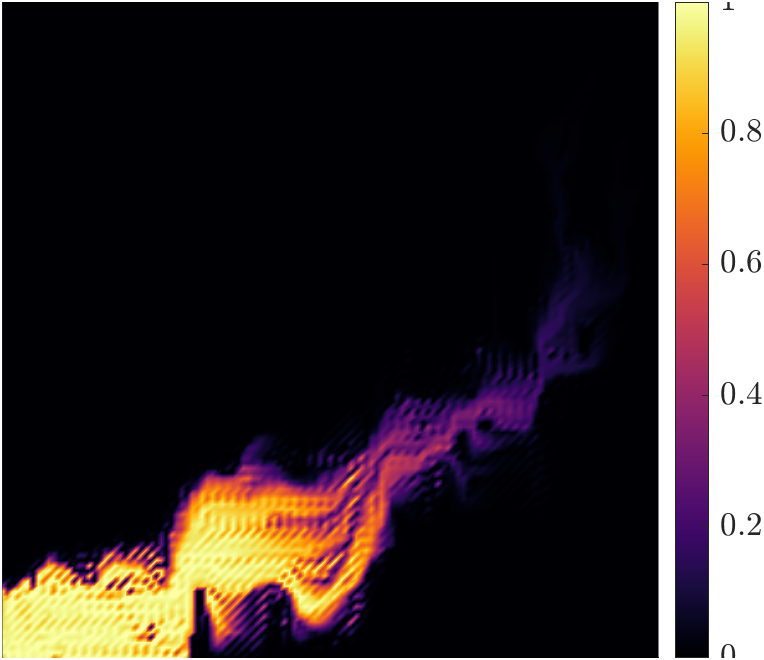}
    \end{subfigure}
    \begin{subfigure}[b]{0.23\textwidth}
        \centering
        \includegraphics[width=1.14\textwidth]{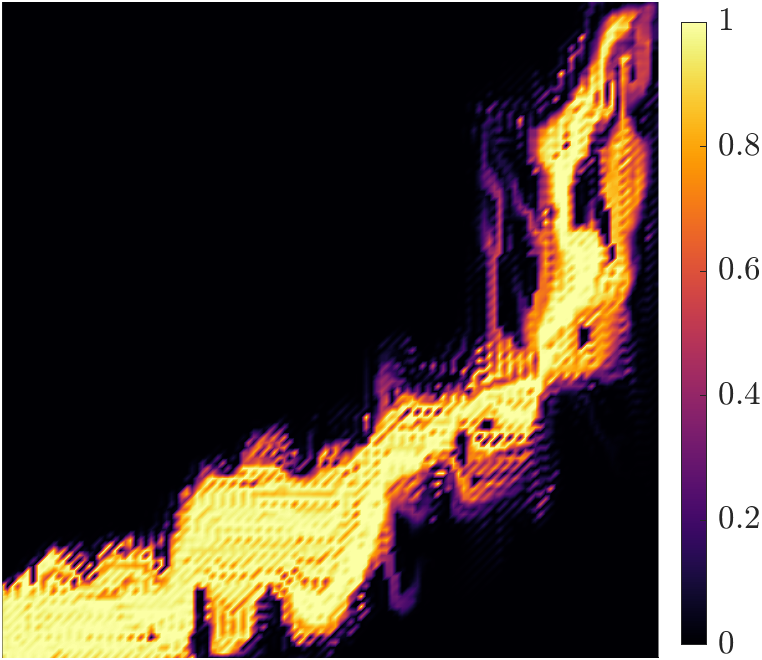}
    \end{subfigure} 
    \caption[]{SPE 10 layer 44.  Concentration snapshots for CCG (top row) and IPDG with full polynomial spaces (bottom row).    The IPDG scheme uses $P_1$ for transport system and $P_2$ for the Darcy system.}
    \label{fig:SPE10_layer44_CCG}
    \end{figure} 
\begin{figure}[htbp!]
   \centering 
    \begin{subfigure}[t]{0.23\textwidth}
        \centering
        \caption*{$t=0.5$ (seconds)}
        \includegraphics[width=\textwidth , trim=0cm 0cm 1.6cm 0cm,clip]{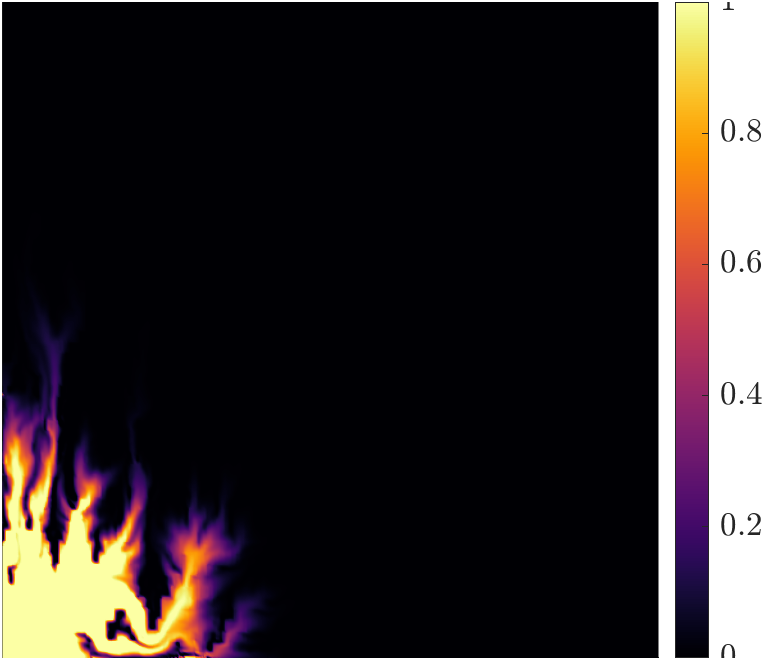} 
    \end{subfigure}
    \begin{subfigure}[t]{0.23\textwidth}
        \centering
        \caption*{$t=1.0$ (seconds)}
        \includegraphics[width=\textwidth , trim=0cm 0cm 1.6cm 0cm,clip]{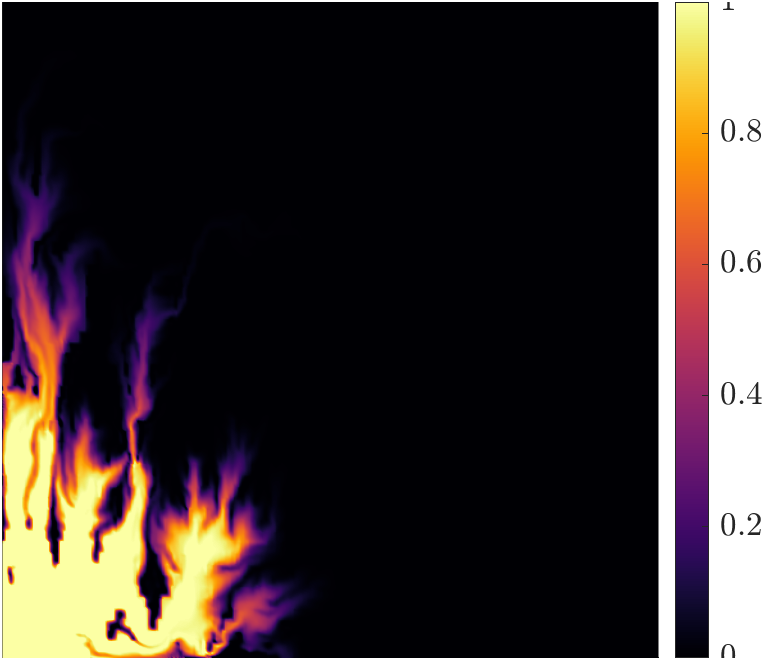} 
    \end{subfigure}
    \begin{subfigure}[t]{0.23\textwidth}
        \centering 
        \caption*{$t=2.5$ (seconds)}   
        \includegraphics[width=1.13\textwidth]{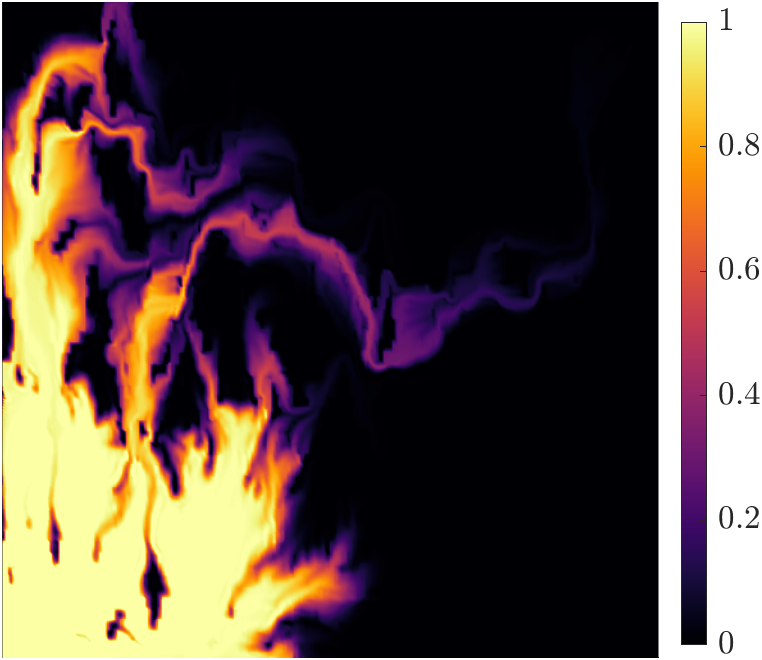} 
    \end{subfigure}    
    
    \vspace{0.5em} 
    \begin{subfigure}[b]{0.23\textwidth}
        \centering 
        \includegraphics[width=\textwidth , trim=0cm 0cm 1.6cm 0cm,clip]{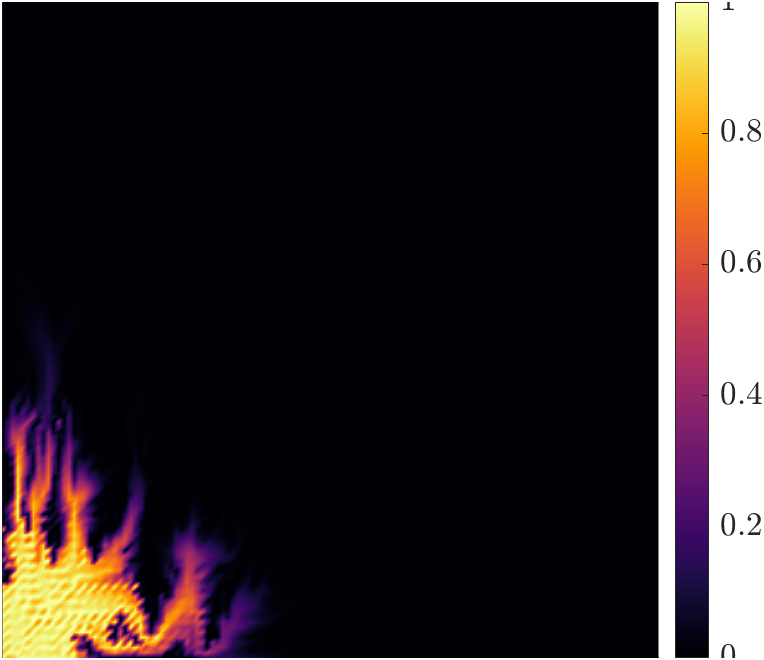}
    \end{subfigure}
    \begin{subfigure}[b]{0.23\textwidth}
        \centering
        \includegraphics[width=\textwidth , trim=0cm 0cm 1.6cm 0cm,clip]{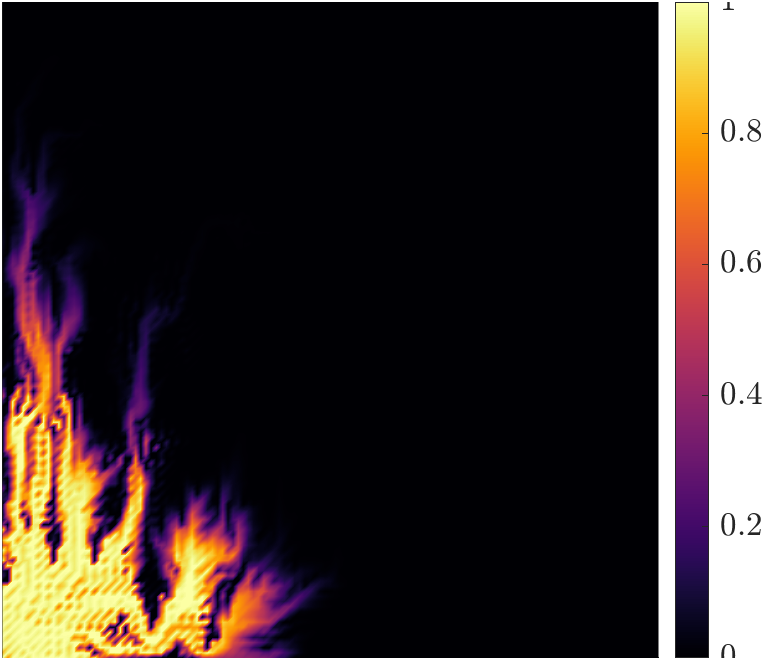}
    \end{subfigure}
    \begin{subfigure}[b]{0.23\textwidth}
        \centering
        \includegraphics[width=1.14\textwidth]{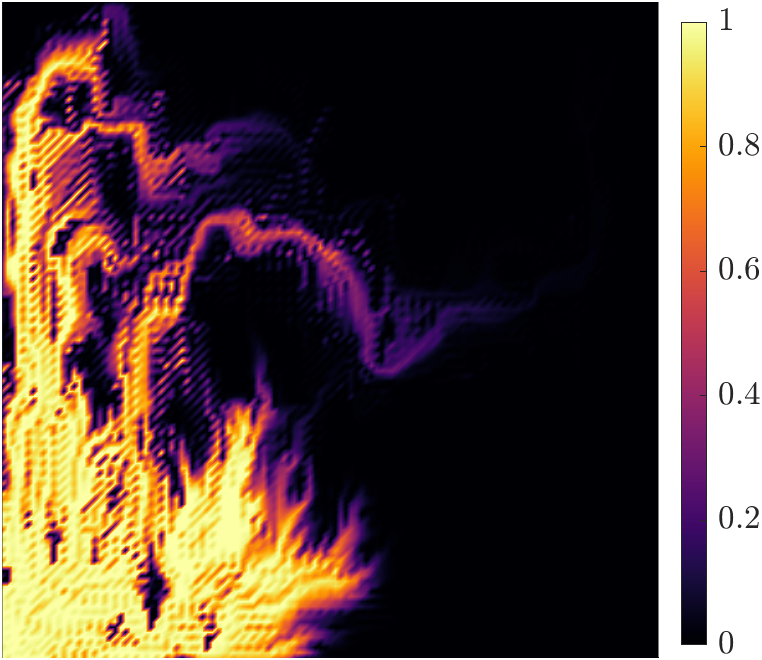}
    \end{subfigure} 
    \caption[]{SPE 10 layer 74.  Concentration snapshots for CCG (top row) and IPDG with full polynomial spaces (bottom row).    The IPDG scheme uses $P_1$ for transport system and $P_2$ for the Darcy system.}
    \label{fig:SPE10_layer74_CCG}
    \end{figure}  
For more the challenging permeability layers (layer 44 and layer 74), the nonphysical oscillations are more noticeable in the IPDG method.  We again mention that various postprocessing techniques can be applied to eliminate these undesirable effects from the DG solution. Overall, both methods capture the general behavior of the fluid concentration: it moves from the injection region to the production region, and must navigate around regions of low permeability.  As the CCG method does not use complete polynomial spaces, the fluid concentration is slightly more diffusive than that of the IPDG method with full polynomial spaces.    
 \subsection{3D numerical experiments}   
	Several 3D experiments are conducted in this section, including comparisons to $P_1$-DG. We revisit the permeability lens problem from Section~\ref{sec:numerics_lens}.  The domain is now $\Omega=[0,1]^3$, and the computational domain is a tetrahedral meshing of $\Omega$.   For injection and production, we put $q^I$ piecewise constant on $[0,0.1]^3$ and zero elsewhere, and $q^P$ piecewise constant on $[0.9,1]^3$ and zero elsewhere. The injection and production rates satisfy
 	\[
 	\int_\Omega q^I = 	\int_\Omega q^P = 0.018~ \frac{m^3}{s}.
 	\] 
\subsection{3D Homogeneous porous medium}  	
We consider the 3D analogue of Section~\ref{sec:numerics_homo}, adopting the same parameter values.  Snapshots of the concentration are provided in Fig.~\ref{fig:lens_3D_0}.  In a homogeneous porous medium, transport is not hindered by heterogeneity and the concentration front advances smoothly. The CCG numerical approximation reproduces this behavior.
	\begin{figure}[htbp!]
	\centering  
    \begin{subfigure}[t]{0.30\textwidth}
        \centering 
        	 \caption{$t=2.5$ seconds}
        	 \vspace*{-0.5ex}
        \includegraphics[width=\textwidth , trim=5cm 0cm 5.8cm 2cm,clip]{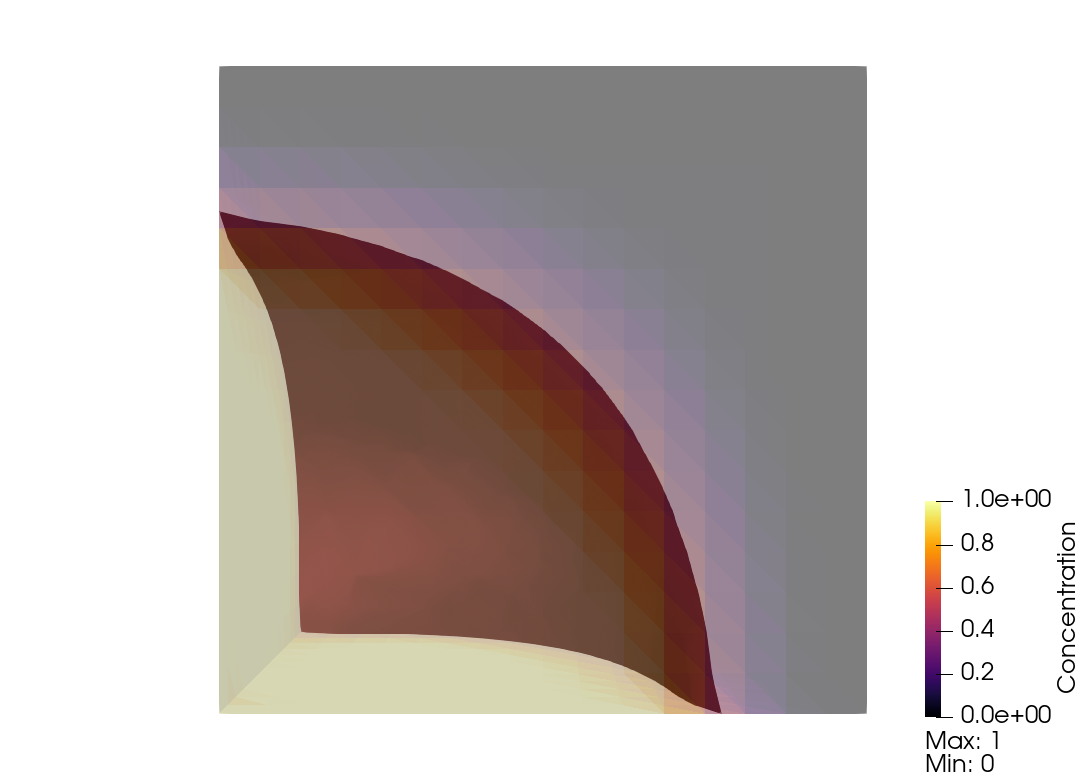}
    \end{subfigure}
    \begin{subfigure}[t]{0.29\textwidth}
        \centering 
        \caption{$t=5.0$ seconds}
        \includegraphics[width=\textwidth , trim=5cm 0cm 5.8cm 2cm,clip]{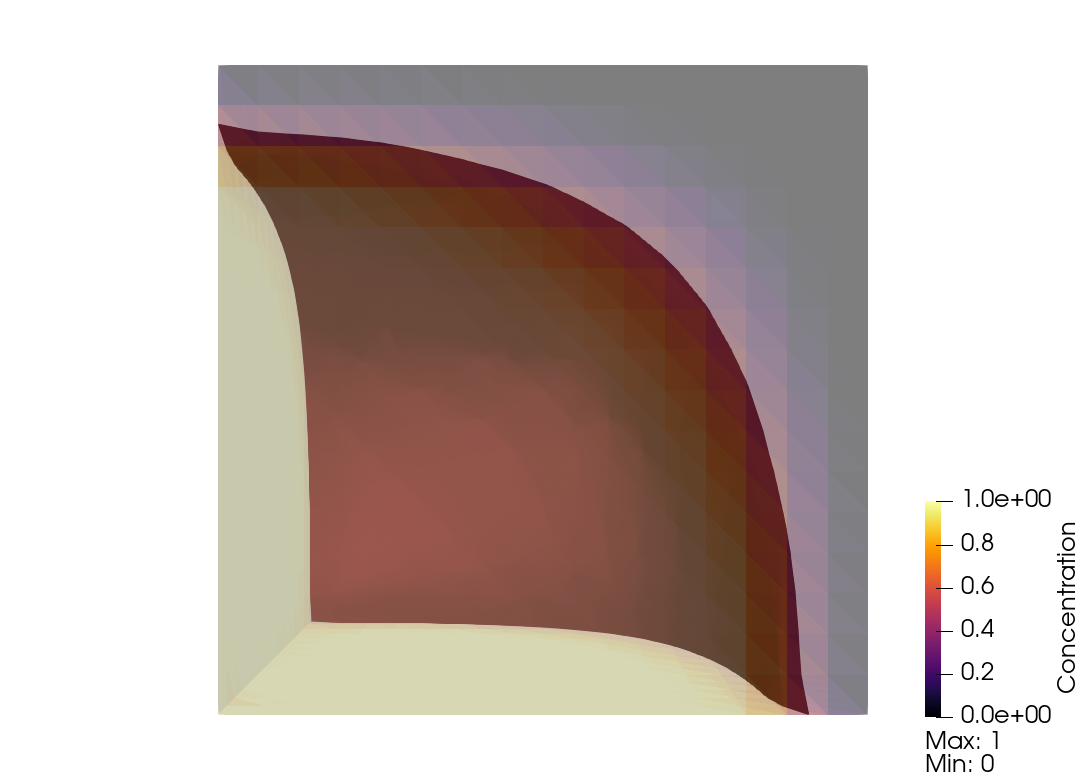}
    \end{subfigure}
    \begin{subfigure}[t]{0.35\textwidth} 
        \centering
        \caption{$t=7.5$ seconds}
        \includegraphics[width=\textwidth , trim=5cm 0cm 0cm 2cm,clip]{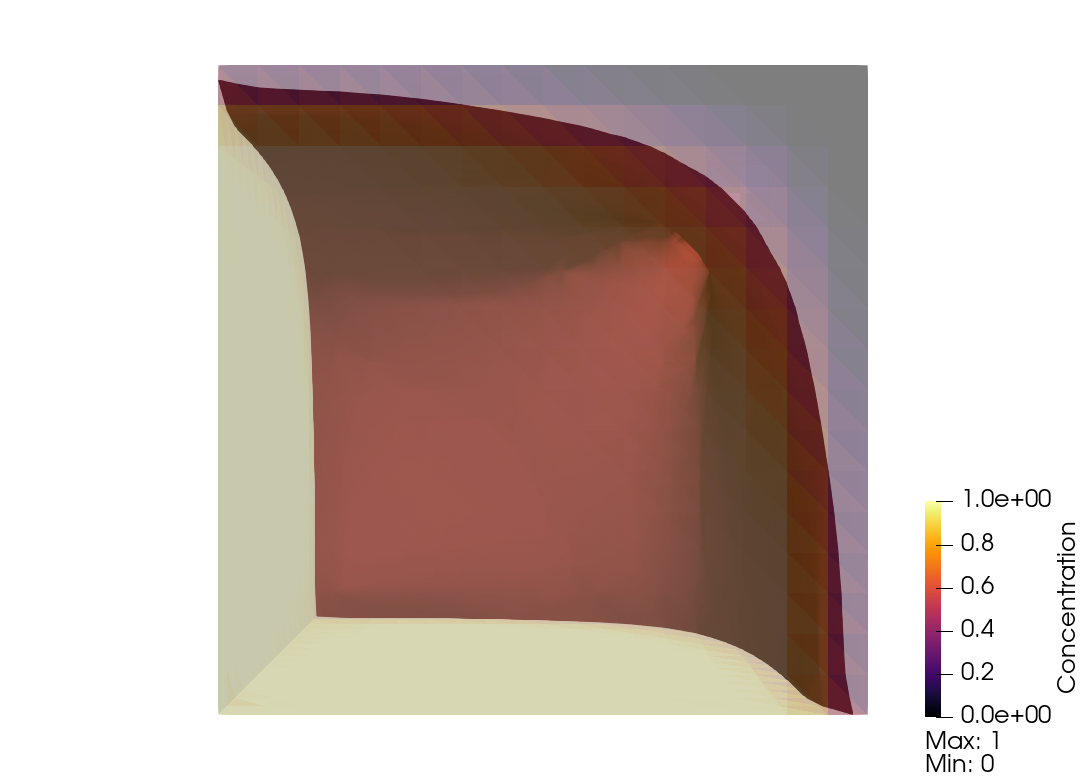}
    \end{subfigure} 
	\caption{Contour plot of the CCG solution $c_h(\vec{x},t)=0.5$ at various times (homogeneous permeability).}
	\label{fig:lens_3D_0}
	\end{figure}

\subsection{3D Heterogeneous porous medium (permeability lens)} \label{sec:3D_lens}	 	
	  On an element $  E \in \mathcal{T}_h$, the permeability is set as
\[
{\bf K}_E = \kappa_E {\bf I}_{3\times 3},
\quad
\kappa_E = 
\begin{cases}
9.44 \cdot 10^{-6}, &\mbox{ if }     (x,y,z)\in [0.25,0.5]\times[0.25,0.5]\times[0,1],
\\
9.44 \cdot 10^{-3}, &\mbox{ otherwise},
\end{cases} .
\]	  
All other parameters are the same as in section~\ref{sec:numerics_lens}. A mesh with 196608 tetrahedral elements is used. Fig.~\ref{fig:lens_3D_1} visualizes the CCG concentration contour ($c_h\equiv 0.5$) at various times.
	\begin{figure}[htbp!]
	\centering
    \begin{subfigure}[t]{0.30\textwidth}
        \centering 
        	 \caption{$t=2.5$ seconds}
        	 \vspace*{-0.5ex}
        \includegraphics[width=\textwidth , trim=5cm 0cm 5.8cm 2cm,clip]{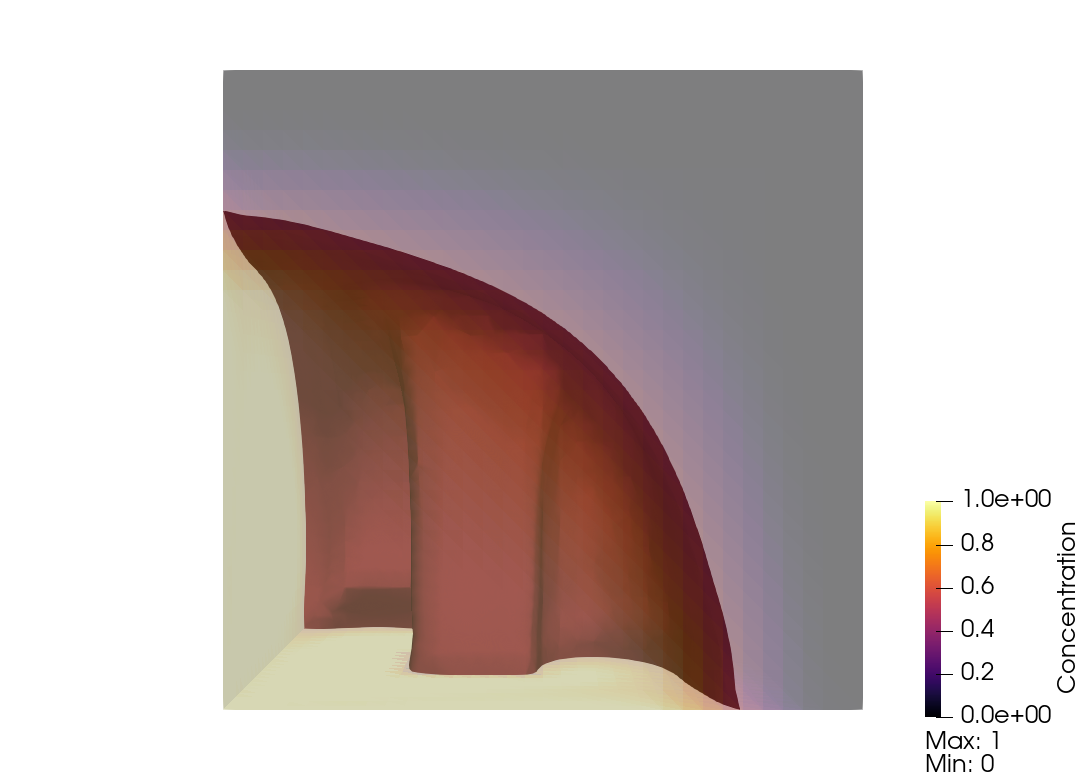}
    \end{subfigure}
    \begin{subfigure}[t]{0.29\textwidth}
        \centering 
        \caption{$t=5.0$ seconds}
        \includegraphics[width=\textwidth , trim=5cm 0cm 5.8cm 2cm,clip]{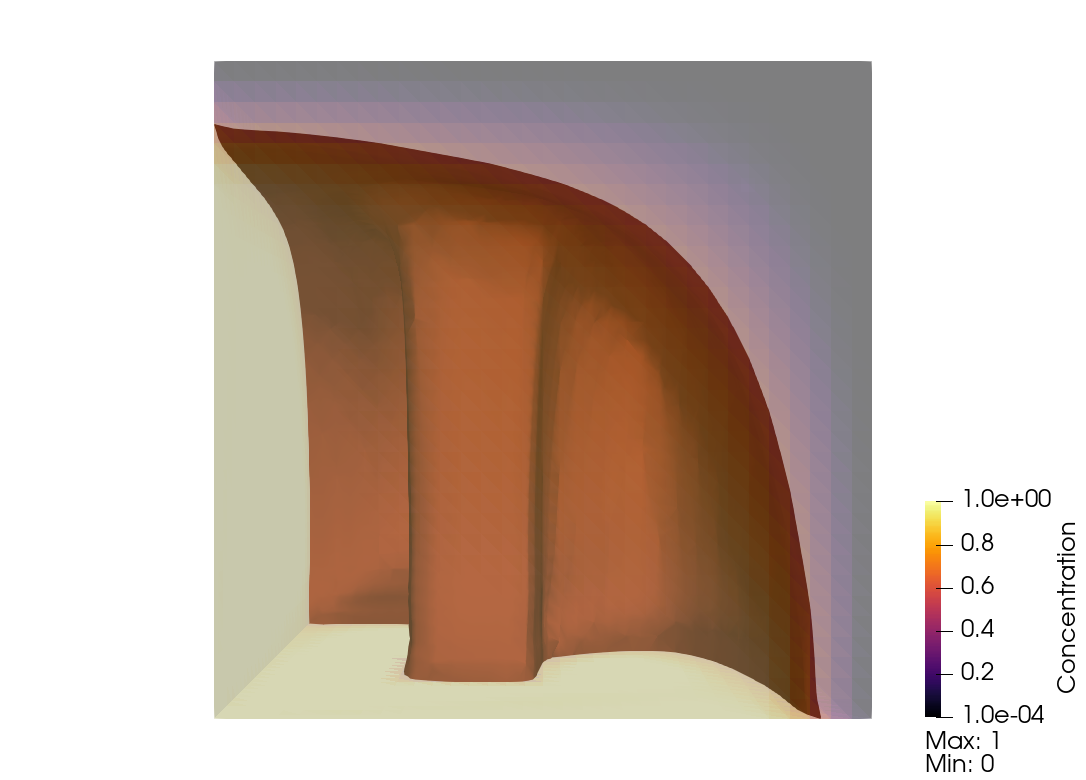}
    \end{subfigure}
    \begin{subfigure}[t]{0.35\textwidth} 
        \centering
        \caption{$t=7.5$ seconds}
        \includegraphics[width=\textwidth , trim=5cm 0cm 0cm 2cm,clip]{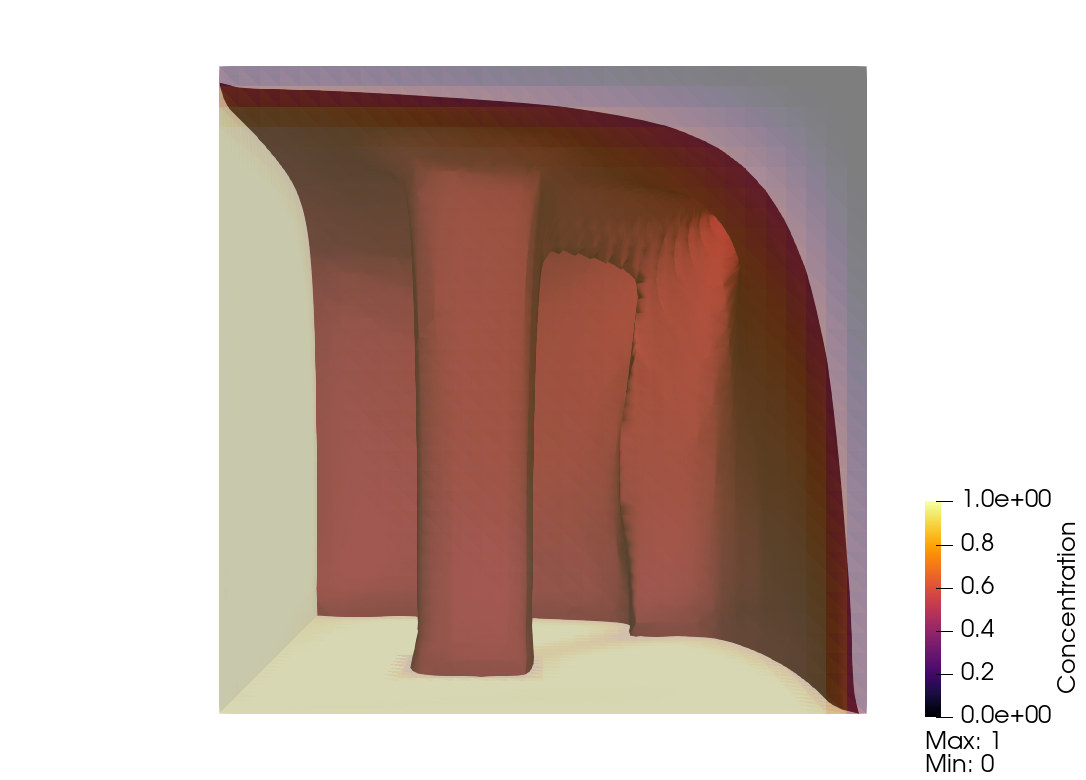}
    \end{subfigure} 
	\caption{Contour plot of the CCG solution $c_h(\vec{x},t)=0.5$ at various times (permeability lens).}
	\label{fig:lens_3D_1}
	\end{figure}
 The CCG method yields the expected results in 3D, with heterogeneous permeability: the concentration travels from the injection site, navigating around the region of lower permeability.
\subsection{Computational timings}  
	Section \ref{sec:computational_cost1} compared the nonzero counts of the CCG and $P_1$-DG discretization matrices. Here we assess a complementary metric: time-to-solution. We report wall-clock times for the CCG and $P_1$-DG methods.  We focus on the experiment from Section~\ref{sec:3D_lens}.  For simplicity, we pick a direct solver (MUMPS \cite{amestoy2001fully}).  Even in the absence of good scaling, the CCG method delivers lower runtimes with stock solvers.  Runtimes are evaluated across successive meshes; results are reported as average time per time step.  Timings were collected on an Intel Core i9‑9900K (3.60 GHz) system with 128 GB DDR4‑2400; runs were single‑threaded.
 
	Table~\ref{tab:timings-xy-speedup} contains the results for the direct solver. Despite the CCG matrix having about 25\% fewer nonzeros in the discretization matrix compared to $P_1$-DG (see Fig.~\ref{fig:ccg_vs_dg_cost_percent}), the CCG method achieves a 1.49$\times$ speedup. Put in other words, the CCG method has a 33\% lower runtime.  
\begin{table}[htbp!]
\centering
\caption{Wall-clock time (seconds) for CCG and $P_1$-DG methods, with speedup of CCG over $P_1$-DG and runtime reduction (in favor of CCG).}
\label{tab:timings-xy-speedup}
\begin{tabular}{l
				l
				l
				c
                S[table-format=2.1]}
\toprule
{$|\mathcal{T}_h|$} & {$P_1$-DG} & {CCG} & {Speedup} & {Runtime reduction [\%]} \\
\midrule
196608 & 4.249$\cdot 10^{+2}$   & 2.846$\cdot 10^{+2}$ & 1.490 & 33.0 \\ 
24576  & 3.977$\cdot 10^{+0}$    & 2.743$\cdot 10^{+0}$      & 1.450 & 31.0 \\
3072   & 9.990$\cdot 10^{-2}$ & 5.980$\cdot 10^{-2}$     & 1.657 & 39.6 \\
\bottomrule
\end{tabular}
\end{table}	

\section{Conclusion}\label{sec:conclusions}
	This paper presents the cell centered Galerkin method for miscible displacement problems in porous media.  The CCG method leverages incomplete polynomial spaces, trace and gradient reconstructions to devise an efficient numerical scheme.  Benefits of the CCG scheme include: one unknown per cell, it can use weak formulations from preexisting DG methods, local mass conservation, and delivers competitive accuracy when compared to $P_1$-DG approximations. Extensive numerical experiments substantiate the effectiveness of the proposed method.  In addition, comparisons between the CCG and $P_k$-DG approximations are provided.  It is found that the CCG method has considerable computational savings (about 25\% less expensive), and, the $P_k$-DG solutions require considerable post-processings.
 
	There are several future directions to explore. Comparing alternative trace and gradient reconstruction operators is needed.  The CCG method with barycentric trace interpolation gives rise to a discretization matrix with a wide stencil.  Also, theoretically speaking, a local loss of accuracy is generally possible because the barycentric trace interpolation does not take the heterogeneous permeability coefficients into account during reconstruction.
 
   Another promising direction is $hp$-adaptive CCG-DG strategies in which some cells use the CCG reconstruction, and other cells utilize the full polynomial DG spaces. The advantage over multi-numerics approaches is that the approximation falls under the same theoretical finite element variational framework. For instance, with CCFV-DG multi-numeric schemes, the interface between methods can be delicate to define, and consistency errors are introduced.  Moreover, many scientific computing softwares specialize in one type of numerical method, which can cause difficulty for multi-numeric schemes.
   
   Finally, we also plan to explore maximum-principle properties for the CCG method in higher spatial dimensions. High-order numerical methods are notorious for violating maximum-principle properties, which are intrinsic to many PDEs.  However, numerical evidence suggest that the CCG method for the Poisson problem obeys a discrete maximum-principle property.

\section*{Acknowledgements} \label{}
The author acknowledges support from the MIT Schwarzman College of Computing and the MIT MLK Professorship.

 \section{Appendix}\label{sec:app}
 We consider the CCG discretization for the 1D model problem:
 \[
 -u''(x) = f(x),\quad u(a)=u(b)=0.
 \]
The domain $\Omega=[a,b]$ is partitioned into $N$ elements $E_\ell=( x_{\ell}, x_{\ell+1})$, for $\ell\in\{0,1,2,\ldots,N-1\}$ with $x_0=a<x_1<\ldots <x_{N-1}<x_N=b$. The cell center is given by $x_{E_\ell} = (x_{\ell}+ x_{\ell+1})/2$. To make the exposition more clear, we assume the mesh is uniform with $h= |x_{\ell} - x_{\ell+1}|$.

The standard IPDG bilinear form is used: we seek $u_h\in V_h^{\textrm{ccg}}$ so that
\begin{equation}  
a_\epsilon(u_h,v_h) = \sum_{\ell=0}^{N-1} \int_{E_\ell} u_h' v_h' + \sum_{\ell=0}^{N } (-\{u_h'\}_{x_{\ell}}[v_h]_{x_{\ell}} +\epsilon  \{v_h'\}_{x_{\ell}}[u_h]_{x_{\ell}} + \frac{\sigma}{h}[u_h]_{x_{\ell}}[v_h]_{x_{\ell}}),
\label{eq:app1}
\end{equation}
for all $v_h \in V_h^{\textrm{ccg}}$ (see equation \eqref{def_ccg_space}). The space $ V_h^{\textrm{ccg}} $ is explored in more detail next.

\subsection{Barycentric trace reconstruction}
We examine the barycentric trace reconstruction (see Definition \ref{eq_bary}).	For an interior face \(F = x_k,\, 1 \leq k \leq N-1\), the two neighboring cell centers are \(m_{k-1},\, m_k\). The barycentric coordinates at \(x_k\) in the interval \([m_{k-1}, m_k]\) are:
\[
\lambda_{k-1,F}(x_k) = \frac{m_k - x_k}{m_k - m_{k-1}},
\qquad
\lambda_{k,F}(x_k) = \frac{x_k - m_{k-1}}{m_k - m_{k-1}}.
\]
With
\[
m_{k-1} = x_{k-1} + \frac{h}{2}, \quad m_k = x_k + \frac{h}{2}, \quad x_k = x_{k-1} + h,
\]
it follows that
\[
x_k - m_{k-1} = h/2, \qquad
m_k - x_k = h/2, \qquad
m_k - m_{k-1} = h.
\]
Therefore,
\[
\lambda_{k-1,F}(x_k) = \frac{h/2}{h} = \frac{1}{2}, \qquad \lambda_{k,F}(x_k) = \frac{h/2}{h} = \frac{1}{2}.
\]
For a face \(F = x_k\), the barycentric trace reconstruction $I_F$ is then given by
\[
I_{x_k}(v_h) =
\begin{cases}
\lambda_{k-1,F}\, v_{k-1} + \lambda_{k,F}\, v_k = \frac{1}{2}(v_{k-1} + v_k), & 1 \leq k \leq N-1 \\[2ex]
0, & k=0\ \text{or}\ k=N\quad (\text{boundary face}).
\end{cases}
\]
\subsection*{Discrete gradient operator}
	With access to the barycentric trace reconstruction $I_F$, we can construct the discrete gradient operator (see Definition \ref{sec:prelim_def2}) for \(v_h = \phi_j\) evaluated in cell \(T_\ell\) is
\[
\begin{aligned}
G_h^{(0)}(\phi_j)\big|_{T_\ell}
&= \frac{1}{h} \left( I_{x_{\ell+1}}(\phi_j) - I_{x_\ell}(\phi_j) \right) \\[2ex]
&= \frac{1}{h} \left[
    \frac{1}{2}(\delta_{j,\ell} + \delta_{j,\ell+1})
    - \frac{1}{2}(\delta_{j,\ell-1} + \delta_{j,\ell})
  \right] \\[2ex]
&= \frac{1}{2h} \left( \delta_{j,\ell+1} - \delta_{j,\ell-1} \right),
\end{aligned}
\]
where \(\delta_{a,b}\) is the Kronecker delta. Explicitly,
\[
G_h^{(0)}(\phi_j)\big|_{T_\ell} =
\begin{cases}
~+\dfrac{1}{2h}, & \ell = j-1 \\[2ex]
~-\dfrac{1}{2h}, & \ell = j+1 \\[2ex]
~0, & \text{otherwise}.
\end{cases}
\]
The reconstructed affine function in each cell is:
\[
\mathcal{A}_h^{(1)}(\phi_j)\big|_{T_\ell}(x) = v_{T_\ell} + G_h^{(0)}(\phi_j)\big|_{T_\ell}\cdot (x - x_{T_\ell}).
\]
The mapped functions $\mathcal{A}_h^{(1)}(\phi_j)$ on cell $T_j$ are summarized below.
\subsection*{For $j=0$ (leftmost basis function):}
\[
\mathcal{A}_h^{(1)}(\phi_0)(x) =
\begin{cases}
1 + \dfrac{1}{2 h}(x - x_{T_0}), & x \in T_0, \\[2ex]
-\dfrac{1}{2 h}(x - x_{T_1}), & x \in T_1, \\[2ex]
0, & \text{otherwise}.
\end{cases}
\]

\subsection*{For $j=N-1$ (rightmost basis function):}
\[
\mathcal{A}_h^{(1)}(\phi_{N-1})(x) =
\begin{cases}
\dfrac{1}{2 h}(x - x_{T_{N-2}}), & x \in T_{N-2}, \\[2ex]
1 - \dfrac{1}{2 h}(x - x_{T_{N-1}}), & x \in T_{N-1}, \\[2ex]
0, & \text{otherwise}.
\end{cases}
\]

\subsection*{For $1 \leq j \leq N-2$ (interior case):}
\[
\mathcal{A}_h^{(1)}(\phi_j)(x) =
\begin{cases}
\dfrac{1}{2 h}(x - x_{T_{j-1}}), & x \in T_{j-1}, \\[2ex]
1, & x \in T_j \\[2ex]
-\dfrac{1}{2 h}(x - x_{T_{j+1}}), & x \in T_{j+1},\\[2ex]
0, & \text{otherwise}.
\end{cases}
\] 

Using the mapped functions $\mathcal{A}_h^{(1)}(\phi_j)$, and the IPDG bilinear form \eqref{eq:app1}, the CCG discretization with barycentric gradient reconstruction gives rise to a 7-diagonal matrix: 
\begin{equation}
T_{\epsilon,\sigma} := \frac{1}{h}
\left[
\begin{array}{cccccccccc}
a_0   & b_0   & c_0   & d     &        &        &        &        &        &        \\
b_0   & a_1   & b     & c     & d      &        &        &        &        &        \\
c_0   & b     & a     & b     & c      & d      &        &        &        &        \\
d     & c     & b     & a     & b      & c      & d      &        &        &        \\
      & d     & c     & b     & a      & b      & c      & d      &        &        \\
      &       & \ddots& \ddots& \ddots & \ddots & \ddots & \ddots & \ddots &        \\
      &       &       & d     & c      & b      & a      & b      & c      & d      \\
      &       &       &       & d      & c      & b      & a      & b      & c_0    \\
      &       &       &       &        & d      & c      & b      & a_1    & b_0    \\
      &       &       &       &        &        & d      & c_0    & b_0    & a_0    \\
\end{array}
\right],
\label{eq:stiffness1D}
\end{equation}
where
\begin{equation}
\begin{aligned}
a   &= \dfrac{5\sigma}{4} - \dfrac{\epsilon}{4} + \dfrac{3}{4}, &
b   &= \dfrac{\epsilon}{16} - \dfrac{15\sigma}{16} - \dfrac{1}{16}, &
c   &= \dfrac{\epsilon}{8} + \dfrac{3\sigma}{8} - \dfrac{3}{8}, &
d   &= \dfrac{1}{16} - \dfrac{\sigma}{16} - \dfrac{\epsilon}{16}, \\[4pt]
a_1 &= \dfrac{5\sigma}{4} - \dfrac{3\epsilon}{16} + \dfrac{11}{16}, & 
a_0 &= \dfrac{13\sigma}{8} - \dfrac{5\epsilon}{16} + \dfrac{13}{16}, &
b_0 &= \dfrac{5}{16} - \dfrac{9\sigma}{8} - \dfrac{\epsilon}{16}, &
c_0 &= \dfrac{3\epsilon}{16} + \dfrac{7\sigma}{16} - \dfrac{7}{16}.
\end{aligned}
\label{eq:coefficients}
\end{equation}
Numerical experiments suggest that the CCG matrix is monotone (in 1D, 2D, and 3D) for a wide range of $\epsilon$ and $\sigma$, but we prove some special cases. We first find conditions on which $T_{\epsilon,\sigma}$ is a $M$-matrix. Then, we prove that under certain conditions, the CCG matrix is not an $M$-matrix, but has a nonnegative inverse.

\begin{theorem}
For the matrix $T_{\epsilon,\sigma}$ we have the following properties:
\begin{enumerate}
\item It is symmetric for any $\epsilon$ and $\sigma$,
\item The matrix is pentadiagonal for $\sigma>0$ and $\epsilon=1-\sigma$, with the special case of $\epsilon=0$ and $\sigma=1$ resulting in a tridiagonal matrix,
\item The matrix has nonpositive off-diagonal entries for $\epsilon=0$ and $\sigma=1$; as well as $\epsilon=1$ and $2/9\le \sigma \le 4/7 $.
\end{enumerate}
\end{theorem}
\begin{proof}
Without a loss of generality, let $h=1$.  Property 1 is obvious by inspection. 

For property 2, if $\epsilon=1-\sigma$, we have $d = \dfrac{1-\sigma-\epsilon}{16} = 0,$ which means $T_{\epsilon,\sigma}$ is pentadiagonal in that case.  If we additionally set $\sigma=1$, then
\begin{align*}
c_0 = \dfrac{3\epsilon}{16} + \dfrac{7\sigma}{16} - \dfrac{7}{16}=0,\quad  c  = \dfrac{\epsilon}{8} + \dfrac{3\sigma}{8} - \dfrac{3}{8}=0,
\end{align*}
so that $T_{\epsilon,\sigma}$ is tridiagonal.
 
Property 3 is equivalent to verifying $T_{\epsilon,\sigma}$ is a $Z$-matrix under certain conditions.  If $\epsilon =0$ and $\sigma=1$, we know that $T_{\epsilon,\sigma}$ is tridiagonal.  In this case, the off-diagonal terms are
\begin{align*} 
b   = \dfrac{\epsilon}{16} - \dfrac{15\sigma}{16} - \dfrac{1}{16}=-\dfrac{15 }{16}, 
\quad
b_0 = \dfrac{5}{16} - \dfrac{9\sigma}{8} - \dfrac{\epsilon}{16}  = -\frac{13}{16},
\end{align*}
so that all off-diagonal terms are nonpositive. 
 
For $\epsilon=1$, we will determine which $\sigma>0$ ensure $T_{\epsilon,\sigma}$ has nonpositive off-diagonal terms: $b_0\le 0$, $c_0\le 0$, $b \le 0$, $c\le 0$, $d\le 0$.  Since $\epsilon=1$, upon inspecting \eqref{eq:coefficients} we find 
 \begin{align*}
  b    =& -\frac{15\sigma}{16} \leq 0 \;\implies\; \sigma \geq 0, ~\quad\quad\quad 
  c    =    \frac{3\sigma}{8} - \frac{1}{4} \leq 0 \;\implies\; \sigma \leq \frac{2}{3}, \quad
  d    =    -\frac{\sigma}{16} \leq 0 \;\implies\; \sigma \geq 0 ,\quad\\[1ex]
  b_0  &=  \frac{1}{4} - \frac{9\sigma}{8} \leq 0 \;\implies\; \sigma \geq \frac{2}{9}, \quad
  c_0 =  \frac{7\sigma}{16} - \frac{1}{4} \leq 0 \;\implies\; \sigma \leq \frac{4}{7}.
\end{align*}
The intersection of these inequalities yields $2/9 \le \sigma\le 4/7$. Hence, $T_{\epsilon,\sigma}$ is a $Z$-matrix for $\epsilon=1$ and $2/9 \le \sigma\le 4/7$. 
\end{proof}

\begin{theorem}\label{thm:NIPG}
If $\epsilon=1$ and $2/9\le \sigma \le 4/7 $, then $T_{\epsilon,\sigma}$ is irreducible, diagonally dominant, nonsingular, and an $M$-matrix.
\end{theorem}
\begin{proof}
Let $h=1$ without a loss of generality.  The matrix $T_{\epsilon,\sigma}$ is banded and none of the nonzero bands contain a zero value.  Therefore it is not possible to find a permutation matrix $P$ so that $PT_{\epsilon,\sigma}P^T$ is block upper triangular. Hence, $T_{\epsilon,\sigma}$ is irreducible.

Next we examine diagonal dominance.   Adding the off-diagonal terms of row $i \in\{1,N\}$ gives $\sum_{i\neq j}|(T_{\epsilon,\sigma})_{ij}| = |b_0|+|c_0|+|d| = |\frac{4}{16} - \frac{9\sigma}{8}  |+|  \frac{7\sigma}{16} - \frac{4}{16}|+ \frac{\sigma}{16}$. For $2/9\le \sigma \le 4/7 $, we have $b_0\le 0$, $c_0\le 0$, $b \le 0$, $c\le 0$, $d\le 0$. This allows us to write
$\sum_{i\neq j}|(T_{\epsilon,\sigma})_{ij}| = 
-(\frac{4}{16} - \frac{9\sigma}{8} ) -(  \frac{7\sigma}{16} - \frac{4}{16})+ \frac{\sigma}{16} = \frac{3\sigma}{4}.
$
It is then easy to verify that $(T_{\epsilon,\sigma})_{ii}= a_0 = \frac{8}{16}+\frac{13\sigma}{8}> \sum_{i\neq j}|(T_{\epsilon,\sigma})_{ij}|$.  Rows 1 and $N$ have strict row diagonal dominance.

For row $i \in\{2,N-1\}$, we set $\sum_{i\neq j}|(T_{\epsilon,\sigma})_{ij}|= |b_0|+|b|+|c|+|d| = |\frac{1}{4} - \frac{9\sigma}{8}| +|\frac{15\sigma}{16}| + |-\frac{1}{4}+\frac{3\sigma}{8} | + |-\frac{\sigma}{16}| = \frac{7\sigma}{4}$. Since $2/9 \le \sigma\le 4/7$, we have $(T_{\epsilon,\sigma})_{ii}=a_1=\frac{5\sigma}{4}+\frac{8}{3}>\sum_{i\neq j}|(T_{\epsilon,\sigma})_{ij}|=\frac{7\sigma}{4}$. We have shown that rows 2 and $N-1$ have strict row diagonal dominance. 

Summing the off-diagonal terms for row $i\in\{3,N-2\}$ gives 
\[\sum_{i\neq j}|(T_{\epsilon,\sigma})_{ij}|=|c_0|+2|b|+|c|+|d| =
\frac{4-7\sigma}{16} + 2\frac{15\sigma}{16} +\frac{2-3\sigma}{8} + \frac{\sigma}{16}
=
\frac{ 18\sigma+8}{16}.\]
 We see that $a=(T_{\epsilon,\sigma})_{33} = \frac{20\sigma+8}{16}>\sum_{i\neq j}|(T_{\epsilon,\sigma})_{ij}|=\frac{8+18\sigma}{16}.$ Thus, rows 3 and and $N-2$ have strict row diagonal dominance.

Finally, we inspect the interior rows $i\in\{ 4,5,\ldots,N-3\}$.  In each interior row have the same nonzero main diagonal and off-diagonal values. Let $4\le i\le N-3.$ For row $i$, the off-diagonal sum is $\sum_{i\neq j}(T_{\epsilon,\sigma})_{ij} =2|b|+2|c|+2|d|=2\frac{15\sigma}{16}+2\frac{2-3\sigma}{8}+2\frac{\sigma}{16}=\frac{20\sigma}{8}$.  Hence, $a=(T_{\epsilon,\sigma})_{ii} = \sum_{i\neq j}(T_{\epsilon,\sigma})_{ij},$ and we have row diagonal dominance.

For $\epsilon=1$ and $2/9\le \sigma \le 4/7 $, we have shown that $T_{\epsilon,\sigma}$ is irreducible and diagonally dominant. Moreover, at least one row has strict row diagonal dominance.  This asserts that $T_{\epsilon,\sigma}$ is nonsingular and positive definite \cite{horn2012matrix}. Since $T_{\epsilon,\sigma}$ is symmetric positive definite, and a $Z$-matrix, it is also an nonsingular $M$-matrix, which means that it has a nonnegative inverse. 
\end{proof}

\begin{theorem}
	If $\epsilon=0$, and $\sigma=1$, then $T_{\epsilon,\sigma}$ is an $M$-matrix.
\end{theorem}
\begin{proof} 
Following Theorem~\ref{thm:NIPG}, it is straightforward to prove that $T_{\epsilon,\sigma}$ is a $Z$-matrix with diagonally dominance.
\end{proof}

 When the CCG matrix $T_{\epsilon,\sigma}$ is not an $M$-matrix, the theory developed in \cite{lorenz1979toeplitz}, \cite{meek1976new}, and \cite{olesky1982monotone} can be leveraged to show that it is still monotone for some restrictions on $\epsilon$ and $\sigma$.
 
We define the following $N\times N$ Toeplitz matrices:
\begin{equation} 
T_M = \begin{bmatrix}
A   & B   & C   & d   & 0   & \cdots & 0 \\
B   & A   & B   & C   & d   & \ddots & \vdots \\
C   & B   & A   & B   & C   & \ddots & 0 \\
d   & C   & B   & A   & B   & \ddots & d \\
0   & d   & C   & B   & A   & \ddots & C \\
\vdots & \ddots & \ddots & \ddots & \ddots & \ddots & B \\
0   & \cdots & 0 & d & C & B & A
\end{bmatrix}
,
\quad
T_m =
\begin{bmatrix}
A & B & & & \\
B & A & B & & \\
& B & A & B & \\
& & \ddots & \ddots & \ddots \\
& & & B & A
\end{bmatrix}
,
\label{eq:range_TmTM}
\end{equation}
with $A=\max\{a_0,a_1,a\}$, $B=\max\{b_0,b\}$, $C=\max\{c_0,c\}$, where $a,b,c,d,a_1,a_0,b_0,c_0$ are defined in~\eqref{eq:coefficients}.

\begin{theorem}[Berman and Plemmons, {\cite{refmono1}}] \label{thm:Plemmons}
If $T_M$ and $T_m$ are monotone matrices such that $T_m\le   T_{\epsilon,\sigma} \le  T_M,$ then $T_{\epsilon,\sigma}$ is monotone.
\end{theorem}

\begin{theorem}[Lorenz and Mackens, \protect\cite{lorenz1979toeplitz}] \label{thm:Mackens}
If all zeros of the polynomial $p(z):=d z^6 + C z^5 + B z^4 + A z^3 + B z^2 + Cz + d$ are real and positive, and if $z=1$ is not a zero of $p(z)$, then $T_M$ is a monotone matrix with strictly positive eigenvalues.
\end{theorem} 

\begin{lemma}\label{lmm:range}
If $\sigma + \epsilon = 1$, and $1 < \sigma\le \frac{104\sqrt{2}}{79} - \frac{12}{79}$, then for $T_m$ and $T_M$ defined in \eqref{eq:range_TmTM} we have
\begin{enumerate}
\item $T_m\le   T_{\epsilon,\sigma}   \le T_M$,
\item and $T_m$ is monotone.
\end{enumerate}
\end{lemma}
\begin{proof}
From \eqref{eq:coefficients}, we can easily deduce that $c_0=c=\frac{\sigma-1}{4}$ because $\epsilon=1 - \sigma$. Furthermore, $C=\max\{c_0,c\}=\frac{\sigma-1}{4}$, and since $1 < \sigma\le  (104\sqrt{2} -  12)/79$, we have $c>0$.  In addition, we find $d =(1-\sigma-\epsilon)/16 = 0.$  A simple entry-wise analysis allows us to arrive at $T_m\le   T_{\epsilon,\sigma} \le  T_M.$

To show monotnicity of $T_m$, we demonstrate that it is an $M-$matrix.  The off-diagonal terms of $T_m$ are $B=\max\{b_0,b\}=b_0=(4-17\sigma)/16.$  The term $B$ is nonpositive whenever $1 <  \sigma\le (104\sqrt{2} -  12)/79$.  Thus, $T_m$ is a $Z$-matrix.  

Next, we investigate the diagonal dominance properties of $T_m$.  For rows 1 and $N$, $(T_m)_{ii} = A = (31\sigma+8)/16$ and $\sum_{i\neq j} | (T_m)_{ij} |=|B|=|(4-17\sigma)/16|,$ so strict diagonal row dominance holds. Rows $i\in\{2,3,\ldots,N-1\}$ have $(T_m)_{ii} = A = (31\sigma+8)/16$ and $\sum_{i\neq j} | (T_m)_{ij} |=2|B|=2|(4-17\sigma)/16|,$ and strict diagonal row dominance holds for $1 < \sigma\le (104\sqrt{2} -  12)/79$.
	
The matrix $T_m$ is a $Z$-matrix, has all positive diagonal elements, and is strictly diagonally dominant, therefore it is an $M$-matrix.
\end{proof}

\begin{theorem}
	If $\sigma + \epsilon = 1$, and $1 < \sigma\le \frac{104\sqrt{2}}{79} - \frac{12}{79}$, then $T_M$ defined in \eqref{eq:range_TmTM} and $T_{\epsilon,\sigma}$ from \eqref{eq:stiffness1D} are monotone matrices.
\end{theorem}
\begin{proof} 
Since $\sigma+\epsilon=1$, then $d=0$, so that $ p(z):= C z^4 + B z^3 + A z^2 + B z + C $ is the polynomial from Theorem~\ref{thm:Mackens}.  Let $w=z + \frac{1}{z}$.  Then,
\begin{align*}
\frac{p(z)}{z^2} &= C(z^2+\frac{1}{z^2}) + B ( z + \frac{1}{z} ) + A
= Cw^2 + Bw +  (A-2C).
\end{align*}
Let $Q(w):=Cw^2 + Bw +  (A-2C)$.  Each root $w$ of $Q$ also satisfies $z+\frac{1}{z}=w$, which means we have $z^2-wz+1=0,$ and the quadratic formula gives 
\begin{equation}
z  = \frac{w\pm \sqrt{w^2-4}}{2} 
 \label{thm:range_0}
.
\end{equation}
   Hence, all of the roots of $p$ are positive if $w\ge 2$.  According to Theorem~\ref{thm:Mackens}, we are not interested in the case of $z=1$, so we focus on $w>2$.

Next we determine what conditions on $A,B,$ and $C$ ensure that the roots $w$ of $Q$ are real and obey $w>2$.  Let $\tilde{w}=w-2$, and consider the shifted polynomial $\tilde{Q}( \tilde{w} ) = Q(\tilde{w}+2)$. Then, $Q$ having two real roots with $w> 2$ is equivalent to $\tilde{Q}( \tilde{w} )$ having two positive roots. A straightforward computation reveals that $\tilde{Q}( \tilde{w} ) = C \tilde{w}^2 + (4C+B)\tilde{w} + (2C+2B+A)$.

To ensure that the roots of $\tilde{Q}$ are real, we set its discriminant to be nonnegative: $\Delta := (4C+B)^2 - 4C(2C+2B+A)\ge0$. If we additionally request that $C>0$, $4C+B<0$, and $2C+2B+A>0$, then Vieta's formulas assert that the roots $w_1,w_2$ of $\tilde{Q}$ satisfy $w_1+w_2 = -\frac{4C+B}{C}>0$ and $w_1w_2= \frac{2C+2B+A}{C}>0$.  If the product of $w_1$ and $w_2$ is positive, they must have the same sign. If the sum of $w_1$ and $w_2$ is also positive, this means that $w_1>0$ and $w_2>0$.

We see that the conditions 
\begin{align}
\Delta &\ge0, \label{thm:range_4}
\\
C&>0, \label{thm:range_1}
\\
4C+B&<0, \label{thm:range_2}
\\
2C+2B+A&>0, \label{thm:range_3}
\end{align}
ensure that the unshifted polynomial $Q$ would have real roots $w> 2$.

Since $\epsilon+\sigma=1$, from \eqref{eq:coefficients} we find $A=\max\{a_0,a_1,a\}=a_0= \frac{31\sigma+8}{16}$, $B=\max\{b_0,b\}=\frac{4-17\sigma}{16}$, $C=\max\{c_0,c\}=\frac{4(\sigma-1)}{16}$.   The condition \eqref{thm:range_4} puts a restriction on $\sigma$. In particular, 
\[
\Delta(\sigma) = \frac{17}{16} - \frac{79\sigma^2}{256} - \frac{3\sigma}{32} \ge0.
\]
The roots of $\Delta(\sigma)=0$ are $\sigma_1=-\frac{104\sqrt{2}}{79} - \frac{12}{79}$ and $\sigma_2=  \frac{104\sqrt{2}}{79} - \frac{12}{79}$. The discriminant $\Delta$ is quadratic in $\sigma$, so it is trivial to check that $\Delta\ge0$ whenever $0\le \sigma\le \sigma_2$.

Suppose $1 < \sigma\le \sigma_2$. Then, we have \eqref{thm:range_1} satisfied: $0<C =\frac{4(\sigma-1)}{16}$. For \eqref{thm:range_2}, some simplifications reveal
\[
4C+B = 4\frac{4(\sigma-1)}{16} + \frac{4-17\sigma}{16} = -\frac{\sigma+12}{16}<0.
\]
Additionally, the term $2C+2B+A$ is positive:
\begin{align*} 
2C+2B+A &= 2\cdot\frac{4(\sigma-1)}{16} + 2\cdot\frac{4-17\sigma}{16}+\frac{31\sigma+8}{16}=\frac{5\sigma+8}{16} > 0,
\end{align*}
so that \eqref{thm:range_3} is satisfied. 

Conditions \eqref{thm:range_4}, \eqref{thm:range_1}, \eqref{thm:range_2}, and  \eqref{thm:range_3} are valid for $1 < \sigma\le v_2$. It follows that $Q$ has two real roots, both of which are strictly greater than 2.  The roots $z $ of the polynomial $p$ are given by~\eqref{thm:range_0}, so we conclude that $z >0$. Consequently, the hypotheses of  Theorem~\ref{thm:Mackens} and hold, which means $T_M$ is monotone. In addition, from Lemma~\ref{lmm:range} and Theorem~\ref{thm:Plemmons} we deduce that $T_{\epsilon,\sigma}$ is monotone.
\end{proof}

 \bibliographystyle{elsarticle-num} 
\bibliography{library.bib}




\end{document}